\documentclass[11pt,oneside]{amsart}
\usepackage[T1]{fontenc}
\usepackage[english]{babel}
\usepackage{fouriernc} 
\usepackage{mathtools}
\usepackage{stmaryrd} 

\usepackage[margin=1in]{geometry} 
\usepackage{xcolor} 
\usepackage{tikz} 
\usetikzlibrary {arrows.meta,bending, positioning}
\usetikzlibrary{math} 
\usepackage{hyperref} 

\usepackage{amsmath,amsthm,amssymb}
\usepackage[all,cmtip]{xy}
\usepackage{caption}
\usepackage{subcaption}

\DeclareMathOperator{\Hom}{Hom}
\DeclareMathOperator{\ind}{ind}

\DeclareMathOperator{\res}{res}

\DeclareMathOperator{\Ho}{Ho}
\DeclareMathOperator{\THH}{THH}

\DeclareMathOperator{\HH}{HH}
\DeclareMathOperator{\Perf}{Perf}

\DeclareMathOperator{\End}{End}

\DeclareMathOperator{\tr}{tr}

\DeclareMathOperator{\TC}{TC}
\DeclareMathOperator{\id}{id}

\DeclareMathOperator{\Aut}{Aut}
\DeclareMathOperator{\coker}{coker}

\usepackage[capitalise]{cleveref}  
\newcommand{\clevertheorem}[3]{
  \newtheorem{#1}[thm]{#2}
  \crefname{#1}{#2}{#3}
}

\numberwithin{equation}{section} 
\numberwithin{figure}{section}

\theoremstyle{plain} 
\newtheorem{thm}{Theorem}[section]
\crefname{thm}{Theorem}{Theorems}
\newtheorem*{thm*}{Theorem}
\clevertheorem{prop}{Proposition}{Propositions}
\newtheorem*{prop*}{Proposition}
\clevertheorem{lemma}{Lemma}{Lemmas}
\clevertheorem{cor}{Corollary}{Corollaries}
\clevertheorem{conj}{Conjecture}{Conjectures}
\clevertheorem{claim}{Claim}{Claim}
\clevertheorem{observation}{Observation}{Observation}

\theoremstyle{definition}  
\clevertheorem{definition}{Definition}{Definitions}
\clevertheorem{notn}{Notation}{Notations}

\theoremstyle{remark}  
\clevertheorem{remark}{Remark}{Remarks}
\clevertheorem{example}{Example}{Examples}

\makeatletter\let\c@equation\c@thm\makeatother
 
\makeatletter\let\c@figure\c@thm\makeatother

\crefname{figure}{Figure}{Figures}
\crefname{equation}{Display}{Displays}  
\crefname{eq}{Display}{Displays}
\crefname{eqn}{Display}{Displays}

\hypersetup{
 pdfauthor={Zhonghui Sun},
}

\usepackage{nicematrix}
\usepackage{quiver}
\usepackage{graphicx,
tikz,tikz-cd} 
\usepackage{braket}

\newcommand{\shspace}{\hspace{.5mm}} 

\newcommand{\sh}[1]{{\ensuremath{\shspace\hspace{1mm}\makebox[-1mm]{$\langle$}\makebox[0mm]{$\langle$}\hspace{1mm}{#1}\makebox[1mm]{$\rangle$}\makebox[0mm]{$\rangle$}\shspace}}}

\begin{document}

\title{Twisted Bicategorical Shadows and Traces}
\author{Zhonghui Sun}
\address{Michigan State University}
\email{sunzhon1@msu.edu}

\begin{abstract}
Bicategorical shadows provide a categorical framework that encompasses Hochschild homology and topological Hochschild homology (THH), encodes their Morita invariance, and extends the notion of trace from symmetric monoidal categories to bicategories. Certain equivariant variants, such as \(C_n\)-twisted THH, do not fit into the ordinary shadow framework. We introduce bicategories with \(G\)-twisting data and show that every ordinary shadow on such a bicategory induces, for each \(g\in G\), a \(g\)-twisted shadow on the associated bicategory of \(G\)-twists. These \(g\)-twisted shadows are invariant under \(G\)-Morita equivalence; examples include \(C_n\)-twisted THH and twisted Hochschild homology of \(C_n\)-Green functors. We further define a \(g\)-twisted bicategorical trace that recovers the \(g\)-twisted Hattori--Stallings trace. For \(C_n\)-Green functors, the orbitwise \(g\)-twisted Hattori--Stallings traces assemble into a morphism of \(C_n\)-Mackey functors that, at the orbit \(C_n/C_n\), recovers the degree-zero twisted Dennis trace.
\end{abstract}

\maketitle

\tableofcontents
\pagenumbering{arabic}
\setcounter{page}{1}

\section{Introduction}
Algebraic \(K\)-theory encodes subtle arithmetic and geometric information about rings and ring spectra, but it is notoriously difficult to compute directly. Trace methods address this difficulty by relating $K$-theory, via trace maps, to more computable Hochschild-type invariants, most notably topological Hochschild homology \(\THH\) and topological cyclic homology \(\TC\) \cite{BHM, NikolausScholze}. Since \(\THH\) carries a natural \(S^1\)-action, and \(\TC\) is constructed from this action, developments in trace methods are closely tied to equivariant stable homotopy theory.

The formalism of bicategorical shadows naturally explains many structural features of Hochschild homology and its trace maps. Introduced by Ponto and Shulman \cite{PONSH}, shadows give a conceptual framework for cyclic symmetry, Morita invariance, and bicategorical traces, and identify classical Hattori--Stallings traces as instances of the bicategorical trace. The same perspective also governs the topological setting: \(\THH\) may be viewed as a shadow on the homotopy bicategory of bimodule spectra \cite{CP}.

Recent work has produced equivariant refinements of \(\THH\).
For a \(C_n\)-ring spectrum \(R\), the twisted theory
\(\THH_{C_n}(R)\) is constructed from the \(C_n\)-twisted cyclic bar
construction. Its realization carries a natural \(S^1\)-action and,
with this equivariant structure, is identified with the relative norm
\(
N_{C_n}^{S^1}R
\)
\cite{ABGHLM}. When \(C_n=e\), this construction recovers ordinary
\(\THH\) as \(N_e^{S^1}R\). This norm description builds on the
multiplicative norm formalism of Hill--Hopkins--Ravenel \cite{HHR}.

Twisted topological Hochschild homology is related to  
algebraic \(K\)-theory through a twisted analogue of the Dennis trace
\cite[Section~6.1]{AGHKK}. After identifying the source of this trace with the
fixed points of equivariant algebraic \(K\)-theory \cite{Merling2017,MM}, one obtains a map
\(
K_{C_n}(R)^{C_n}
\to
\THH_{C_n}(R)
\)
\cite[Section~6.2]{AGHKK}; see also
\cite{CGK25}.

On the algebraic side, the corresponding invariant for Green
functors is twisted Hochschild homology, a
Mackey-functor-valued algebraic analogue of
\(\THH_{C_n}\) \cite{BGHL}.

However, these twisted Hochschild-type theories do not fit into the
ordinary shadow formalism. As observed in
\cite[Remarks~4.1.6 and~4.2.6]{AGHKK}, they admit analogues of the
usual cyclicity isomorphisms, but cyclic rotation also transports the
twisting by the group action from one side of the composite to the
other. Consequently, these cyclicity isomorphisms do not in general
satisfy the coherence axioms required of an ordinary shadow.

Twisted shadows and traces for an endobicategory
\((\mathcal C,\Sigma)\) were introduced by Beliakova, Putyra, and
Wehrli \cite{BPW}. The framework developed here gives a
group-indexed, internal realization of this point of view. On the
bicategory \(\mathcal B^{G\text{-tw}}\) constructed below, the resulting
\(g\)-twisted shadow axioms agree with the BPW axioms for
\(\Sigma=\operatorname{Id}_{\mathcal B^{G\text{-tw}}}\). The
group-indexed coherence additionally gives rise to conjugation
transport, centralizer actions, and periodicity, and it underlies the
Mackey-functor compatibility established for Green functors below. In
this sense, the construction provides the type of equivariant shadow
framework anticipated in
\cite[Remarks~4.1.6 and~4.2.6]{AGHKK}.

To construct this group-indexed realization, we introduce bicategories
with \(G\)-twisting data
(Definition~\ref{def:G_twisting_data}).
For each object \(A\), the twisting data consist of a strong monoidal
functor
\[
\Theta_A\colon G\longrightarrow \mathrm{Pic}_{\mathcal B}(A)
\]
into the Picard groupoid of invertible endomorphism \(1\)-cells of
\(A\). Given a \(1\)-cell \(M\colon A\to B\) and an element \(g\in G\),
these data determine left and right twists
\[
{}^gM:=\Theta_A(g^{-1})\odot M,
\qquad
M^g:=M\odot\Theta_B(g).
\]

From these data, we construct the bicategory of \(G\)-twists
\(\mathcal B^{G\text{-tw}}\)
(Definition~\ref{def:bicat_of_G_twists}). A \(1\)-cell from \(A\) to
\(B\) in \(\mathcal B^{G\text{-tw}}\) consists of a \(1\)-cell
\(M\colon A\to B\) in \(\mathcal B\), together with coherent twisting
constraints
\(
\tau_{g,M}\colon
{}^{g^{-1}}M
\xrightarrow{\cong}
M^g.
\)
These constraints express how \(M\) intertwines the twists at its
source and target and allow the twist to be transported across \(M\).
In bicategories of bimodules, they are equivalent to semilinear
\(G\)-actions (Corollary~\ref{cor:equiv-mod-G}).

On this bicategory, we define, for each \(g\in G\), a \(g\)-twisted shadow (Definition~\ref{def:g-twisted-shadow}). If the underlying bicategory \(\mathcal B\) with \(G\)-twisting data is equipped with an ordinary shadow \(\sh{-}\), then the \(g\)-twisted shadow is obtained by applying the ordinary shadow to the left-twisted \(1\)-cell:
\(
{}^g\sh{M}:=\sh{{}^gM}.
\)

We prove that these \(g\)-twisted shadows satisfy twisted analogues of the usual shadow cyclicity axioms. We also show that the \(g\)- and \(h\)-twisted shadows are naturally isomorphic whenever \(g\) and \(h\) are conjugate, and that they are invariant under \(G\)-Morita equivalence (Definition~\ref{def:G_morita_equivalence}). When \(g\) has finite order \(n\), we show that the twofold twisted
cyclicity operator has order dividing \(n\); equivalently, the
\(2n\)-fold iterated twisted cyclicity map is the identity.

Our first main result is that twisted Hochschild-type invariants arise from this formalism. In particular, for \(C_n\)-ring spectra, \(\THH_{C_n}\) arises as a twisted shadow.

\begin{thm}[Theorem~\ref{thm:twisted-thh-shadow}]
For the chosen generator \(g=e^{2\pi i/n}\), the assignment sending
\((R,M)\) to \(\THH_{C_n}(R;M)\)
defines a \(g\)-twisted shadow on
\(\Ho\mathbf{Bimod}/\mathbf{Sp}^{C_n}\),
the homotopy bicategory of \(C_n\)-ring spectra and bimodules
(see Example~\ref{ex:HoBimodSpCn}).
\end{thm}

More generally, let \(\mathcal V\) be a symmetric monoidal, simplicial model category such that the homotopy bicategory \(\Ho\mathbf{Mod}(\mathcal V)\) of monoids and bimodules is equipped with \(G\)-twisting data (see Remark~\ref{rem:HoModV}). For every \(g\in G\), the \(g\)-twisted Hochschild construction defines a \(g\)-twisted shadow on the associated bicategory of \(G\)-twists. In particular, this recovers the twisted Hochschild homology of \(C_n\)-Green functors.

\begin{thm}[Theorem~\ref{thm:twisted-HH-shadow} and
Corollary~\ref{cor:twisted-HH-Green}]
Let \(\mathcal V\) be a symmetric monoidal, simplicial model category,
and suppose that \(\Ho\mathbf{Mod}(\mathcal V)\) is equipped with
\(G\)-twisting data. For every \(g\in G\), the assignment
\(
(A,M)\mapsto \HH^g_{\mathcal V}(A;M)
\)
defines a \(g\)-twisted \(\Ho(\mathcal V)\)-valued shadow on 
\(
\Ho\mathbf{Mod}(\mathcal V)^{G\text{-}\mathrm{tw}}.
\)

In particular, take \(G=C_n\) and
\(\mathcal V=s\mathbf{Mack}_{C_n}\), and let
\(g=e^{2\pi i/n}\) be the chosen generator. For every \(i\geq 0\), the
composite
\(
H_i\circ\HH^g_{s\mathbf{Mack}_{C_n}}
\)
defines a \(g\)-twisted \(\mathbf{Mack}_{C_n}\)-valued shadow, denoted
\(\underline{\HH}^{\,C_n}_i\).
\end{thm}

On bimodule Mackey functors regarded as constant simplicial objects,
\(\underline{\HH}^{\,C_n}_i\) agrees with the twisted Hochschild
homology of Green functors in \cite{BGHL,AGHKK}.

We then develop the corresponding trace theory. In the ordinary shadow formalism, the classical Hattori--Stallings trace is recovered as the bicategorical trace of an endomorphism of a dualizable \(1\)-cell in the bicategory of rings and bimodules. The degree-zero Dennis trace is obtained from this construction by applying the Hattori--Stallings trace to identity endomorphisms of finitely generated projective modules.

Our twisted shadow formalism gives an analogous construction: each
\(g\)-twisted shadow gives rise to a corresponding \(g\)-twisted
bicategorical trace \(\operatorname{tr}^g\). The next theorem shows that, for \(G\)-rings, this trace recovers the
\(g\)-twisted Hattori--Stallings trace.
\begin{thm}[Theorem~\ref{thm:twHS-recovery-general}]
Fix \(g\in G\). Let \(R\) and \(S\) be \(G\)-rings, let
\((P,\gamma)\colon R\to S\) be a left dualizable \(1\)-cell in the
bicategory of \(G\)-twists, and let \(f\colon P\Rightarrow P\) be a
\(2\)-cell. Then, for every \(s\in S\),
\(
\operatorname{tr}^g(f)([s])
=
\operatorname{tr}_{\mathrm{HS},\gamma}^g(f\circ\rho_s)
=
\operatorname{tr}_{\mathrm{HS},\gamma}^g(\rho_s\circ f),
\)
where \(\rho_s\colon P\to P\) denotes right multiplication by \(s\).
\end{thm}
For \(C_n\)-Green functors, the same recovery result holds at the
level of morphisms of \(C_n\)-Mackey functors
(Theorem~\ref{thm:tw-eq-bicat-vs-HS}).

This comparison with Hattori--Stallings traces also has a
\(K\)-theoretic consequence: on classes
arising from \(G\)-twisted \(1\)-cells, it gives a categorical realization of
the degree-zero twisted Dennis trace of
\cite[Section~6.1]{AGHKK}. In the homotopical setting, let \(R\) be a
cofibrant monoid equipped with a \(G\)-action. On classes represented
by perfect modules whose maps to their \(g\)-twists arise from a
chosen \(G\)-twist structure, we show that the degree-zero
\(g\)-twisted Dennis trace agrees with the \(g\)-twisted
bicategorical trace
(Proposition~\ref{prop:algebraic_dennis_trace_recovery}). This
comparison extends to generalized Dennis traces with coefficients in
the twisted bimodule \({}^gM\)
(Corollary~\ref{cor:generalized-twisted-trace-comparison}).

We next use the \(g\)-twisted Hattori--Stallings trace to construct a
Mackey-functor refinement of the degree-zero twisted Dennis trace for
Green functors.

\begin{thm}[Theorem~\ref{thm:twisted-equivariant-Dennis}]
Let \(\underline R\) be a \(C_n\)-Green functor and let \(g\in C_n\).
The orbitwise \(g\)-twisted Hattori--Stallings traces descend to
Grothendieck groups and assemble into a morphism of \(C_n\)-Mackey
functors
\(\operatorname{tr}_{\mathrm{Dennis}}^g\colon
\underline K_0^g(\underline R)
\to
\underline{\HH}_0^g(\underline R).
\)
\end{thm}
Its component at \(C_n/C_n\) recovers the degree-zero twisted
Dennis trace of \cite[Section~6.1]{AGHKK}.

The construction extends to arbitrary, not necessarily invertible,
parametrized endomorphisms
\(f\colon\underline P\to{}^g\underline P\), yielding a reduced
morphism of \(C_n\)-Mackey functors
(Proposition~\ref{prop:reduced-twisted-Dennis-Mackey}).
For \(C_n\)-Tambara functors, this reduced trace recovers the \(m=1\)
Teichm\"uller maps of \cite[Theorem~6.13]{BGHL}.

Finally, we establish twisted analogues of the properties of bicategorical traces, including tightening, sliding, additivity, and multiplicativity (Propositions~\ref{prop:tw-tightening}, \ref{prop:twisted_sliding_general}, \ref{prop:tw-additivity}, and \ref{prop:tw-multiplicativity}).

 \subsection*{Organization}
In Section 2, we review the necessary background on equivariant algebra, bicategories with shadows, bicategorical traces, and twisted topological Hochschild homology.  Section~3 introduces bicategories of \(G\)-twists, identifies their
bimodule incarnation with semilinear equivariant bimodules, and records
further examples. Section 4 develops the theory of \(g\)-twisted shadows, establishes their basic properties, and identifies \(\THH_{C_n}\) within this framework. Section~5 develops the corresponding bicategorical trace theory, relates it to twisted Hattori--Stallings traces, provides a categorical realization of the degree-zero twisted Dennis trace on classes
arising from \(G\)-twisted \(1\)-cells, and constructs its Mackey-functor-valued   refinement. Section~6 establishes the properties of the twisted bicategorical trace.

\subsection*{Acknowledgments}
The author has been supported by the Michigan State University Graduate School and by NSF grants DMS-2052042, DMS-2104233, and DMS-2404932. The author thanks the Department of Mathematics at Michigan State University for providing a supportive research environment. The author is especially grateful to Teena Gerhardt for her guidance, as well as to Kate Ponto and Cary Malkiewich for helpful conversations. The author is also grateful to David Chan and Maximilien Péroux for their careful reading of earlier drafts of this paper. The author also thanks the Isaac Newton Institute for Mathematical Sciences, Cambridge, for support and hospitality during the programme “Equivariant homotopy theory in context,” where part of this work was undertaken.

\section{Background}
In this section, we review the background needed in the remainder of the paper. Section 2.1 recalls the framework of equivariant algebra, focusing on the Burnside category, Mackey functors, and Green functors. Section 2.2 reviews shadowed bicategories, and Section 2.3 recalls twisted topological Hochschild homology, which provides the primary topological motivation for our bicategorical framework.

\subsection{Equivariant algebra}
We begin by briefly recalling the Burnside category, Mackey functors, and the monoidal structure on \(\mathbf{Mack}_G\). For   details and standard definitions, we refer the reader to \cite{green1971, dress1971notes}.

\begin{definition}
For a finite group $G$, the \emph{Burnside category} $\mathcal{B}_G$ has finite $G$-sets as objects, and its morphisms are the Grothendieck groups of spans of finite $G$-sets. A \emph{\(G\)-Mackey functor} is an additive functor $\underline M\colon\mathcal B_G\to\mathbf{Ab}$. 
Equivalently, a Mackey functor is determined by its values on orbits $G/H$, equipped with restriction maps $\res_K^H\colon M(G/H)\to M(G/K)$, transfer maps $\tr_K^H\colon M(G/K)\to M(G/H)$ for $K\le H$, and conjugation maps $c_g^*\colon M(G/H)\to M(G/gHg^{-1})$ for $g\in G$ satisfying standard relations, including the Mackey double-coset formula.
A morphism of Mackey functors is a natural
transformation, equivalently a family of homomorphisms compatible
with restriction, transfer, and conjugation.
\end{definition}
 Mackey functors naturally interact with subgroup inclusions via change-of-groups operations:
 
\begin{definition}\label{def:mackey-change-of-groups}
For \(H\leq G\), restriction \(\res_H^G\colon\mathcal B_G\to\mathcal B_H\) and induction \(\ind_H^G(Y) := G\times_HY\) of finite sets induce change-of-groups functors \(\operatorname{Res}_H^G \underline{M} := \underline{M}\circ\ind_H^G\) and \(\operatorname{Ind}_H^G \underline{N} := \underline{N}\circ\res_H^G\). These satisfy \((\operatorname{Res}_H^G \underline{M})(Y) = \underline{M}(G\times_HY)\) and \((\operatorname{Ind}_H^G \underline{N})(X) = \underline{N}(\res_H^G X)\), yielding the canonical identification \((\operatorname{Res}_H^G \underline{M})(H/K) \cong \underline{M}(G/K)\) for \(K\leq H\). Furthermore, \(\operatorname{Res}_H^G\) and \(\operatorname{Ind}_H^G\) form an ambidextrous adjunction \cite{thevenaz1995structure}.
\end{definition}

\begin{example}
    The \emph{Burnside Mackey functor} \(\underline{A}\) assigns to each orbit \(G/H\) the Burnside ring \(\underline{A}(G/H) = A(H)\). For \(K \le H\), the restriction map \(\mathrm{res}_K^H \colon A(H) \to A(K)\) is given by viewing an \(H\)-set as a \(K\)-set, and the transfer map \(\mathrm{tr}_K^H \colon A(K) \to A(H)\) sends a \(K\)-set \(X\) to the induced \(H\)-set \(H \times_K X\). This functor is naturally isomorphic to the representable functor \(\underline{A}_{G/G} := \mathcal{B}_G(G/G,-)\). More generally, for any finite \(G\)-set \(X\), we have the representable Mackey functor \(\underline{A}_X := \mathcal{B}_G(X,-)\), satisfying \(\mathcal{B}_G(X,Y) \cong \underline{A}(X\times Y)\). 
\end{example}

The category \(\mathbf{Mack}_G\) is a closed symmetric monoidal category under the box product \(\square\) (Day convolution), with monoidal unit \(\underline{A}\) and internal Hom \(\underline{\Hom}\) \cite[§2]{LEWMAN}. Evaluating the box product of a Mackey functor \(\underline{M}\) with a representable yields \( \underline{M}\square \underline{A}_X \cong \underline{M}_X \), where \(\underline{M}_X(Y) := \underline{M}(X\times Y)\).

The box product and the change-of-groups functors satisfy the following projection formula.

\begin{lemma}[Projection formula]\label{lem:projection-formula-Mackey}
Let $K \leq H$. For $\underline{M} \in \mathbf{Mack}_H$ and $\underline{N} \in \mathbf{Mack}_K$, there is a natural isomorphism
\(
\operatorname{Ind}_K^H \bigl( (\operatorname{Res}_K^H \underline{M}) \square \underline{N} \bigr) \cong \underline{M} \square \operatorname{Ind}_K^H \underline{N}.
\)
\end{lemma}

Monoid objects in $(\mathbf{Mack}_G,\square,\underline A)$ are called Green functors.

 \begin{definition}
A (commutative) \emph{Green functor} is a (commutative) monoid
in the closed symmetric monoidal category $(\mathbf{Mack}_G,\square,\underline{A})$.
That is, a Green functor is a Mackey functor \(\underline M\) equipped with multiplication and unit morphisms of Mackey functors
\(
\mu:\underline{M}\square\underline{M}\to \underline{M},  \eta:\underline{A}\to \underline{M},
\)
satisfying the usual associativity and unitality axioms.
\end{definition}

\begin{example}\label{ex:endomorphism-green}
For any Mackey functor $\underline{M}$, its internal endomorphism object
\(
\underline{\End}(\underline{M}) := \underline{\Hom}(\underline{M},\underline{M})
\)
is a Green functor. Its multiplication is obtained from Bouc's associative pairing \cite[Def.~6.1.1, Prop.~6.1.2]{bouc1997}
\[
\hat\circ:\ \underline{\Hom}(\underline Y,\underline Z)\square 
\underline{\Hom}(\underline X,\underline Y)\longrightarrow
\underline{\Hom}(\underline X,\underline Z)
\] by taking $\underline X = \underline Y = \underline Z = \underline M$.

Evaluation at $G/G$ gives  canonical identifications
\(
\underline{\Hom}(\underline U,\underline V)(G/G)\ \cong\ 
\mathbf{Mack}_G(\underline U,\underline V).
\)
Under these identifications, $(\hat\circ)_{G/G}$ agrees with ordinary composition:
for $f\in \mathbf{Mack}_G(\underline Y,\underline Z)$ and 
$g\in \mathbf{Mack}_G(\underline X,\underline Y)$,
\[
(\hat\circ)_{G/G}(f,g)=f\circ g\in \mathbf{Mack}_G(\underline X,\underline Z).
\]

More generally, for any finite $G$-set $T$, applying the Yoneda lemma with the representable functor $\underline{A}_T$ and the $(\square, \underline{\Hom})$-adjunction yields
\[
\underline{\Hom}(\underline U,\underline V)(T)
\ \cong\ \mathbf{Mack}_G(\underline A_T, \underline{\Hom}(\underline U,\underline V))
\ \cong\ \mathbf{Mack}_G(\underline A_T\square \underline U,\ \underline V).
\]

Via these isomorphisms, Bouc's pairing $\hat{\circ}$ induces a componentwise pairing
\[
(\hat{\circ})_T: \underline{\Hom}(\underline{Y}, \underline{Z})(T) \times \underline{\Hom}(\underline{X}, \underline{Y})(T) \longrightarrow \underline{\Hom}(\underline{X}, \underline{Z})(T).
\]
Specifically, an element $a \in \underline{\Hom}(\underline{Y}, \underline{Z})(T)$ corresponds to a map $a : \underline{A}_T \square \underline{Y} \to \underline{Z}$, and $b \in \underline{\Hom}(\underline{X}, \underline{Y})(T)$ corresponds to $b : \underline{A}_T \square \underline{X} \to \underline{Y}$. The element $(\hat{\circ})_T(a,b)$ is defined as the unique element whose adjunct map $\underline{A}_T \square \underline{X} \to \underline{Z}$ is given by the composition
\[
\underline{A}_T \square \underline{X} \xrightarrow{\Delta_T \square \text{id}} (\underline{A}_T \square \underline{A}_T) \square \underline{X} \cong \underline{A}_T \square (\underline{A}_T \square \underline{X}) \xrightarrow{\text{id} \square b} \underline{A}_T \square \underline{Y} \xrightarrow{a} \underline{Z},
\]
where the comultiplication $\Delta_T : \underline{A}_T \to \underline{A}_T \square \underline{A}_T$ corresponds to the span $T \times T \xleftarrow{\Delta} T \xrightarrow{\text{id}} T$ in the Burnside category.
\end{example}

We next recall the category of modules over a Green functor.

\begin{definition}
Let $\underline{R}$ be a Green functor. A left \emph{$\underline{R}$-module} is a Mackey functor $\underline{M}$ equipped with an associative and unital action map $\lambda\colon\underline{R}\square \underline{M}\to \underline{M}$. A morphism of left $\underline{R}$-modules is a morphism of Mackey functors compatible with this action. 
\end{definition}

\begin{remark}
The category ${}_{\underline{R}}\mathbf{Mod}$ of left $\underline{R}$-modules is an abelian category, where exactness is checked levelwise at each orbit $G/H$. Orbitwise, $\underline{M}(G/H)$ inherits the structure of an $\underline{R}(G/H)$-module, and the transfer maps satisfy the projection formula $\tr_K^H(\res_K^H(r)m) = r\tr_K^H(m)$ for $r\in\underline{R}(G/H)$ and $m\in\underline{M}(G/K)$ \cite[Def.~2.6]{ventura2005homological}.
\end{remark}

\begin{definition} \label{def:relative-box-equalizer}
For a right $\underline{R}$-module $\underline{M}$ and a left $\underline{R}$-module $\underline{N}$, the \emph{relative box product} $\underline{M} \square_{\underline{R}} \underline{N}$ is defined by the coequalizer:
\[
\begin{tikzcd}
\underline{M}\square \underline{R}\square \underline{N}\arrow[r, shift right = .75ex,"\id\square\lambda"'] 
\arrow[r, shift left = .75ex,"\rho\square \id"]  &\underline{M}\square \underline{N}\ar[r,"b"]&\underline{M}\square_{\underline{R}} \underline{N}.
\end{tikzcd}
\]
Similarly, for left $\underline{R}$-modules $\underline{M}$ and $\underline{N}$, the \emph{internal Hom} $\underline{\Hom}_{\underline{R}}(\underline{M},\underline{N})$ is given by the equalizer: 
\[
\begin{tikzcd}
\underline{\Hom}_{\underline{R}}(\underline{M},\underline{N}) \ar[r] & \underline{\Hom}(\underline{M},\underline{N}) \ar[rr, shift left=.6ex, "f \mapsto \lambda_{\underline{N}}\circ(\id_{\underline{R}}\square f)"] \ar[rr, shift right=.6ex, "f \mapsto f\circ\lambda_{\underline{M}}"'] && \underline{\Hom}(\underline{R}\square\underline{M},\underline{N}).
\end{tikzcd}
\]
\end{definition}

If $\underline{R}$ is commutative, these operations endow ${}_{\underline{R}}\mathbf{Mod}$ with a closed symmetric monoidal structure \cite{lewis1981theory}. Analogous to Example~\ref{ex:endomorphism-green}, the relative internal endomorphism object $\underline{\End}_{\underline{R}}(\underline{M}) := \underline{\Hom}_{\underline{R}}(\underline{M},\underline{M})$ naturally forms a Green functor under Bouc's pairing $\hat{\circ}$.

We will use the following finiteness conditions for $\underline R$-modules.

\begin{definition}\label{def:finiteness-R-modules}Following the notation of \cite[Definition 3.10]{CCM}, let
      \( \underline{R} \) be a Green functor. For a finite \( G \)-set \( X \), denote  
    \(\underline{R}_X := \underline{R} \square \underline{A}_X,
    \)
    which is a left \( \underline{R} \)-module.
    
    \begin{enumerate}
        \item An \( \underline{R} \)-module is \textit{finite free} if it is isomorphic to \( \underline{R}_X \) for some \( X \).
        
        \item An $\underline R$-module $\underline M$ is
\emph{finitely generated} if there is a surjection
\(
\underline R_X\twoheadrightarrow \underline M
\)
for some finite $G$-set $X$.
        
        \item An \( \underline{R} \)-module is (finitely generated) \textit{projective} if it is a direct summand of a finite free $\underline R$-module.
    \end{enumerate}
\end{definition}

\begin{remark}
By definition, for any Green functor $\underline{R}$ and finite $G$-set $X$, the object $\underline{R}_X$ is finite free (hence finitely generated projective) as a left $\underline{R}$-module. Note that, equivalently, an $\underline{R}$-module $\underline{P}$ is projective if and only if it is finitely generated and the functor $\Hom_{\underline{R}}(\underline{P}, -)$ is exact.

In the case of the Burnside ring Green functor $\underline{R}=\underline{A}$, the free modules are simply $\underline{A}_X$. The decomposition of the $G$-set $X \cong \bigsqcup_i G/H_i$ corresponds to the decomposition of $\underline{A}_X$ into transitive summands, $\underline{A}_X \cong \bigoplus_i \underline{A}_{G/H_i}$.
\end{remark}

 \subsection{Shadowed bicategories and traces}
We briefly recall duality and the symmetric monoidal trace. In a symmetric monoidal category $(\mathcal{V}, \otimes, I)$, an object $M$ is \emph{dualizable} if there exists a dual object $M^\ast$ equipped with coevaluation $\eta\colon I \to M \otimes M^{\ast}$ and evaluation $\epsilon\colon M^{\ast}\otimes M \to I$ satisfying the standard triangle identities \cite{DoldPuppe}. For an endomorphism $f\colon M \to M$, its \emph{symmetric monoidal trace} is defined as the composite $\epsilon \circ \mathfrak{s}_{M,M^\ast} \circ (f \otimes \mathrm{id}) \circ \eta$, where $\mathfrak{s}_{M,M^\ast}$ is the symmetry isomorphism. This construction is independent of the choice of dual \cite[Def.~1.2.4]{PSSY}.

By a result of Bouc, the dualizable objects of $(\mathbf{Mack}_G, \square, \underline{A})$ are precisely the finitely generated projective Mackey functors \cite[Lemma~2.2]{boucproj}. We will use representable Mackey functors as our basic examples.

\begin{lemma}\label{lemma:selfdual}
For any finite $G$-set $X$, the representable Mackey functor $\underline{A}_X$ is self-dual in \((\mathbf{Mack}_G, \square, \underline{A})\).
\end{lemma}

\begin{proof}
In the Burnside category \(\mathcal B_G\), every object \(X\) is self-dual via the evaluation and coevaluation spans \(X\times X \xleftarrow{\Delta} X \to G/G\) and \(G/G \leftarrow X \xrightarrow{\Delta} X\times X\). The assignment \(X\mapsto \underline{A}_X=\mathcal B_G(X,-)\) defines a strong symmetric monoidal functor \(\mathcal B_G^{op}\to \mathbf{Mack}_G\), since \(\underline{A}_X\square \underline{A}_Y\cong \underline{A}_{X\times Y}\). Since \(X\) is self-dual in \(\mathcal B_G\), it is also self-dual in \(\mathcal B_G^{op}\), with evaluation and coevaluation interchanged. Applying this functor gives a dual pair \((\underline{A}_X,\underline{A}_X)\) in \(\mathbf{Mack}_G\). Therefore \(\underline{A}_X\) is self-dual.
\end{proof}

The symmetric monoidal trace relies on duality and on the symmetry of the tensor product. Since a monoidal category may be regarded as a bicategory with one object, bicategorical traces provide a natural generalization in which \(1\)-cells and their composition replace objects and tensor products. A general bicategory, however, has no symmetry isomorphism for its composition. A \emph{shadow} supplies the corresponding cyclic structure, enabling traces of dualizable \(1\)-cells.

We first recall the definition of a bicategory.

\begin{definition} \label{def:bicategory_convention}
A \emph{bicategory} \( \mathcal{B} \) consists of:
\begin{itemize}
    \item a collection of objects, also called 0-cells, denoted
    \( A, B, C, \ldots \);
    \item for each pair of objects \( (A,B) \), a category
    \( \mathcal{B}(A,B) \), whose objects are called 1-cells
    \( A \to B \) and whose morphisms are called 2-cells;
    \item composition functors
    \(
    \odot \colon \mathcal{B}(A,B) \times \mathcal{B}(B,C)
    \to \mathcal{B}(A,C);
    \)
    \item unit 1-cells \( U_A \in \mathcal{B}(A,A) \).
\end{itemize}
    
These data are equipped with natural isomorphisms: the \emph{associator} \(\mathfrak a\)
and the left and right \emph{unitors} \(\mathfrak l,\mathfrak r\),
which satisfy the standard pentagon and triangle coherence axioms
(see \cite[Definition~4.1.1]{PSSY}).
\end{definition}

With this convention, if \(M\colon A\to B\) and \(N\colon B\to C\), then
\(
M\odot N\colon A\to C.
\)
Thus horizontal composition is written from left to right: one first applies \(M\) and then \(N\).

\begin{example}\label{ex:bicat_examples}
Examples of bicategories include the following:
\begin{itemize}
  \item Any monoidal category $(\mathcal{V}, \otimes, I)$ can be viewed as a bicategory with a single object $\ast$.
  The 1-cells are the objects of $\mathcal{V}$, the 2-cells are the morphisms of $\mathcal{V}$, composition is given by $\otimes$, and the unit 1-cell is $I$.
\item \cite[\S7.5]{MAL} The bicategory $\mathcal{R}/\mathcal{T}op$, whose objects are spaces and whose 1-cells $A \to B$ are retractive spaces $A \times B \xrightarrow{i} X \xrightarrow{p} A \times B$ with $p \circ i = \mathrm{id}$. Composition of $X \in \mathcal{R}(A \times B)$ and $Y \in \mathcal{R}(B \times C)$ is defined via the external smash product $\bar{\wedge}$ along with pullback and pushforward: $X \odot Y \coloneqq (r_B)_! \,\Delta_B^* (X \,\bar{\wedge}\, Y)$, where $\Delta_B$ is the diagonal and $r_B \colon B \to *$ is the terminal map. The unit $U_A$ is $(\Delta_A)_!(r_A)^*S^0$.
\end{itemize}
\end{example}

The primary examples in this paper arise from bimodules.

\begin{example}\label{ex:ModV}
Let $(\mathcal V,\otimes,I)$ be a bicomplete closed symmetric monoidal category. For monoids $A,B$ in $\mathcal V$, write ${}_A\mathbf{Mod}_B$ for the category of $(A,B)$-bimodules in $\mathcal V$. We write $\mathbf{Mod}(\mathcal V)$ for the bicategory whose objects are monoids in $\mathcal V$, whose 1-cells $A\to B$ are objects of ${}_A\mathbf{Mod}_B$, and whose 2-cells are morphisms in ${}_A\mathbf{Mod}_B$. Composition is given by the relative tensor product in $\mathcal{V}$.
\begin{enumerate}
    \item If $\mathcal{V} = \mathbf{Ab}$, this is the bicategory $\mathbf{Bimod}/\mathcal{R}\mathrm{ing}$ of rings and bimodules.
    \item If $\mathcal{V} = \mathbf{Ch}(\mathbf{Ab})$ with unit $I = \mathbb{Z}_\bullet$ (the complex concentrated in degree 0), this is the bicategory of differential graded (DG) rings and DG bimodules. The bicategory $\mathbf{Ch}/\mathcal{R}\mathrm{ing}$ is the full sub-bicategory spanned by ordinary rings, with chain complexes of bimodules as 1-cells.
    \item If $\mathcal{V} = \mathbf{Mack}_{G}$ equipped with the box product $\square$, this is the bicategory $\mathbf{Bimod}/\mathbf{Green}^{G}$ of $G$-Green functors and bimodule Mackey functors.
    \item If \(\mathcal V=(\mathcal{T}\text{op}_*,\wedge,S^0)\), then
\(\mathbf{Mod}(\mathcal V)\) is the bicategory of based topological monoids
and bimodules in based spaces. The bicategory \(\mathcal{T}\text{op}_*/\mathcal{G}\text{p}\) is the full
sub-bicategory spanned by the monoids \(G_+=G\sqcup\{*\}\) arising from
topological groups \(G\), with multiplication induced by the group multiplication
on \(G\). A \((G_+,H_+)\)-bimodule is equivalently a based space with commuting
left \(G\)- and right \(H\)-actions fixing the basepoint.
\end{enumerate}
\end{example}

 \begin{remark}\label{rem:HoModV}
Many bicategories admit homotopical versions, obtained by localizing the categories of $1$-cells at weak equivalences and then passing to homotopy classes of $1$-cell maps for the $2$-cells. 

Let $\mathcal V$ be a symmetric monoidal simplicial model category satisfying \cite[Assumption~3.1.1]{AGHKK}. Following \cite[Notation~3.1.3]{AGHKK}, for a simplicial object $X_\bullet$ in $\mathcal V$, we write
\(
|X_\bullet| \coloneqq \left( \operatorname{hocolim}_{\Delta^{\mathrm{op}}}X_\bullet \right)^f,
\)
where $(-)^f$ denotes functorial fibrant replacement.

We write $\Ho\mathbf{Mod}(\mathcal V)$ for the bicategory $\mathcal R_{\mathcal V}$ of \cite[Definition~3.1.2]{AGHKK}. 
Its objects are monoids with cofibrant underlying objects, its hom-categories are the homotopy categories $\Ho\,{}_A\mathbf{Mod}_B$, and composition is given by  $M\odot N = \vert{}B_\bullet(M;B;N)\vert{}$.
\begin{itemize}
    \item For $\mathcal V=\mathbf{Sp}$ (orthogonal spectra), $\Ho\mathbf{Mod}(\mathbf{Sp})$ is the bicategory $\Ho\mathbf{Bimod}/\mathcal{R}_{\mathbf{Sp}}$. Its objects are ring spectra whose underlying spectra are cofibrant, its 1-cells are obtained by localizing at stable equivalences \cite{MMSS}, and composition models the derived smash product \cite{EKMM97}.

    \item For \(\mathcal V=\mathbf{sAb}\),
\(\Ho\mathbf{Mod}(\mathbf{sAb})\) is the homotopy bicategory of
simplicial rings and simplicial bimodules. If
\(M\colon A\to B\) and \(N\colon B\to C\), their horizontal composite is
\(
M\odot N
=
\left|B_\bullet(M;B;N)\right|,
\)
which represents the derived relative tensor product
\(M\otimes_B^{\mathbf L}N\).

For ordinary rings \(A\) and \(B\), regarded as constant simplicial
rings, the Dold--Kan correspondence identifies the hom-category of
simplicial \((A,B)\)-bimodules with the homotopy category of
nonnegatively graded complexes of \((A,B)\)-bimodules. Under this
identification, horizontal composition corresponds to the derived
relative tensor product.

    \item For \(\mathcal V=s\mathbf{Mack}_G\), the category of
    simplicial \(G\)-Mackey functors with the box product, the required
    model-categorical assumptions hold by
    \cite[Theorem~4.3 and Proposition~4.4]{BGHL}. The resulting
    bicategory
    \(
    \Ho\mathbf{Mod}(s\mathbf{Mack}_G)
    \)
    has simplicial \(G\)-Green functors as \(0\)-cells and homotopy
    categories of simplicial bimodule Mackey functors as its
    hom-categories.
\end{itemize}
\end{remark}

Traces in symmetric monoidal categories generalize to bicategories once a \emph{shadow} is specified, encoding the cyclic invariance property in a broader context \cite{PON1}. 

\begin{definition}[{\cite[Definition~4.4.1]{PON1}}]\label{def:classical_shadow}
Let $\mathcal{B}$ be a bicategory and $\mathcal{T}$ a fixed category.  
A \emph{$\mathcal{T}$-shadow functor} on $\mathcal{B}$ consists of:
\begin{itemize}
  \item For each object $R \in \mathcal{B}$, a functor
  \(
    \sh{-}_R \colon \mathcal{B}(R,R) \to \mathcal{T},
  \)
  \item together with a natural isomorphism (\emph{cyclic invariance})
  \(
    \theta_{M,N} \colon \sh{M \odot N} \;\cong\; \sh{N \odot M}
  \)
  for all 1-cells $M \colon R \to S$ and $N \colon S \to R$,
\end{itemize}
 that satisfies the hexagon and triangle coherence diagrams \[\begin{tikzcd}
\sh{ (M\odot N)\odot P }  \ar[r,"\theta"] \ar[d,"\sh{ \mathfrak{a} } "']  & \sh{ P\odot (M\odot N) } \ar[r, "\sh{\mathfrak{a}}" ] & \sh{ (P\odot M)\odot N } \\
  \sh{ M\odot (N\odot P) }  \ar[r,"\theta"]  &  \sh{(N\odot P)\odot M }\ar[r, "\sh{ \mathfrak{a} }"]& \sh{ N\odot (P\odot M) }\ar[u, "\theta"']
  \end{tikzcd}\]
  
\[ \begin{tikzcd}
  \sh{ M\odot  U_R}  \arrow[r,"\theta"] \arrow[rd,"\sh{\mathfrak{r}}"']&
  \sh{ U_R\odot M} \arrow[r,"\theta"]\arrow[d,"\sh{\mathfrak{l} }"] &\sh{ M\odot  U_R} \arrow[ld,"\sh{\mathfrak{r} }"]\\
  &   \sh{ M} &
    \end{tikzcd}\]
\end{definition}

\begin{definition}
A  \emph{$\mathcal{T}$-shadowed bicategory} is a pair $(\mathcal B,\sh{-})$ where $\mathcal B$ is a bicategory 
and $\sh{-}$ is a $\mathcal T$-shadow functor on $\mathcal B$.
\end{definition}

\begin{example}[{\cite[§4-§6]{PONSH}}]
Examples of shadowed bicategories include:
\begin{itemize}
  \item In $\mathbf{Bimod}/\mathcal{R}\mathrm{ing}$, an
$\mathbf{Ab}$-shadow of an $R$-bimodule $M$ is given by
\[
\sh{M}
=
\mathrm{coeq}\!\bigl(
R \otimes M
\mathrel{
  \begin{matrix}
    \xrightarrow{\mathmakebox[\widthof{$\scriptstyle \rho\circ \mathfrak{s}$}][c]{\scriptstyle \lambda}} \\[-1.5ex]
    \xrightarrow[{\mathmakebox[\widthof{$\scriptstyle \rho\circ \mathfrak{s}$}][c]{\scriptstyle \rho\circ \mathfrak{s}}}]{} 
  \end{matrix}
}
M
\bigr)
\cong
\HH_0(R;M),
\]
where $\lambda: R\otimes M\to M$ is the left $R$-action,
$\rho: M\otimes R\to M$ is the right $R$-action, and
$\mathfrak{s}:R\otimes M\xrightarrow{\cong} M\otimes R$
is the symmetry.  Thus $\sh{M}$ is the zeroth
Hochschild homology $\HH_0(R;M)$.
    
  \item In $\mathbf{Ch}/\mathcal{R}\mathrm{ing}$, a 
$\mathbf{Ch}_{\mathbb{Z}}$-shadow of an $A$-bimodule complex $M_\bullet$ is defined
in the same way as for $\mathbf{Bimod}/\mathcal{R}\mathrm{ing}$, but now taken in the category
$\mathbf{Ch}_{\mathbb{Z}}$ of chain complexes of abelian groups.
Permuting tensor factors with the Koszul sign convention gives the cyclic isomorphism $\theta$.

  \item In $\mathcal{T}\mathrm{op}_*/\mathcal{G}p$, a 
$\mathcal{T}\mathrm{op}_*$-shadow of a $G$-$G$-space $M$ is given by
\[
\sh{M}
=
\mathrm{coeq}\!\bigl(
G_+ \wedge M
\mathrel{
  \begin{matrix}
    \xrightarrow{\mathmakebox[\widthof{$\scriptstyle \rho\circ \mathfrak{s}$}][c]{\scriptstyle \lambda}} \\[-1.5ex]
    \xrightarrow[{\mathmakebox[\widthof{$\scriptstyle \rho\circ \mathfrak{s}$}][c]{\scriptstyle \rho\circ \mathfrak{s}}}]{} 
  \end{matrix}
}
M
\bigr),
\]
where $\lambda: G_+\wedge M\to M$ is the left $G$-action,
$\rho: M\wedge G_+\to M$ is the right $G$-action, and
$\mathfrak{s}: G_+\wedge M\xrightarrow{\cong} M\wedge G_+$
is the symmetry.
\end{itemize}
\end{example}

Important examples of shadows arise on homotopical variants of bicategories. We briefly recall the Hochschild construction, which unifies ordinary and topological Hochschild homology, and explain how it fits into the framework of shadows.

 \begin{definition}\label{def:cyclic_bar_general}
Let $\mathcal V$ be a symmetric monoidal simplicial model category satisfying the assumptions of Remark~\ref{rem:HoModV}. Let $A$ be a monoid in $\mathcal V$, and let $M$ be an $A$-bimodule. The \emph{cyclic bar construction} is the simplicial object $B_\bullet^{\mathrm{cyc}}(A;M)$ in $\mathcal V$ given in degree $q$ by $B_q^{\mathrm{cyc}}(A;M) = M\otimes A^{\otimes q}$.

For $q\geq 1$, the two face maps $d_0, d_q \colon B_q^{\mathrm{cyc}}(A;M) \to B_{q-1}^{\mathrm{cyc}}(A;M)$ involving the bimodule actions are
\[
d_0 = \rho\otimes\id_A^{\otimes(q-1)},\qquad
d_q = (\lambda\otimes\id_A^{\otimes(q-1)})\circ\tau_q,
\]
where $\rho$ and $\lambda$ are the right and left actions, and $\tau_q$ cyclically permutes the last $A$-factor to the front. For $0<i<q$, the remaining face maps and the degeneracy maps are respectively given by
\(
d_i = \id_M\otimes\id_A^{\otimes(i-1)}\otimes\mu\otimes\id_A^{\otimes(q-i-1)}
\)
and
\(
s_i = \id_M\otimes\id_A^{\otimes i}\otimes\eta\otimes\id_A^{\otimes(q-i)}
\)
($0\le i\le q$), where $\mu$ and $\eta$ denote the multiplication and unit of $A$.

The \emph{Hochschild construction} of $A$ with coefficients in $M$ \cite[Definition~3.1.4]{AGHKK} is defined as
\[
\HH_{\mathcal V}(A;M) \coloneqq \left|B_\bullet^{\mathrm{cyc}}(A;M)\right|.
\]
\end{definition}

\begin{remark}\label{rem:HH-construction}
By \cite[Proposition~3.1.5]{AGHKK}, the Hochschild construction equips $\Ho\mathbf{Mod}(\mathcal V)$ with a $\Ho\mathcal V$-valued shadow. This unified framework subsumes the classical examples:
\begin{itemize}
\item For \(\mathcal V=\mathbf{sAb}\), the Dold--Kan correspondence
identifies \(\Ho(\mathbf{sAb})\) with
\(\Ho\bigl(\mathbf{Ch}_{\mathbb Z}^{\geq 0}\bigr)\).
Under the chain-complex model described in
Remark~\ref{rem:HoModV}, the value of the Hochschild construction on
a nonnegatively graded complex \(M_\bullet\) of \(A\)-bimodules is
represented by the total complex associated to the simplicial chain
complex
\(
B_\bullet^{\mathrm{cyc}}(A;M_\bullet).
\)
When \(M_\bullet=M[0]\) is concentrated in degree zero, this total
complex is the ordinary Hochschild chain complex, and its homology is
\(\HH_*(A;M)\).
    \item For $\mathcal V=\mathbf{Sp}$ (with smash products replacing tensor products), it recovers topological Hochschild homology $\THH(A;M)$ as a $\Ho\mathbf{Sp}$-valued shadow \cite{CP}.
\end{itemize}
When \(M=A\), the simplicial object
\(B_\bullet^{\mathrm{cyc}}(A)\) extends to a cyclic object
\cite{loday2013cyclic}. In the spectral case, its realization is
\(\THH(A)\) and carries the usual \(S^1\)-action \cite{BHM}.
\end{remark}

Bicategorical traces require dualizable 1-cells, whose definition we recall next.

\begin{definition}[{\cite[Def.~16.4.1]{may2004parametrized}}]\label{def:dualizability}
Let \(\mathcal{B}\) be a bicategory. A 1-cell \(M \colon A \to B\) is \emph{right dualizable} if there exists a 1-cell \(M^* \colon B \to A\) and 2-cells \(\eta \colon U_A \Rightarrow M \odot M^*\) (coevaluation) and \(\epsilon \colon M^* \odot M \Rightarrow U_B\) (evaluation) satisfying the standard triangle identities. The pair \((M,M^*)\) is called a \emph{dual pair}. Dually, \(M\) is \emph{left dualizable} if there exists a \emph{left dual} \({}^*M \colon B \to A\) and 2-cells \(\bar{\eta} \colon U_B \Rightarrow {}^*M \odot M\) and \(\bar{\epsilon} \colon M \odot {}^*M \Rightarrow U_A\) satisfying the analogous triangle identities.
\end{definition}

Dual pairs also encode Morita equivalence. In classical algebra, Morita equivalent rings (e.g., a ring \(R\) and its matrix ring \(\operatorname{Mat}_n(R)\) have equivalent categories of modules. Bicategorically, this phenomenon is expressed by invertible dual pairs.

\begin{definition}\label{def:morita_equivalence}
A pair of 1-cells \(M\colon R\to S\) and \(N\colon S\to R\) in \(\mathcal{B}\) is a \emph{Morita equivalence} between \(R\) and \(S\) if both \((M,N)\) and \((N,M)\) are dual pairs, and the evaluation map of each dual pair is an isomorphism whose inverse is the coevaluation map of the other. An assignment on the \(0\)-cells of \(\mathcal{B}\) is a \emph{Morita invariant} if it takes Morita equivalent \(0\)-cells to isomorphic values.
\end{definition}

\begin{prop}[{\cite[Prop.~4.6, Prop.~4.8]{CP}}]
Let \((\mathcal{B},\sh{-})\) be a shadowed bicategory. If \(R\) and \(S\) are Morita equivalent objects of \(\mathcal{B}\), then \(\sh{U_R}\cong \sh{U_S}\). Equivalently, the shadow of the unit \(1\)-cell is a Morita invariant.
\end{prop}

Dual pairs also allow one to define traces of suitable \(2\)-cells in a shadowed bicategory.

\begin{definition}[{\cite[Def.~4.5.1]{PON1}}]\label{def:bicat-trace}
Let $(\mathcal B,\sh{-})$ be a shadowed bicategory, and let $(M,N)$ be a dual pair
with $M\colon A\to B$ and $N\colon B\to A$.
For 1-cells $Q\colon A\to A$ and $P\colon B\to B$, the \emph{bicategorical trace}
of a 2-cell \(f \colon Q \odot M \Rightarrow M \odot P,\) denoted \( \tr_{\mathcal B}(f) \), is the composite
\[
\begin{aligned}
\sh{Q}
&\cong \sh{Q \odot U_A}
\xrightarrow{\sh{\id_Q \odot \eta}}
\sh{Q \odot M \odot N}
\xrightarrow{\sh{f \odot \id_N}}
\sh{M \odot P \odot N}
\\
&\xrightarrow{\theta}
\sh{N \odot M \odot P}
\xrightarrow{\sh{\epsilon \odot \id_P}}
\sh{U_B \odot P}
\cong \sh{P}.
\end{aligned}
\]

Dually, for a 2-cell \( g\colon N \odot Q \Rightarrow P \odot N \), the bicategorical trace 
\( \operatorname{tr}_{\mathcal B}(g) \) is the composite
\[
\begin{aligned}
\sh{Q}
&\cong \sh{U_A \odot Q}
\xrightarrow{\sh{\eta \odot \id_Q}}
\sh{M \odot N \odot Q}
\xrightarrow{\sh{\id_M \odot g}}
\sh{M \odot P \odot N}
\\
&\xrightarrow{\theta}
\sh{P \odot N \odot M}
\xrightarrow{\sh{\id_P \odot \epsilon}}
\sh{P \odot U_B}
\cong \sh{P}.
\end{aligned}
\]
Here $\eta$ and $\epsilon$ are the coevaluation and evaluation maps for the dual pair $(M,N)$, and $\theta$ denotes the cyclic symmetry of the shadow.
\end{definition}

\begin{remark}Taking $Q = U_A$ and $P = U_B$ yields the definition of the trace of an endomorphism
\( f\colon M \Rightarrow M \) in the bicategory~$\mathcal B$.
Moreover, if \(\mathcal B\) is the one-object bicategory associated to a
symmetric monoidal category \((\mathcal V,\otimes,I)\), and the shadow
\(\sh{-}\) is the identity functor with cyclicity isomorphism induced by the
symmetry of \(\mathcal V\), then this construction reduces to the  
symmetric monoidal trace.

Analogous formulas define the trace when \(M\) is left dualizable. In that case,
one uses a left dual \({}^*M\) together with the corresponding coevaluation and
evaluation maps \(\bar{\eta}\) and \(\bar{\epsilon}\) to define the trace of a
2-cell of the form \( M \odot Q \Rightarrow P \odot M \).
\end{remark}

Following Ponto and Shulman \cite{PONSH}, we recall how the classical
Hattori--Stallings trace is recovered as a bicategorical trace in
\(\mathbf{Bimod}/\mathcal R\mathrm{ing}\), illustrating the algebraic
content of the shadow formalism.

\begin{definition}[{\cite{HA}}]\label{def:HS-ordinary}
Let $R$ be a ring and $P$ a finitely generated projective right $R$-module with dual 
$P^*=\Hom_R(P,R)$.  
Using the isomorphism $\delta: P \otimes_R P^* \cong \End_R(P)$ and the map 
$\pi: P \otimes_R P^* \to R/[R,R]$, $\pi(p \otimes \xi) \coloneqq [\xi(p)]$, the 
\emph{Hattori--Stallings trace} of $f \in \End_R(P)$ is
\(
\operatorname{tr}_{\mathrm{HS}}(f) := \pi(\delta^{-1}(f)).
\)
\end{definition}

\begin{example}[{\cite[Ex.~5.3, Ex.~6.1]{PONSH}}]\label{ex:HS-as-bicat-trace}
In the bicategory $\mathbf{Bimod}/\mathcal{R}\mathrm{ing}$, 
a finitely generated projective right $R$-module $M$,  viewed as a $(\mathbb{Z}, R)$-bimodule, is right dualizable  with dual 
$M^*=\Hom_R(M,R)$.  
For an endomorphism $f:M\Rightarrow M$, the bicategorical trace
\[
\operatorname{tr}_{\mathcal B}(f): \sh{\mathbb{Z}}=\mathbb{Z}\;\longrightarrow\;\sh{R}=\HH_0(R;R)
\]
recovers the Hattori--Stallings trace: $\operatorname{tr}_{\mathcal B}(f)(1)=\operatorname{tr}_{\mathrm{HS}}(f)\ \in\ \HH_0(R;R).$

More generally, let $M$ be an $(S,R)$-bimodule which is finitely generated projective as a right 
$R$-module. For any $(S,R)$-bimodule endomorphism $f:M\to M$, the bicategorical trace
\[
\operatorname{tr}_{\mathcal B}(f):\ \sh{S}=\HH_0(S;S)\longrightarrow \sh{R}=\HH_0(R;R)
\]
is
\(
\operatorname{tr}_{\mathcal B}(f)([s])=\operatorname{tr}_{\mathrm{HS}}(\lambda_s\circ f)=\operatorname{tr}_{\mathrm{HS}}(f\circ\lambda_s),
\)
where $s\in S$ and $\lambda_s:M\to M$ is left multiplication by $s$.
In particular, for the Euler characteristic ($f=\id_M$),
\(
\operatorname{tr}_{\mathcal B}(\id_M)([s])=\operatorname{tr}_{\mathrm{HS}}(\lambda_s)\in \HH_0(R;R).
\)
Equivalently, this means that
the following diagram in $\mathbf{Ab}$ commutes:
\begin{equation}\label{diag:classical-bicat-vs-HS}
\begin{tikzcd}
S \arrow[r, "{\,s\mapsto f\circ\lambda_s\,}"] \arrow[d, "q"'] 
  & \End_R(M) \arrow[d, "\operatorname{tr}_{\mathrm{HS}}"] \\
\HH_0(S;S) \arrow[r, "\operatorname{tr}_{\mathcal B}(f)"'] 
  & \HH_0(R;R)
\end{tikzcd}
\end{equation}
where \(q\colon S\to \HH_0(S;S)\cong S/[S,S]\) is the canonical quotient map.
\end{example}

The Hattori--Stallings trace is defined analogously for finitely generated projective left $R$-modules, which is the convention we use below.
\begin{remark}Let $\End(A)$ denote the exact category whose objects are pairs $(P, f)$, where $P$ is a finitely generated projective left $A$-module and $f \colon P \to P$ is a left $A$-module endomorphism. 
There is a natural homomorphism
\[
K_0(\End(A)) \;\longrightarrow\; \HH_0(A)\cong A/[A,A], \qquad [(P,f)] \mapsto \operatorname{tr}_{\mathrm{HS}}(f),
\]
given by the Hattori--Stallings trace \cite{HA}. 
Restricting along $K_0(A)\to K_0(\End(A))$, $[P]\mapsto[(P,\id_P)]$, recovers the Dennis trace in degree 0 \cite[V.11.1]{weibel:homological}:
\[
\operatorname{tr}_{\mathrm{Dennis}}: K_0(A)\to \HH_0(A), \quad [P]\mapsto \operatorname{tr}_{\mathrm{HS}}(\id_P).
\]

Properties of bicategorical traces, such as cyclicity \cite[Prop.~7.2 and Cor.~7.3]{PONSH} and additivity \cite[Thm.~1.4]{PSL}, recover the familiar identities for the Hattori--Stallings trace \cite{HA}:
\[
\operatorname{tr}_{\mathrm{HS}}(f\oplus g)
=
\operatorname{tr}_{\mathrm{HS}}(f)+\operatorname{tr}_{\mathrm{HS}}(g),
\qquad
\operatorname{tr}_{\mathrm{HS}}(gf)
=
\operatorname{tr}_{\mathrm{HS}}(fg),
\]
where $f\colon P\to Q$ and $g\colon Q\to P$ are morphisms between finitely generated projective modules.
\end{remark}

\subsection{Twisted topological Hochschild homology}
We recall the definition of twisted topological Hochschild homology. Let $C_n$ denote the cyclic group of order $n$. To define this construction, we first introduce the notion of a twisted bimodule associated to an automorphism.

\begin{definition}\label{def:twisted-bimodule-spectra}
Let $R$ be a ring spectrum and let $g\in \Aut(R)$. Let $M$ be an $R$-bimodule
with left action $\lambda \colon R \wedge M \to M$ and right action
$\rho \colon M \wedge R \to M$.

The \emph{left $g$-twisted bimodule} ${}^gM$ is the $R$-bimodule with the same
underlying spectrum $M$ and the same right action $\rho$, but with left action
obtained by precomposing $\lambda$ with $g \wedge \id_M$, namely
\[
{}^g\lambda \coloneqq \lambda \circ (g \wedge \id_M).
\]
\end{definition}

\begin{remark}\label{rem:right-twisted-bimodule}
Dually, one may also define a right-twisted bimodule $M^g$ by keeping the left
action fixed and precomposing the right action map with $\id_M \wedge g$.
The simultaneous appearance of left and right twists will be formalized
bicategorically in Section~3.
\end{remark}

Now suppose that \(R\) is a \(C_n\)-ring spectrum. Identifying $C_n$ with the subgroup $\langle e^{2\pi i/n}\rangle \subset S^1$, we let $g=e^{2\pi i/n}$ denote the chosen generator.
 The action of \(g\)
on \(R\) determines a ring-spectrum automorphism, which we also
denote by \(g\).

\begin{definition}[{\cite[Def.~8.1]{ABGHLM},
                      \cite[\S4.2]{AGHKK}}]
Let $R$ be a $C_n$-ring spectrum and $M$ a \(C_n\)-equivariant $R$-bimodule. The \emph{$C_n$-twisted cyclic bar construction} is the
simplicial \(C_n\)-spectrum $B^{\mathrm{cyc},C_n}_\bullet(R; M)$ with
\(
B^{\mathrm{cyc},C_n}_q(R; M) \coloneqq M \wedge R^{\wedge q}.
\)
The degeneracy maps $s_i$ are induced by the unit $\eta \colon S \to R$ in the standard way. For $q \ge 1$, the face maps \(d_i\colon
B_q^{\mathrm{cyc},C_n}(R;M)
\to
B_{q-1}^{\mathrm{cyc},C_n}(R;M)\) are induced by the bimodule actions and multiplication $\mu$ of $R$, except that the last face map uses the twisted
bimodule structure ${}^gM$  associated with the specific generator $g$:
\begin{itemize}
    \item $d_0 = \rho \wedge \id_{R^{\wedge (q-1)}}$,
    \item $d_i = \id_{M} \wedge \id_{R^{\wedge (i-1)}} \wedge \mu \wedge \id_{R^{\wedge (q-1-i)}} \quad (0 < i < q)$,
    \item $d_q = ({}^g\lambda \wedge \id_{R^{\wedge (q-1)}}) \circ \tau_q$,
\end{itemize}
where $\tau_q \colon M \wedge R^{\wedge q} \xrightarrow{\cong} R \wedge M \wedge R^{\wedge (q-1)}$ is the cyclic
permutation moving the last $R$-factor to the front.
When $M=R$, we write $B^{\mathrm{cyc},C_n}_\bullet(R)$.
\end{definition}

\begin{remark}\label{rem:twisted-bar-equivalence}
The twisted cyclic bar construction may equivalently be viewed as the
ordinary cyclic bar construction with twisted coefficients. Directly
from the definitions, there is an isomorphism of simplicial
\(C_n\)-spectra
\(
B^{\mathrm{cyc},C_n}_\bullet(R;M)
\cong
B^{\mathrm{cyc}}_\bullet(R;{}^gM).
\)
\end{remark}

\begin{definition}[{\cite[Def.~4.2.2]{AGHKK}}]
\label{def:twisted-THH-coeff}
The \emph{\(C_n\)-twisted topological Hochschild homology} of \(R\)
with coefficients in \(M\) is the \(C_n\)-spectrum
\begin{equation}\label{eq:twisted-thh-cyc-realization}
\THH_{C_n}(R;M)
\coloneqq
\left|B^{\mathrm{cyc},C_n}_\bullet(R;M)\right|
\cong
\left|B^{\mathrm{cyc}}_\bullet(R;{}^gM)\right|.
\end{equation}
When \(M=R\), we write simply \(\THH_{C_n}(R)\).
The displayed isomorphism of \(C_n\)-spectra is induced by the
isomorphism of simplicial \(C_n\)-spectra from
Remark~\ref{rem:twisted-bar-equivalence}.
\end{definition}

\begin{remark}
When \(M=R\), the twisted simplicial object carries a
\(\Lambda_n^{op}\)-structure, and its realization therefore carries
an \(S^1\)-action extending the original \(C_n\)-action. After
incorporating the standard change-of-universe functors, this
\(S^1\)-spectrum is the relative norm
\(
N_{C_n}^{S^1}R
\)
of \cite[Def.~8.2]{ABGHLM}. Thus, suppressing change-of-universe
notation, one writes
\(
\THH_{C_n}(R)\simeq N_{C_n}^{S^1}R.
\)
\end{remark}

Twisted topological Hochschild homology has a natural algebraic
analogue in the context of Green functors. As above, let
\(g=e^{2\pi i/n}\) denote the chosen generator of
\(C_n\subset S^1\).

\begin{definition}[{\cite[Defs.~2.20 and~2.25]{BGHL},
                         \cite[Defs.~4.3.1--4.3.3]{AGHKK}}]
\label{def:twisted-HH-Green}
Let \(\underline R\) be a \(C_n\)-Green functor and let
\(\underline M\) be an \(\underline R\)-bimodule. The
\emph{\(C_n\)-twisted Hochschild homology of \(\underline R\) with
coefficients in \(\underline M\)} is
\[
\underline{\HH}^{C_n}_*
(\underline R;\underline M)
\coloneqq
H_*\!\left(
B^{\mathrm{cyc}}_\bullet
(\underline R;{}^g\underline M)
\right).
\]
Here \(B^{\mathrm{cyc}}_\bullet
(\underline R;{}^g\underline M)\) is the simplicial \(C_n\)-Mackey
functor whose \(q\)-simplices are
\(
B^{\mathrm{cyc}}_q
(\underline R;{}^g\underline M)
=
{}^g\underline M
\,\square\,
\underline R^{\square q},
\)
where \(\square\) denotes the box product. The face and degeneracy
maps are the usual Hochschild maps. In particular, the last face map
uses the twisted left action on \({}^g\underline M\), defined by
precomposing the original left action with the automorphism induced
by \(g\). When \(\underline M=\underline R\), we write simply
\(\underline{\HH}^{C_n}_*(\underline R)\).
\end{definition}

These twisted Hochschild-type theories satisfy analogues of cyclic
invariance. However, as observed in
\cite[Remarks~4.1.6 and~4.2.6]{AGHKK}, the resulting cyclicity maps do
not in general satisfy the coherence axioms required of an ordinary
shadow. Thus they do not define ordinary shadows. In the topological
case, however, the \(C_n\)-fixed points of
\(\THH_{C_n}(R;-)\) form an ordinary shadow by
\cite[Proposition~4.2.5]{AGHKK}.

This failure of ordinary shadow coherence is precisely what motivates
the equivariantly twisted refinement introduced in the next section.

\section{Bicategories of $G$-twists}
In this section, we formalize equivariant twisting data at the bicategorical level. In Section 2.3, twisting appears concretely by precomposing the left or right action maps with a chosen automorphism. Bicategorically, this operation is recast as pre- and postcomposition by invertible 1-cells, leading to the notion of $G$-twisting data on the 0-cells of a bicategory. The associated bicategory of \(G\)-twists then consists of certain
\(1\)-cells equipped with twisting constraints that intertwine the
twists at their sources and targets. These constraints play the role of equivariance data in the bicategorical setting. In the motivating bicategories $\mathbf{Mod}(\mathcal{V})$, this recovers the familiar notion of semilinear $G$-action on bimodules. 

By the coherence theorem for bicategories \cite{maclane1985coherence}, we may systematically suppress the associators $\mathfrak{a}$ and unitors $\mathfrak{l}, \mathfrak{r}$ in subsequent computations and diagrams. We treat the composition of 1-cells as associative and unital unless explicit interaction with the $G$-twisting data is required. All horizontal composites in this paper are written using the convention of Definition~\ref{def:bicategory_convention}. Thus, if $M\colon A\to B$ and $N\colon B\to C$, then $M\odot N: A\to C$ means first $M$, then $N$.

\subsection{Invertible 1-cells and twisted 1-cells}
This subsection describes how to twist a 1-cell in an arbitrary bicategory and records the compatibilities of this operation with units and composition.
\begin{definition}\label{def:invertible_1_cell}
Let \(\mathcal B\) be a bicategory and let \(A\) be an object.
An \emph{invertible \(1\)-cell} \(u\colon A\to A\) is a \(1\)-cell
equipped with a chosen \(1\)-cell \(u^{-1}\colon A\to A\) and
specified invertible \(2\)-cells
\(
u\odot u^{-1}\xRightarrow{\cong} U_A,
u^{-1}\odot u\xRightarrow{\cong} U_A.
\)
The \emph{Picard groupoid} \(\mathrm{Pic}_{\mathcal B}(A)\) is the groupoid
whose objects are invertible endomorphism \(1\)-cells \(A\to A\) and
whose morphisms are invertible \(2\)-cells between them. Horizontal
composition equips \(\mathrm{Pic}_{\mathcal B}(A)\) with a monoidal
structure whose unit is \(U_A\).
\end{definition}

\begin{definition}\label{def:twisted_1_cells}
Let $\mathcal B$ be a bicategory. Let $\alpha\in \mathrm{Pic}_{\mathcal B}(A)$ and $\beta\in \mathrm{Pic}_{\mathcal B}(B)$ be invertible endomorphism 1-cells. Given a 1-cell $M\colon A\to B$, we define the \emph{left-twisted 1-cell} of $M$ by $\alpha$ to be ${}^\alpha M \coloneqq \alpha^{-1}\odot M$ and the \emph{right-twisted 1-cell} of $M$ by $\beta$ to be $M^\beta \coloneqq M\odot \beta$.
\end{definition}

\begin{remark}\label{rem:base_change_twist}
The twisted 1-cells \({}^\alpha M\) and \(M^\beta\) are obtained by
pre- and postcomposition with invertible 1-cells. For
\(\alpha,\alpha'\in\mathrm{Pic}_{\mathcal B}(A)\), the inverse in the
definition of the left twist is chosen so that
\(
{}^{\alpha\odot\alpha'}M
\cong
{}^{\alpha'}({}^\alpha M).
\)

In a bicategory equipped with base-change 1-cells, an automorphism
\(g\colon A\to A\) determines an invertible base-change 1-cell.
Thus, for automorphisms \(g\colon A\to A\) and \(h\colon B\to B\),
Definition~\ref{def:twisted_1_cells} gives
\(
{}^gM=(A\xrightarrow{g^{-1}}A)\odot M\) and \(
M^h=M\odot(B\xrightarrow{h}B),
\)
where the displayed arrows denote the corresponding base-change
1-cells.
\end{remark}

We record three elementary isomorphisms showing how twists pass through units and composites.

\begin{lemma}\label{lem:twist_properties}
\label{lem:twist_unit}
\label{lem:twist_swap}
\label{lem:twist_cancel}
Let $\mathcal{B}$ be a bicategory. The following canonical isomorphisms hold for twisted 1-cells:
\begin{enumerate}
    \item For any object $A$ and $g\in \mathrm{Pic}_{\mathcal B}(A)$, ${}^{g^{-1}}U_A \xRightarrow{\cong} U_A^g$.
    \item For 1-cells $M\colon A\to B$, $N\colon C\to A$ and $g\in \mathrm{Pic}_{\mathcal B}(A)$, $\iota\colon N\odot {}^gM \xRightarrow{\cong} N^{g^{-1}}\odot M$.
    \item For 1-cells $M\colon A\to B$, $N\colon B\to C$ and $h\in \mathrm{Pic}_{\mathcal B}(B)$, $\mathfrak{m}\colon M^h\odot {}^hN \xRightarrow{\cong} M\odot N$.
\end{enumerate}
\end{lemma}
\begin{proof}
These follow immediately from associativity, the unit axioms, and the invertibility of the twisting 1-cells in $\mathcal{B}$. Specifically: (1) ${}^{g^{-1}}U_A = g\odot U_A \cong g \cong U_A\odot g = U_A^g$; (2) $N\odot {}^gM = N\odot (g^{-1}\odot M) \cong (N\odot g^{-1})\odot M = N^{g^{-1}}\odot M$; and (3) $M^h\odot {}^hN = (M\odot h)\odot (h^{-1}\odot N) \cong M\odot (h\odot h^{-1})\odot N \cong M\odot N$.
\end{proof}

A $G$-twisting datum coherently packages these twists for all $g\in G.$

\begin{definition}\label{def:G_twisting_data}
Let $G$ be a group, regarded as a discrete monoidal category. A \emph{bicategory with $G$-twisting data} consists of a bicategory $\mathcal{B}$, and for each 0-cell $A \in \mathcal{B}$, a strong monoidal functor $\Theta_A \colon G \to \mathrm{Pic}_{\mathcal{B}}(A)$. Explicitly, this consists of:
\begin{enumerate}
\item For each $g \in G$, an invertible 1-cell $\Theta_A(g) \colon A \to A$.
\item For each $g, h \in G$, an invertible 2-cell $\mu^A_{g,h} \colon \Theta_A(g) \odot \Theta_A(h) \Rightarrow \Theta_A(gh)$.
\item An invertible 2-cell $\eta^A \colon U_A \Rightarrow \Theta_A(e)$.
\end{enumerate}
These structure 2-cells satisfy the standard associativity and unit
coherence conditions for a strong monoidal functor.
\end{definition}

\begin{remark}\label{rem:twist_notation}
For a 1-cell \(M\colon A\to B\) and \(g\in G\), we write
\(
{}^gM\coloneqq\Theta_A(g^{-1})\odot M\) and
\(M^g\coloneqq M\odot\Theta_B(g).
\)
Since the monoidal structure of \(\Theta_A\) canonically identifies
\(\Theta_A(g^{-1})\) with an inverse of \(\Theta_A(g)\), the first
expression is canonically isomorphic to the left twist indexed by
\(\Theta_A(g)\) in Definition~\ref{def:twisted_1_cells}.
The monoidal structure maps give canonical isomorphisms
\(
{}^{gh}M\cong{}^h({}^gM)\) and \(
M^{gh}\cong(M^g)^h.
\)
\end{remark}

\begin{example}[Twists in $\mathbf{Mod}(\mathcal V)$]\label{ex:modv_twists}
Suppose that $A$ and $B$ are monoids in $\mathcal V$ equipped with $G$-actions, and let $\Theta_A(g) = A^g$, $\Theta_B(g) = B^g$ be the corresponding base-change bimodules. For an $(A,B)$-bimodule $M$, its abstract left and right twists in this framework are defined as $A^{g^{-1}}\odot_A M$ and $M\odot_B B^g$. These   twists will be explicitly identified with the action-twisted bimodules ${}^gM$ and $M^g$ in Lemma~\ref{lem:base-change-twists}.
\end{example}

The following example shows that the same base-change
mechanism also occurs in a bicategory not of the form
\(\mathbf{Mod}(\mathcal V)\).

\begin{example}[Twisting data for retractive spaces]
\label{ex:retractive-G-twisting-data}
Let \(\mathcal R/G\text{-}\mathcal{T}op\) be the bicategory whose
0-cells are left \(G\)-spaces and whose 1-cells and 2-cells are those
of \(\mathcal R/\mathcal{T}op\) after forgetting the \(G\)-actions.
Thus, a 1-cell \(X\colon A\to B\) is an ordinary retractive space over
\(A\times B\), with no prescribed \(G\)-action.

For a left \(G\)-space \(A\) and \(g\in G\), define
\(
\Theta_A(g)
\coloneqq
(A\xrightarrow{g^{-1}}A)\colon A\to A,
\)
where \(g^{-1}\colon A\to A\) denotes the homeomorphism induced by the
\(G\)-action and the displayed arrow denotes its associated
base-change 1-cell. Explicitly, \(\Theta_A(g)\) is the retractive
space over \(A\times A\)
\[
A\times A
\xrightarrow{i}
(A\times A)\sqcup A
\xrightarrow{p_g}
A\times A,
\]
where \(p_g\) restricts to the identity on \(A\times A\) and to
\((\id_A,g^{-1})\) on the additional copy of \(A\).

Since \(g^{-1}\) is a homeomorphism, \(\Theta_A(g)\) is invertible,
with
\(
\Theta_A(g)^{-1}\cong\Theta_A(g^{-1}).
\)
The standard composition isomorphisms for base-change 1-cells give
\[
\begin{aligned}
\Theta_A(g)\odot\Theta_A(h)
&\cong
(A\xrightarrow{h^{-1}\circ g^{-1}}A)  =
(A\xrightarrow{(gh)^{-1}}A)
=
\Theta_A(gh),
\end{aligned}
\]
together with \(\Theta_A(e)\cong U_A\). These isomorphisms satisfy the
associativity and unit coherence conditions for base-change.
Consequently,
\(
\Theta_A\colon G\to\mathrm{Pic}_{\mathcal R/\mathcal{T}op}(A)
\)
defines \(G\)-twisting data.
\end{example}

\subsection{Bicategories of $G$-twists}

Definition~\ref{def:G_twisting_data} equips 0-cells with $G$-twisting data via the functors $\Theta$. To extend this twisting data to 1-cells, we specify how they interact with these 0-cell twists. We do this by collecting the 1-cells that coherently intertwine the left and right $\Theta$-twists into a new bicategory.

\begin{definition}\label{def:bicat_of_G_twists}
Let $\mathcal{B}$ be a bicategory with $G$-twisting data. The associated bicategory of $G$-twists, denoted by $\mathcal{B}^{G\text{-tw}}$, consists of the following:
\begin{itemize}
    \item \textbf{0-cells:} The same 0-cells as $\mathcal{B}$.

  \item \textbf{Category of 1-cells:} 
  For each pair of 0-cells $A,B$, an object of
$\mathcal B^{G\text{-tw}}(A,B)$ is a pair
\((M,\tau):A\to B,
\)
where $M:A\to B$ is a 1-cell in $\mathcal B$, and
\(
    \tau=\{\tau_{g}\}_{g\in G}
\)
is a family of invertible 2-cells for each $g \in G$, called the \emph{twisting constraints} of the pair \((M,\tau)\) where
\(
\tau_{g}:\Theta_A(g)\odot M \xRightarrow{\cong} M\odot \Theta_B(g)
\)  and \(g\in G\).
We call \((M,\tau)\)  a \emph{$G$-twisted 1-cell} and  write $\tau_g$ for $\tau_{g,M}$ when $M$ is clear. This family
    satisfies the following two coherence conditions:
    \begin{enumerate}
        \item[(a)] \textbf{Unit compatibility.} The composite
        \[ M \xRightarrow{\mathfrak{l}_M^{-1}} U_A \odot M \xRightarrow{\eta^A \odot \id} \Theta_A(e) \odot M \xRightarrow{\tau_e} M \odot \Theta_B(e) \xRightarrow{\id \odot (\eta^B)^{-1}} M \odot U_B \xRightarrow{\mathfrak{r}_M} M \]
        is the identity $2$-cell on $M$.

        \item[(b)] \textbf{Multiplicativity.} For all $g, h \in G$, the following diagram of 2-cells commutes:
        \[\begin{tikzcd}[column sep=5mm, row sep=6mm]
        \Theta_A(gh) \odot M \ar[r, Rightarrow, "\tau_{gh}"] \ar[d, Rightarrow, "(\mu^A_{g,h})^{-1} \odot \id"'] & M \odot \Theta_B(gh) \\
        (\Theta_A(g) \odot \Theta_A(h)) \odot M \ar[d, Rightarrow, "\mathfrak a"'] & M \odot (\Theta_B(g) \odot \Theta_B(h)) \ar[u, Rightarrow, "\id \odot \mu^B_{g,h}"'] \\
        \Theta_A(g) \odot (\Theta_A(h) \odot M) \ar[d, Rightarrow, "\id \odot \tau_h"'] & (M \odot \Theta_B(g)) \odot \Theta_B(h) \ar[u, Rightarrow, "\mathfrak a^{-1}"'] \\
        \Theta_A(g) \odot (M \odot \Theta_B(h)) \ar[r, Rightarrow, "\mathfrak a^{-1}"] & (\Theta_A(g) \odot M) \odot \Theta_B(h) \ar[u, Rightarrow, "\tau_g \odot \id"']
        \end{tikzcd}\]
    \end{enumerate}
A morphism
\(f:(M,\tau)\Longrightarrow (N,\sigma)
\)
in $\mathcal B^{G\text{-tw}}(A,B)$ is a 2-cell
$f:M\Rightarrow N$ in $\mathcal B(A,B)$ such that, for every $g\in G$,
the following square commutes:
\begin{equation}\label{eq:Gtw-2cell-compat}
    \begin{tikzcd}[column sep=5mm, row sep=5mm]
    \Theta_A(g) \odot M \ar[r, Rightarrow, "\tau_g"] \ar[d, Rightarrow, "\id \odot f"'] & M \odot \Theta_B(g) \ar[d, Rightarrow, "f \odot \id"] \\
    \Theta_A(g) \odot N \ar[r, Rightarrow, "\sigma_g"'] & N \odot \Theta_B(g)
    \end{tikzcd}
    \end{equation}
\item \textbf{Composition:} If
\(
    (M,\tau):A\to B,
    (N,\sigma):B\to C
\)
are 1-cells in $\mathcal B^{G\text{-tw}}$, then their composite is
\(    (M,\tau)\odot (N,\sigma)
    :=
    (M\odot N,\tau\star\sigma),\)  where the twisting constraint
\begin{equation}\label{eq:tau_def}
        (\tau\star\sigma)_{g}:
        \Theta_A(g)\odot (M\odot N)
        \xRightarrow{\cong}
        (M\odot N)\odot \Theta_C(g)
\end{equation} is
defined by the canonical composite
\[\begin{tikzcd}[row sep=5mm, column sep=6mm]
    \Theta_A(g)\odot (M\odot N) \arrow[r,Rightarrow,"\mathfrak a^{-1}"] & (\Theta_A(g)\odot M)\odot N \arrow[r,Rightarrow,"\tau_g\odot \id"] & (M\odot \Theta_B(g))\odot N \arrow[d,Rightarrow,"\mathfrak a"] \\
    (M\odot N)\odot \Theta_C(g) & M\odot (N\odot \Theta_C(g)) \arrow[l,Rightarrow,"\mathfrak a^{-1}"] & M\odot (\Theta_B(g)\odot N) \arrow[l,Rightarrow,"\id\odot \sigma_g"]
    \end{tikzcd}\]
  \item \textbf{Identity 1-cells:} For each object $A$, the identity 1-cell in
    $\mathcal B^{G\text{-tw}}(A,A)$ is
    \(
        (U_A,\tau^U_A),
    \)
    where $U_A$ is the unit 1-cell of $\mathcal B$, and the twisting constraint
    $
    \tau^U_g:\Theta_A(g)\odot U_A \xRightarrow{\cong} U_A\odot \Theta_A(g)
    $
    is given by
    \(
    \tau^U_g = \mathfrak l_{\Theta_A(g)}^{-1}\circ \mathfrak r_{\Theta_A(g)}.
    \)
\end{itemize}
\end{definition}

\begin{prop}\label{prop:bicat-of-G-twists}
Definition~\ref{def:bicat_of_G_twists} defines a bicategory
\(\mathcal B^{G\text{-tw}}\).
\end{prop}

\begin{proof}
For each pair of objects \(A,B\), the data specified in
Definition~\ref{def:bicat_of_G_twists} form a category
\(\mathcal B^{G\text{-tw}}(A,B)\). Identity \(2\)-cells clearly satisfy
\eqref{eq:Gtw-2cell-compat}, and the composite of compatible \(2\)-cells is
again compatible by functoriality of composition in \(\mathcal B\).

If $(M,\tau):A\to B$ and $(N,\sigma):B\to C$ are 1-cells in $\mathcal{B}^{G\text{-tw}}$, their composite is the pair $(M\odot N, \tau\star\sigma)$, where the constraint $\tau\star\sigma$ defined in \eqref{eq:tau_def} satisfies the unit compatibility by the unit compatibility of $\tau$ and $\sigma$, together with the triangle axiom in $\mathcal{B}$. Similarly, the
multiplicativity condition for  $\tau\star\sigma$ follows from the
multiplicativity of \(\tau\) and \(\sigma\), using the pentagon axiom and the
naturality of the associator in \(\mathcal B\). The identity \(1\)-cells
\((U_A,\tau^U)\) satisfy the same axioms by coherence of the unitors.

Moreover, let
\(
f:(M,\tau)\Rightarrow (M',\tau')\) and \(
k:(N,\sigma)\Rightarrow (N',\sigma')
\)
be morphisms in \(\mathcal B^{G\text{-tw}}\). Thus the underlying
\(2\)-cells \(f:M\Rightarrow M'\) and \(k:N\Rightarrow N'\) satisfy
\eqref{eq:Gtw-2cell-compat} with respect to the pairs of twisting constraints
\((\tau,\tau')\) and \((\sigma,\sigma')\), respectively. We claim that the
 composite
\(
f\odot k:M\odot N\Rightarrow M'\odot N'
\)
satisfies \eqref{eq:Gtw-2cell-compat} with respect to the composite twisting
constraints \(\tau\star\sigma\) and \(\tau'\star\sigma'\). For each \(g\in G\),
expand \((\tau\star\sigma)_g\) and \((\tau'\star\sigma')_g\) by definition. The
required compatibility square is then the outer boundary of the  diagram
obtained from the compatibility squares for \(f\) and \(k\), together with the
naturality squares for the associator in \(\mathcal B\).

It remains to verify that the associator and unitors of \(\mathcal B\) define
\(2\)-cells in \(\mathcal B^{G\text{-tw}}\). For composable 1-cells
\(
(L,\lambda), (M,\tau), (N,\sigma)
\)
in $\mathcal B^{G\text{-tw}}$, the associator
\(
\mathfrak a_{L,M,N}\colon (L\odot M)\odot N \xRightarrow{\cong} L\odot(M\odot N)
\)
satisfies \eqref{eq:Gtw-2cell-compat} with respect to the induced twisting
constraints on both sides by naturality of \(\mathfrak a\) and the definition of
the composite constraint. The same argument applies to the left and right
unitors.

Since these underlying \(2\)-cells satisfy the pentagon and triangle identities
in \(\mathcal B\), the same identities hold in \(\mathcal B^{G\text{-tw}}\).
Hence Definition~\ref{def:bicat_of_G_twists} defines a bicategory
\(\mathcal B^{G\text{-tw}}\).
\end{proof}

\begin{remark}\label{rem:twist_notation_and_consistency}
Using the notation of Remark~\ref{rem:twist_notation}, the constraint of a
\(G\)-twisted \(1\)-cell \((M,\tau)\) is
\(
\tau_{g}:{}^{g^{-1}}M \xrightarrow{\cong} M^g.
\)
Equivalently, it defines the transport map
\begin{equation}\label{eq:g_M_def}
g_M \coloneqq (\tau_{g^{-1}})^{-1}: M^{g^{-1}} \xrightarrow{\cong} {}^gM.
\end{equation}
We use the constraints \(\tau_{g}\) in Definition~\ref{def:bicat_of_G_twists} because they compose naturally from left to right under composition, whereas the transport maps \(g_M\) run in the opposite direction and involve \(g^{-1}\).
\end{remark}

The following basic properties of the transport maps \(g_M\) are clear from the definitions.

\begin{lemma}\label{lem:gM-coherence}
For any $g\in G$, the following hold in $\mathcal B^{G\text{-tw}}$:
\begin{enumerate}
\item[(1)]
For every morphism
\(
\alpha:(M,\tau)\Rightarrow (N,\sigma)
\)
in \(\mathcal B^{G\text{-tw}}(A,B)\),     the compatibility condition
\eqref{eq:Gtw-2cell-compat} for the element \(g^{-1}\in G\) is equivalent to the commutativity of the following square
\[
\begin{tikzcd}[column sep=6mm, row sep=4mm]
M^{g^{-1}} \arrow[r,Rightarrow,"g_M"] \arrow[d,Rightarrow,"\alpha^{g^{-1}}"']
& {}^gM \arrow[d,Rightarrow,"{}^g\alpha"] \\
N^{g^{-1}} \arrow[r,Rightarrow,"g_N"']
& {}^gN.
\end{tikzcd}
\]

\item[(2)] For composable \(1\)-cells \((M,\tau):A\to B\) and \((N,\sigma):B\to C\), where
\(M\odot N\) is equipped with the composite twisting constraint
\(\tau\star\sigma\), the transport map
\(
g_{M\odot N}:(M\odot N)^{g^{-1}}\xRightarrow{\cong} {}^g(M\odot N)
\)
is the canonical composite
\begin{equation}\label{eq:gMN-formula}
(M\odot N)^{g^{-1}}
\xRightarrow{\;\mathfrak a\;}
M\odot N^{g^{-1}}
\xRightarrow{\;\id_M\odot g_N\;}
M\odot {}^gN
\xRightarrow{\;\iota\;}
M^{g^{-1}}\odot N
\xRightarrow{\;g_M\odot \id_N\;}
{}^gM\odot N
\xRightarrow{\;\mathfrak a\;}
{}^g(M\odot N).
\end{equation}

\item[(3)] For the unit \(1\)-cell \(U_A\), the transport map
\(
g_{U_A}:U_A^{g^{-1}}\xRightarrow{\cong}{}^gU_A
\)
is the inverse of the canonical isomorphism
\(
{}^gU_A\xRightarrow{\cong}U_A^{g^{-1}}
\)
from Lemma~\ref{lem:twist_unit}, applied to \(g^{-1}\).
\end{enumerate}
\end{lemma}

We now relate the twisting constraint \(\tau_{g}\) to a corresponding \emph{bi-twist isomorphism} \(\gamma_g\).

\begin{prop}\label{prop:gM-vs-gammaM}
For each $1$-cell $M:A\to B$ in $\mathcal B^{G\text{-tw}}$ and each $g\in G$, the twisting constraint
\(
\tau_g:\Theta_A(g)\odot M \xRightarrow{\cong} M\odot \Theta_B(g)
\)
canonically determines a bi-twist isomorphism
\begin{equation}\label{eq:gammaM}
\gamma_g: M \xRightarrow{\cong} {}^gM^g,
\end{equation}
where
\(
{}^gM^g \coloneqq \Theta_A(g^{-1})\odot M\odot \Theta_B(g).
\)
Conversely, every such isomorphism \(\gamma_g\) canonically determines
\(\tau_g\). For each \(g\in G\), these two constructions give a
bijective correspondence between twisting constraints and bi-twist
isomorphisms.
\end{prop}

\begin{proof}
Fix \(g\in G\) and a \(1\)-cell \(M:A\to B\) in
\(\mathcal B^{G\text{-tw}}\). Given \(\tau_g\), define \(\gamma_g\)
as the composite
\[
M \xRightarrow{\mathfrak l_M^{-1}} U_A\odot M
\xRightarrow{\eta^A\odot \id}
\Theta_A(e)\odot M
\xRightarrow{(\mu^A_{g^{-1},g})^{-1}\odot \id}
(\Theta_A(g^{-1})\odot \Theta_A(g))\odot M
\]
\[
\xRightarrow{\mathfrak a}
\Theta_A(g^{-1})\odot (\Theta_A(g)\odot M)
\xRightarrow{\id\odot \tau_{g}}
\Theta_A(g^{-1})\odot (M\odot \Theta_B(g))
\xRightarrow{\mathfrak a^{-1}}
(\Theta_A(g^{-1})\odot M)\odot \Theta_B(g).
\]
This defines the required bi-twist isomorphism
\(
\gamma_g:M\xRightarrow{\cong}{}^gM^g
\)
and hence an assignment \(F_g:\tau_g\mapsto\gamma_g\).

Conversely, given \(\gamma_g:M\xRightarrow{\cong}{}^gM^g\), define \(\tau_g\) by
\[
\Theta_A(g)\odot M
\xRightarrow{\id\odot \gamma_g}
\Theta_A(g)\odot\bigl((\Theta_A(g^{-1})\odot M)\odot \Theta_B(g)\bigr)
\xRightarrow{\mathfrak a^{-1}}
(\Theta_A(g)\odot(\Theta_A(g^{-1})\odot M))\odot \Theta_B(g)
\]
\[
\xRightarrow{\mathfrak a^{-1}\odot \id}
((\Theta_A(g)\odot\Theta_A(g^{-1}))\odot M)\odot \Theta_B(g)
\xRightarrow{(\mu^A_{g,g^{-1}}\odot \id)\odot \id}
(\Theta_A(e)\odot M)\odot \Theta_B(g)
\]
\[
\xRightarrow{((\eta^A)^{-1}\odot \id)\odot \id}
(U_A\odot M)\odot \Theta_B(g)
\xRightarrow{\mathfrak l_M\odot \id}
M\odot \Theta_B(g).
\]
This defines the assignment \(G_g:\gamma_g\mapsto\tau_g\).

Substituting either construction into the other, the inserted inverse
twists cancel through the coherence maps \(\mu^A\) and \(\eta^A\).
Hence, by the coherence axioms for \(\Theta_A\) and bicategorical
coherence,
\(
G_g(F_g(\tau_g))=\tau_g\) and \(F_g(G_g(\gamma_g))=\gamma_g.
\)
Thus \(F_g\) and \(G_g\) establish the claimed bijective correspondence.
\end{proof}

\begin{remark}\label{rem:naturality-is-equivariance}
The transport map \(g_M\) moves a single twist across a \(1\)-cell, whereas
\(\gamma_g\) identifies \(M\) with its corresponding bi-twist \({}^gM^g\).
Under the correspondence of Proposition~\ref{prop:gM-vs-gammaM}, the unit and
multiplicativity axioms for the twisting constraints \(\tau_{g}\) translate
into the following conditions on the bi-twist isomorphisms \(\gamma_g\):
\begin{enumerate}
    \item[(1)] \textbf{Unit.} The isomorphism
\(
\gamma_{e,M}:M\xrightarrow{\cong}{}^eM^e
\)
is the canonical isomorphism induced by the unit constraints
$\Theta_A(e)\cong U_A$ and $\Theta_B(e)\cong U_B$.
Equivalently, the composite
\(M\xRightarrow{\gamma_{e}}{}^eM^e\xrightarrow{\cong}M
\)
is the identity of $M$.

    \item[(2)] \textbf{Multiplicativity.} For every \(g,h\in G\), the isomorphism
    \(\gamma_{gh}\colon M\to {}^{gh}M^{gh}\) agrees with the canonical composite
    \(
    M \xRightarrow{\gamma_{h}} {}^hM^h
    \xRightarrow{{}^h\gamma_g^h} {}^h({}^gM^g)^h
    \xRightarrow{\cong} {}^{gh}M^{gh},
    \)
    where the monoidal structure maps induce the final isomorphism
    \(\mu^A,\mu^B\) and the associators in \(\mathcal B\).
\end{enumerate}

Likewise, the compatibility condition on \(2\)-cells with the twisting
constraints \(\tau_{g,-}\) is equivalent to the commutativity of the square
\begin{equation}\label{equiv:constraint}
\begin{tikzcd}
M \arrow[r,Rightarrow, "\alpha"]
  \arrow[d,Rightarrow,"\gamma_{g,M}"']
& N \arrow[d,Rightarrow,"\gamma_{g,N}"] \\
{}^gM^g \arrow[r,Rightarrow,"{}^g\alpha^g"']
& {}^gN^g,
\end{tikzcd}
\end{equation}
where \({}^g\alpha^g\) denotes the \(2\)-cell
\(
\Theta_A(g^{-1})\odot \alpha \odot \Theta_B(g).
\)
\end{remark}

\subsection{Equivalence with semilinear equivariant bimodules}
In this subsection, we specialize the construction of Section~3.2
to the bicategory of monoids and bimodules. We show that the twisting
constraints of Definition~\ref{def:bicat_of_G_twists} are equivalent
to semilinear \(G\)-actions on bimodules. This yields a biequivalence
between the bicategory
\(\mathbf{Mod}(\mathcal V)^{G\text{-tw}}\) constructed below and the
bicategory \(\mathbf{Mod}(\mathcal V)^G\) of \(G\)-monoids and
bimodules equipped with semilinear \(G\)-actions
(Corollary~\ref{cor:equiv-mod-G}).

Throughout this subsection, we let $(\mathcal V,\otimes,I)$ be a cocomplete closed symmetric monoidal category for which the relevant relative tensor products exist, and write  $\mathbf{Mod}(\mathcal V)$  for the bicategory of bimodules over monoids recalled in Example~\ref{ex:ModV}.

 Let $\mathbf{Mod}(\mathcal{V})_{G\text{-mon}}$ be the bicategory whose $0$-cells are pairs $(A, \alpha)$, where $A$ is a monoid in $\mathcal{V}$ and $\alpha\colon G\to \operatorname{Aut}_{\mathrm{Mon}(\mathcal{V})}(A)$ is a group homomorphism. Its \(1\)-cells are \((A,B)\)-bimodules, and its \(2\)-cells are
bimodule morphisms.
We equip $\mathbf{Mod}(\mathcal{V})_{G\text{-mon}}$ with $G$-twisting data by defining   $\Theta_{(A,\alpha)}(g) \coloneqq A^g$, where $A^g$ denotes the $(A,A)$-bimodule whose underlying object is $A$, equipped with the standard left action and the right action twisted by $\alpha_g:= \alpha(g)$.
For each \(g\in G\), the bimodule \(A^g\) is invertible, with inverse
canonically isomorphic to \(A^{g^{-1}}\). The standard base-change
isomorphisms
\(
A^g\odot_A A^h\xRightarrow{\cong}A^{gh},
A\xRightarrow{\cong}A^e,
\)
satisfy the associativity and unit coherence conditions and therefore
make
\[
\Theta_{(A,\alpha)}
\colon
G\longrightarrow \operatorname{Pic}_{\mathbf{Mod}(\mathcal V)_{G\text{-mon}}}
\bigl((A,\alpha)\bigr)
\]
a strong monoidal functor.

Applying Definition~\ref{def:bicat_of_G_twists}, set
\(
\mathbf{Mod}(\mathcal V)^{G\text{-tw}}
\coloneqq
\bigl(\mathbf{Mod}(\mathcal V)_{G\text{-mon}}\bigr)^{G\text{-tw}}.
\)
Thus a 1-cell from \((A,\alpha)\) to \((B,\beta)\) is an \((A,B)\)-bimodule \(M\) equipped with twisting constraints
\(
\tau_g\colon A^g\odot_A M \xRightarrow{\cong} M\odot_B B^g
\)
satisfying the unit and multiplicativity conditions of Definition~\ref{def:bicat_of_G_twists}. Its 2-cells are the compatible bimodule morphisms.

For comparison, let \(\mathbf{Mod}(\mathcal V)^G\) be the bicategory with the same objects, whose 1-cells are bimodules equipped with semilinear \(G\)-actions. Explicitly, a $1$-cell $(A,\alpha)\to (B,\beta)$ is a pair $(M,\phi)$, where $M$ is an $(A,B)$-bimodule and $\phi=\{\phi_g\}_{g\in G}$ is a family of automorphisms $\phi_g\colon M\xrightarrow{\cong} M$ in $\mathcal V$ such that \(\phi_e=\id_M,\)
\(\phi_{gh}=\phi_g\circ\phi_h,\) and each $\phi_g$ is $(\alpha_g,\beta_g)$-semilinear, i.e., the diagrams
\begin{equation}\label{eq:semilinearity-modV}
\begin{tikzcd}
A\otimes M \ar[r,"\lambda"] \ar[d,"\alpha_g\otimes \phi_g"'] & M \ar[d,"\phi_g"] \\
A\otimes M \ar[r,"\lambda"'] & M
\end{tikzcd}
\qquad
\begin{tikzcd}
M\otimes B \ar[r,"\rho"] \ar[d,"\phi_g\otimes \beta_g"'] & M \ar[d,"\phi_g"] \\
M\otimes B \ar[r,"\rho"'] & M
\end{tikzcd}
\end{equation}
commute. A $2$-cell $f\colon (M,\phi)\Rightarrow (N,\psi)$ is a bimodule map $f\colon M\to N$ such that $f\circ\phi_g=\psi_g\circ f$ for all $g\in G$. For composable $1$-cells $(M,\phi)$ and $(N,\psi)$, their composite is $M\odot_B N$, equipped with the automorphism $(\phi\odot_B\psi)_g$ induced by $\phi_g\otimes \psi_g$. The unit $1$-cell on $(A,\alpha)$ is the regular bimodule $A$, equipped with the family $\alpha_g\colon A\to A$. The associators and unitors are inherited from \(\mathbf{Mod}(\mathcal V)\).

By Proposition~\ref{prop:gM-vs-gammaM}, the constraints \(\tau_g\) may equivalently be encoded by bi-twist isomorphisms \(\gamma_{g}\). We first identify the left and right twists in \(\mathbf{Mod}(\mathcal V)\) with the usual action-twisted bimodules, and then identify the resulting bi-twist isomorphisms with semilinear automorphisms.

\begin{lemma}\label{lem:base-change-twists}
Let \((A,\alpha)\) and \((B,\beta)\) be \(G\)-monoids in
\(\mathcal V\), and let \(M\) be an \((A,B)\)-bimodule. For every
\(g\in G\), there are canonical isomorphisms of
\((A,B)\)-bimodules
\(
A^{g^{-1}}\odot_A M\xrightarrow{\cong}{}^gM
\) and
\(M\odot_B B^g\xrightarrow{\cong}M^g,
\)
where \({}^gM\) denotes \(M\) with left action
\(
\lambda_{{}^gM}
=
\lambda_M\circ(\alpha_g\otimes\id_M)
\)
and unchanged right \(B\)-action, while \(M^g\) denotes \(M\) with
unchanged left \(A\)-action and right action
\(
\rho_{M^g}
=
\rho_M\circ(\id_M\otimes\beta_g).
\)
\end{lemma}

\begin{proof}
We first construct the left isomorphism. Recall that the underlying object of \(A^{g^{-1}}\) is \(A\), equipped with its usual left \(A\)-action and the right \(A\)-action \(\mu_A\circ(\id_A\otimes\alpha_{g^{-1}})\). Because \(\alpha_g\) is a monoid automorphism and \(\alpha_g\circ\alpha_{g^{-1}}=\id_A\), the underlying morphism \(\alpha_g\colon A^{g^{-1}}\to A\) is an isomorphism of right \(A\)-modules. 

Applying \(-\odot_AM\), followed by the canonical left unitor, yields an isomorphism in \(\mathcal V\)
\(
F_g: A^{g^{-1}}\odot_AM \xrightarrow{\alpha_g\odot_A\id_M} A\odot_AM \xrightarrow{\cong} M.
\)
The compatibility of \(F_g\) with the left \(A\)-actions is immediate from the equality \(F_g\circ\lambda_{A^{g^{-1}}\odot_AM} = \lambda_M\circ(\alpha_g\otimes F_g)\). Since the right \(B\)-action is induced   from the unchanged right \(B\)-action on \(M\), \(F_g\) defines the desired isomorphism of \((A,B)\)-bimodules \(A^{g^{-1}}\odot_AM\xrightarrow{\cong}{}^gM\).

The right isomorphism
\(M\odot_B B^g\xrightarrow{\cong}M^g\)
is constructed analogously from the left \(B\)-module isomorphism
\(\id_B\colon B^g\to B\) and the canonical right unitor. The resulting
map is compatible with the unchanged left \(A\)-action and the twisted
right \(B\)-action on \(M^g\).
\end{proof}

\begin{lemma}\label{lem:gamma-vs-semilinear}
Fix \(g\in G\). Under the canonical identification of the underlying
object of \({}^gM^g\) with \(M\), giving an isomorphism of
\((A,B)\)-bimodules
\(
\gamma_g\colon M\xrightarrow{\cong}{}^gM^g
\)
is equivalent to giving an
\((\alpha_g,\beta_g)\)-semilinear automorphism
\(
\phi_g\colon M\xrightarrow{\cong}M
\)
in \(\mathcal V\).
\end{lemma}

\begin{proof}
Transport \(\gamma_g\) along the canonical identification of the
underlying object of \({}^gM^g\) with \(M\), and denote the resulting
morphism in \(\mathcal V\) by \(f\colon M\to M\).
Under this identification, \(\gamma_g\) is a morphism of left
\(A\)-modules if and only if
\(
f\circ\lambda_M
=
\lambda_M\circ(\alpha_g\otimes f).
\)
This is exactly the left semilinearity condition in
\eqref{eq:semilinearity-modV}. Similarly, \(\gamma_g\) is a morphism of
right \(B\)-modules if and only if
\(
f\circ\rho_M
=
\rho_M\circ(f\otimes\beta_g).
\)
Hence \(f\) defines a bimodule map
\(M\to{}^gM^g\) if and only if it is
\((\alpha_g,\beta_g)\)-semilinear.
 Taking
\(\phi_g=f\) gives the desired correspondence.
\end{proof}

Combining Lemma~\ref{lem:gamma-vs-semilinear} with Remark~\ref{rem:naturality-is-equivariance} allows us to identify the  twisting data with a coherent semilinear $G$-action.

\begin{prop}\label{prop:comparison-vs-semilinear}
Let $(A,\alpha)$ and $(B,\beta)$ be objects of $\mathbf{Mod}(\mathcal V)^G$, and let $M$ be an $(A,B)$-bimodule.
The following data are equivalent:
\begin{enumerate}
\item\label{it:cmp-gamma}
for each $g\in G$, an isomorphism
$\gamma_g\colon M\xrightarrow{\cong} {}^gM^g$
such that the unit and multiplicativity conditions of
Remark~\ref{rem:naturality-is-equivariance} are satisfied;

\item\label{it:semi-gamma}
a family of automorphisms
$\phi_g\colon M\xrightarrow{\cong} M$
in $\mathcal V$ such that
$\phi_e=\id_M$, $\phi_{gh}=\phi_g\circ\phi_h$,
and each $\phi_g$ is $(\alpha_g,\beta_g)$-semilinear in the sense of \eqref{eq:semilinearity-modV}.
\end{enumerate}
Moreover, for composable \(G\)-twisted bimodules
\(M\colon A\to B\) and \(N\colon B\to C\), the bi-twist isomorphism
induced on \(M\odot_BN\) by the composite twisting constraint
corresponds to the semilinear automorphism
\(
\phi_g^M\odot_B\phi_g^N
\colon
M\odot_BN\to M\odot_BN.
\)
\end{prop}

\begin{proof}
By Lemma~\ref{lem:gamma-vs-semilinear}, for each $g\in G$, the datum of
$\gamma_g$ is equivalent to the datum of an
$(\alpha_g,\beta_g)$-semilinear automorphism $\phi_g$ of $M$.
Under the correspondence of Proposition~\ref{prop:gM-vs-gammaM} and
Remark~\ref{rem:naturality-is-equivariance}, the unit axiom for
$\gamma$ becomes $\phi_e=\id_M$, and the multiplicativity axiom becomes
$\phi_{gh}=\phi_g\phi_h$. This proves the equivalence of
\eqref{it:cmp-gamma} and \eqref{it:semi-gamma}.
Finally, unwinding the composite twisting constraint and using
Lemma~\ref{lem:gamma-vs-semilinear}, its associated map on the
underlying object of \(M\odot_BN\) is
\(\phi_g^M\odot_B\phi_g^N\).
\end{proof}

\begin{cor}\label{cor:equiv-mod-G}
There is a biequivalence of bicategories
\(
\mathbf{Mod}(\mathcal V)^G
\simeq
\mathbf{Mod}(\mathcal V)^{G\text{-tw}}.
\)
\end{cor}

\begin{proof}
The correspondence of
Proposition~\ref{prop:comparison-vs-semilinear} defines an assignment
\(
F\colon
\mathbf{Mod}(\mathcal V)^G
\to
\mathbf{Mod}(\mathcal V)^{G\text{-tw}}
\)
which is the identity on objects, underlying bimodules, and underlying
bimodule morphisms. Under the identification of
Lemma~\ref{lem:gamma-vs-semilinear}, the compatibility square of
Remark~\ref{rem:naturality-is-equivariance} becomes precisely
\(
f\circ\phi_g^M=\phi_g^N\circ f.
\)
Thus \(F\) induces an equivalence on every hom-category.

The final assertion of
Proposition~\ref{prop:comparison-vs-semilinear} gives compatibility
with horizontal composition. Moreover, the canonical twisting
constraints on the regular bimodule \(A\) correspond to the
semilinear action \(\{\alpha_g\}_{g\in G}\), so the unit \(1\)-cells
are preserved. Hence \(F\) defines a pseudofunctor. Since it is the
identity on objects and induces an equivalence on every hom-category,
it is a biequivalence.
\end{proof}

\subsection{Examples}
\label{subsec:G-twist-examples}
We now describe the bicategories of \(G\)-twists in several concrete
settings.

\begin{example}\label{ex:semilinear-G-twists}
The biequivalence of Corollary~\ref{cor:equiv-mod-G} has the following
concrete descriptions for rings, spectra, and based bispaces.
\begin{enumerate}
\item\label{case:ring-bimod}
For \(\mathcal V=\mathbf{Ab}\), Corollary~\ref{cor:equiv-mod-G}
specializes to
\(
\mathbf{Bimod}/\mathcal R\mathrm{ing}^G
\simeq
\mathbf{Bimod}/\mathcal R\mathrm{ing}^{G\text{-tw}}.
\) Concretely, a \(1\)-cell from \((A,\alpha)\) to \((B,\beta)\) is an \((A,B)\)-bimodule \(M\) equipped with a semilinear \(G\)-action satisfying \(\phi_g(amb) = \alpha_g(a)\phi_g(m)\beta_g(b)\). The corresponding twisting constraint is \(\tau_{g}\colon A^g\otimes_A M \xrightarrow{\cong} M\otimes_B B^g\), given by \(1\otimes m\mapsto \phi_g(m)\otimes 1\).

\item\label{case:borel-twisted-bimodules}
For \(\mathcal V=\mathbf{Sp}\),
Corollary~\ref{cor:equiv-mod-G} identifies  twisting
constraints on bimodule spectra with semilinear \(G\)-actions.
Thus, if \(R\) and \(S\) are ring spectra with \(G\)-actions and
\(M\) is an \((R,S)\)-bimodule spectrum with a semilinear
\(G\)-action, passage to the homotopy bicategory gives a
\(G\)-twisted \(1\)-cell in
\(
\bigl(\Ho\mathbf{Bimod}/\mathcal R_{\mathbf{Sp}}\bigr)^{G\text{-tw}}
\)
with twisting constraints
\(
\tau_{g,M}\colon
R^g\odot M
\xRightarrow{\cong}
M\odot S^g.
\)
These are \(G\)-actions on ordinary spectra, hence define naive
\(G\)-spectra rather than genuine \(G\)-spectra.
\item\label{case:twisted-bispaces}
For \(\mathcal V=\mathcal T\mathrm{op}_*\), restrict
\(\mathbf{Mod}(\mathcal T\mathrm{op}_*)^{G\text{-tw}}\) to the
\(0\)-cells \((H_+,\alpha^H)\) arising from topological groups \(H\)
equipped with \(G\)-actions by automorphisms. A \(1\)-cell from
\((K_+,\alpha^K)\) to \((H_+,\alpha^H)\) is a based
\((K,H)\)-bispace \(X\) equipped with twisting constraints
\(
\tau_{g}\colon
K^g\odot X
\xrightarrow{\cong}
X\odot H^g,\) where \(g\in G.
\)
By Corollary~\ref{cor:equiv-mod-G}, such constraints are equivalently
a based \(G\)-action \(\{\phi_g\}_{g\in G}\) on \(X\) satisfying
\(
\phi_g(k\cdot x\cdot h)
=
\alpha_g^K(k)\cdot\phi_g(x)\cdot\alpha_g^H(h).
\)
\end{enumerate}
\end{example}

In the preceding example, twisting constraints encode additional
equivariant structure on a \(1\)-cell. In the following two examples,
the ambient equivariance itself provides canonical twisting
constraints.

\begin{example}
\label{ex:HoBimodSpCn}
By \cite[Example~4.1.1(1)]{AGHKK},
\(\mathbf{Sp}^{C_n}\) satisfies the hypotheses of
Remark~\ref{rem:HoModV}. Hence
\(
\Ho\mathbf{Mod}(\mathbf{Sp}^{C_n})
=
\Ho\mathbf{Bimod}/\mathbf{Sp}^{C_n}
\)
is defined.

Since \(C_n\) is abelian, the action of each \(g\in C_n\) on a
genuine \(C_n\)-spectrum is a morphism in
\(\mathbf{Sp}^{C_n}\). Thus, for \(C_n\)-ring spectra \(R,S\) and an
equivariant \((R,S)\)-bimodule spectrum \(M\), the action of \(g\)
defines an
\((\alpha_g^R,\alpha_g^S)\)-semilinear automorphism of \(M\).
Applying Corollary~\ref{cor:equiv-mod-G} at the point-set level and
then passing to the homotopy bicategory yields a family of canonical
twisting constraints
\(
\tau_{g,M}^{\mathrm{can}}\colon
R^g\odot M
\xRightarrow{\cong}
M\odot S^g,\) for \(g\in C_n,
\)
which makes \(M\) a canonical \(C_n\)-twisted \(1\)-cell.
\end{example}

\begin{example}\label{ex:green-functors}
Let \(G\) be a finite abelian group, and let
\(
\mathcal V
=
(\mathbf{Mack}_G,\square,\underline A)
\)
be the symmetric monoidal category of \(G\)-Mackey functors.  A monoid
in \(\mathcal V\) is a \(G\)-Green functor.  For \(g\in G\) and
\(H\leq G\), left multiplication gives a \(G\)-equivariant orbit
automorphism
\(
\ell_g^H\colon
G/H\rightarrow G/H,
xH\mapsto gxH.
\)
For a \(G\)-Green functor \(\underline R\), these orbit
automorphisms induce automorphisms
\(
\alpha_g\colon
\underline R\xrightarrow{\cong}\underline R
\)
compatible with restrictions, transfers, multiplication, and the
unit.  Likewise, they induce on every
\((\underline R,\underline S)\)-bimodule Mackey functor
\(\underline M\) an
\((\alpha_g^{\underline R},\alpha_g^{\underline S})\)-semilinear
automorphism
\(
\phi_g^{\underline M}\colon
\underline M\xrightarrow{\cong}\underline M.
\)
Corollary~\ref{cor:equiv-mod-G} thus gives canonical twisting
constraints
\(
\tau_{g}^{\mathrm{can}}\colon
\underline R^g\square_{\underline R}\underline M
\xRightarrow{\cong}
\underline M\square_{\underline S}\underline S^g.
\)

By \cite[Theorem~4.3 and Proposition~4.4]{BGHL},
\(s\mathbf{Mack}_G\) has the model-categorical structure required in
Remark~\ref{rem:HoModV}; in the cyclic case, see also
\cite[Example~4.1.1(2)]{AGHKK}. Hence
\(\Ho\mathbf{Mod}(s\mathbf{Mack}_G)\) is defined.

The automorphisms above extend levelwise to simplicial \(G\)-Green
functors and their bimodules, since they commute with the simplicial
face and degeneracy maps. For a simplicial
\((\underline R_\bullet,\underline S_\bullet)\)-bimodule
\(\underline M_\bullet\), they therefore induce canonical twisting
constraints
\(
\tau_{g,\underline M_\bullet}^{\mathrm{can}}\colon
\Theta_{\underline R_\bullet}(g)
\square_{\underline R_\bullet}^{\mathbb L}
\underline M_\bullet
\xRightarrow{\cong}
\underline M_\bullet
\square_{\underline S_\bullet}^{\mathbb L}
\Theta_{\underline S_\bullet}(g),\) for \(g\in G.
\)
Thus \(\underline M_\bullet\) determines a canonical \(G\)-twisted
\(1\)-cell in
\(\bigl(\Ho\mathbf{Mod}(s\mathbf{Mack}_G)\bigr)^{G\text{-tw}}.
\)
\end{example}

\section{Twisted shadows on bicategories of $G$-twists}
In this section, we construct \(g\)-twisted shadows on bicategories of
\(G\)-twists, establish their properties, and identify
twisted Hochschild and topological Hochschild constructions as examples.

\subsection{The induced $g$-twisted shadow}
\begin{definition}\label{def:g-twisted-shadow}
Let $(\mathcal B,\langle\!\langle-\rangle\!\rangle,\mathcal T)$ be a $\mathcal T$-shadowed bicategory equipped with $G$-twisting data, and let $\mathcal B^{G\text{-tw}}$ be the associated bicategory of $G$-twists. Fix $g\in G$. For each object $R\in \mathrm{Ob}(\mathcal B)$, the \emph{induced $g$-twisted $\mathcal T$-shadow functor} is the functor
\(
{}^g\!\langle\!\langle-\rangle\!\rangle_R:\mathcal B^{G\text{-tw}}(R,R)\to \mathcal T
\)
defined on a $1$-cell $(M,\tau):R\to R$ by
\(
{}^g\!\langle\!\langle (M,\tau)\rangle\!\rangle_R
\coloneqq
\langle\!\langle {}^g M\rangle\!\rangle_R,
\)
and on a $2$-cell $f:(M,\tau)\Rightarrow (N,\sigma)$ by
\(
{}^g\!\langle\!\langle f\rangle\!\rangle_R
\coloneqq
\langle\!\langle {}^g f\rangle\!\rangle_R.
\)
\end{definition}

\begin{prop}\label{prop:g_twisted_axioms}
The induced functor ${}^g\!\langle\!\langle-\rangle\!\rangle$ admits natural isomorphisms
\[
{}^g\theta_{M,N}:
{}^g\!\langle\!\langle M\odot N\rangle\!\rangle_R
\xrightarrow{\cong}
{}^g\!\langle\!\langle N\odot M\rangle\!\rangle_S
\]
for \(1\)-cells
\(
(M,\tau):R\to S, (N,\sigma):S\to R
\)
in \(\mathcal B^{G\text{-tw}}\), called the \emph{twisted cyclic isomorphisms}, satisfying the following twisted hexagon and twisted triangle axioms:

\begin{itemize}
\item For \(1\)-cells \((M,\tau):R\to S\), \((N,\sigma):S\to T\), and \((P,\omega):T\to R\), the \emph{twisted hexagon} diagram commutes:
\[
\begin{tikzcd}[column sep=5mm, row sep=6mm]
{}^g\!\langle\!\langle (M \odot N)\odot P\rangle\!\rangle_R
  \arrow[r,"{}^g\theta"]
  \arrow[d,"{}^g\!\langle\!\langle\mathfrak a\rangle\!\rangle"']
&
{}^g\!\langle\!\langle P\odot (M\odot N)\rangle\!\rangle_T
  \arrow[r,"{}^g\!\langle\!\langle\mathfrak a\rangle\!\rangle"]
&
{}^g\!\langle\!\langle (P\odot M)\odot N\rangle\!\rangle_T
  \arrow[d,"{}^g\theta"]
\\
{}^g\!\langle\!\langle M\odot (N\odot P)\rangle\!\rangle_R
  \arrow[r,"{}^g\theta"']
&
{}^g\!\langle\!\langle (N\odot P)\odot M\rangle\!\rangle_S
  \arrow[r,"{}^g\!\langle\!\langle\mathfrak a\rangle\!\rangle"']
&
{}^g\!\langle\!\langle N\odot (P\odot M)\rangle\!\rangle_S
\end{tikzcd}
\]

\item For any \(1\)-cell \((M,\tau):R\to R\), the \emph{twisted triangle} diagram commutes:
\[
\begin{tikzcd}[column sep=3mm, row sep=5mm]
{}^g\!\langle\!\langle M\odot U_R\rangle\!\rangle_R
  \arrow[rr,"{}^g\theta"]
  \arrow[dr,"{}^g\!\langle\!\langle\mathfrak r\rangle\!\rangle"']
&&
{}^g\!\langle\!\langle U_R\odot M\rangle\!\rangle_R
  \arrow[dl,"{}^g\!\langle\!\langle\mathfrak l\rangle\!\rangle"]
\\
&
{}^g\!\langle\!\langle M\rangle\!\rangle_R
&
\end{tikzcd}
\]
\end{itemize}
\end{prop}

\begin{proof}
 For \(1\)-cells \((M,\tau):R\to S\) and \((N,\sigma):S\to R\),   the twisted cyclicity isomorphism ${}^g\theta_{M,N}$ is defined as the composite
\begin{equation}\label{def:twisted-cyclicity}
{}^g\sh{M\odot N}_R \cong \sh{{}^gM \odot N}_R \xRightarrow{\theta} \sh{N \odot {}^gM}_S \xRightarrow{\sh{\iota_{M,N}}} \sh{N^{g^{-1}} \odot M}_S \xRightarrow{\sh{g_N \odot \id_M}} \sh{{}^gN \odot M}_S \cong {}^g\sh{N \odot M}_S.
\end{equation}

For the twisted hexagon, let
\((M,\tau):R\to S\), \((N,\sigma):S\to T\), and
\((P,\omega):T\to R\). Consider the composites along the upper-right and down-right paths.
Expanding \({}^g\theta\) along either route
yields an alternating sequence of ordinary cyclicity maps
\(\theta\) and associators \(\mathfrak a\), interleaved with the
\(2\)-cells \(\iota\) and \(g_{(-)}\).   
Crucially, because the ordinary shadow $\theta$ and the associator $\mathfrak{a}$ are natural isomorphisms, they commute with $\iota$ and $g_{(-)}$. This naturality allows us to reorder the composition: we can evaluate the sequence of $\theta$ and $\mathfrak{a}$ maps first, followed by the sequence of $\iota$ and $g_{(-)}$ maps. 

After this reordering, the underlying sequence of $\theta$ and $\mathfrak{a}$ evaluates exactly to the ordinary shadow hexagon applied to the triple $({}^gM, N, P)$, which commutes by the axioms of the shadowed bicategory. It then remains only to compare the sequence of $\iota$ and $g_{(-)}$ maps. The upper-right path applies $g_P$ and then $g_N$, whereas the down-right path applies the combined map $g_{N \odot P}$. By Lemma~\ref{lem:gM-coherence}(2), the transport map
\(
g_{N\odot P}:(N\odot P)^{g^{-1}}\xRightarrow{\cong}{}^g(N\odot P)
\)
is the canonical composite obtained from \(g_P\), then \(\iota\), then \(g_N\),
together with the evident associators. This is exactly the transport composite
appearing along the upper-right route.  

For the twisted triangle, apply Lemma~\ref{lem:gM-coherence}(3) to identify
\(g_{U_R}\) with the canonical unit isomorphism. The two routes then reduce,
via the ordinary shadow triangle and the naturality of \(\theta\), to the same
map
\(
{}^g\sh{M\odot U_R}_R \longrightarrow {}^g\sh{M}_R.
\)
Hence the twisted triangle commutes.
\end{proof}

\begin{remark}\label{rem:BPW-comparison}
Unlike the ordinary shadow hexagon, the right-hand vertical arrow in
the twisted hexagon above points downward because cyclic rotation
also transports the \(g\)-twist. Up to our left-to-right convention
for horizontal composition, the twisted hexagon and triangle are the
same as the axioms of \cite[Definition~3.1]{BPW} with
\(\Sigma=\operatorname{Id}_{\mathcal B^{G\text{-tw}}}\).
In \cite{BPW}, the twist is specified by an endobifunctor
\(\Sigma\). Here, the \(g\)-twist instead comes from the internal
twisting \(1\)-cells \(\Theta_A(g)\) and the twisting constraints
\(\tau\).
\end{remark}

We now study the structural features of \(g\)-twisted shadows that arise from 
the \(G\)-twisting data. Conjugation transports \(g\)-twisted shadows to conjugate twists and induces centralizer actions.
In the cyclic case, this equivariant structure controls the twofold composite 
of the twisted cyclicity isomorphism and consequently yields the corresponding 
\(2n\)-periodicity statement.

\begin{prop}\label{prop:conj-transport-twisted-shadows}
Let \(\mathcal B\) be a \(\mathcal T\)-shadowed bicategory equipped
with \(G\)-twisting data. For every object \(R\) and every
\(g,h\in G\), there is a natural isomorphism
\[
\mathbf C_{h,g}\colon
{}^g\sh{-}_R
\xRightarrow{\cong}
{}^{hgh^{-1}}\sh{-}_R.
\]
These conjugation-transport isomorphisms are compatible with the
twisted cyclicity maps of
Proposition~\ref{prop:g_twisted_axioms} and satisfy
\(
\mathbf C_{e,g}=\id\) and \(
\mathbf C_{k,hgh^{-1}}\circ\mathbf C_{h,g}
=
\mathbf C_{kh,g}
\)
for all \(g,h,k\in G\).
\end{prop}
\begin{proof}For \(M\in\mathcal B^{G\text{-tw}}(R,R)\), define the component
\(\mathbf C_{h,g}(M)\) by the following composite:
$$\begin{aligned}
{}^g\sh{M}_R &= \sh{\Theta_R(g^{-1})\odot M}_R \xRightarrow{\cong} \sh{\Theta_R(g^{-1})\odot M\odot \Theta_R(h^{-1})\odot \Theta_R(h)}_R \\
&\xRightarrow{\theta}\sh{\Theta_R(h)\odot \Theta_R(g^{-1})\odot M\odot \Theta_R(h^{-1})}_R \xRightarrow{\;\sh{\id\odot\id\odot h_M}\;} \sh{\Theta_R(h)\odot \Theta_R(g^{-1})\odot \Theta_R(h^{-1})\odot M}_R \\
&\xRightarrow{\;\sh{(\mu^R_{hg^{-1},\,h^{-1}}\circ(\mu^R_{h,g^{-1}}\odot \id))\odot \id_M}\;} \sh{\Theta_R(hg^{-1}h^{-1})\odot M}_R = {}^{hgh^{-1}}\sh{M}_R.
\end{aligned}$$
Here \(h_M\colon M^{h^{-1}}\xRightarrow{\cong}{}^hM\) is the twist transport map from
Remark~\ref{rem:twist_notation_and_consistency} and
Lemma~\ref{lem:gM-coherence}.

Naturality in \(M\) follows from the naturality of the transport map \(h_M\),
together with functoriality of the shadow \(\sh{-}_R\) and the naturality of the
ordinary cyclicity isomorphism \(\theta\).

For \(h=e\), the transport map is
\(
e_M=(\tau_{e,M})^{-1}\colon
M\odot\Theta_R(e)\to\Theta_R(e)\odot M.
\)
By the unit compatibility of the twisting constraints, this is the
canonical composite
\[
M\odot\Theta_R(e)
\xRightarrow{\id\odot(\eta^R)^{-1}}
M\odot U_R
\xRightarrow{\mathfrak r_M}
M
\xRightarrow{\mathfrak l_M^{-1}}
U_R\odot M
\xRightarrow{\eta^R\odot\id}
\Theta_R(e)\odot M.
\]
Consequently, after identifying \(\Theta_R(e)\) with \(U_R\), the
cyclic rotation involving this factor reduces, by the ordinary shadow
triangle axiom, to the identity. The remaining multiplication maps
in the definition of \(\mathbf C_{e,g}\) reduce to the appropriate
unitors by the unit coherence of the strong monoidal functor
\(\Theta_R\). Hence
\(\mathbf C_{e,g}=\id\).

To verify the composition law
\(
\mathbf C_{k,hgh^{-1}}\circ\mathbf C_{h,g}
=
\mathbf C_{kh,g},
\)
expand both sides using the definition of the conjugation-transport
maps. The ordinary shadow hexagon identifies the two successive
cyclic rotations on the left-hand side with the single cyclic
rotation occurring on the right-hand side. The multiplicativity
axiom for the twisting constraints identifies the successive
transport maps associated to \(h\) and \(k\) with the transport map
associated to \(kh\). Finally, the associativity coherence of the
strong monoidal functor \(\Theta_R\) identifies the two ways of
combining the inserted twisting \(1\)-cells. Hence the two composites
agree.

Compatibility with twisted cyclicity means that, for
\(M:R\to S\) and \(N:S\to R\), the square
\[
\begin{tikzcd}[column sep=4mm,row sep=5mm]
{}^g\sh{M\odot N}_R
  \arrow[r,"{}^g\theta"]
  \arrow[d,"\mathbf C_{h,g}(M\odot N)"']
&
{}^g\sh{N\odot M}_S
  \arrow[d,"\mathbf C_{h,g}(N\odot M)"]
\\
{}^{hgh^{-1}}\sh{M\odot N}_R
  \arrow[r,"{}^{hgh^{-1}}\theta"']
&
{}^{hgh^{-1}}\sh{N\odot M}_S
\end{tikzcd}
\]
commutes. After expanding the four maps, this follows from the
ordinary shadow hexagon, the naturality of ordinary cyclicity, and
the multiplicativity of the twisting constraints.
\end{proof}

 \begin{cor}\label{cor:centralizer-action}
Fix \(g\in G\), and let \(C_G(g)\) denote the centralizer of \(g\).
For each object \(R\), the assignment
\[
h\in C_G(g)\ \longmapsto\ \mathbf C_{h,g}(-)\;:\;
{}^g\sh{(-)}_R \Rightarrow {}^g\sh{(-)}_R
\]
defines a left action of \(C_G(g)\) on the \(g\)-twisted shadow
\({}^g\sh{(-)}_R\) by natural automorphisms.
Equivalently, the family \(\{\mathbf C_{h,g}\}_{h\in C_G(g)}\) determines a group homomorphism
\(
C_G(g)\to \mathrm{Aut}\bigl({}^g\sh{(-)}_R\bigr),
\)
and this action is compatible with the twisted cyclicity isomorphisms
\({}^g\theta\) of Proposition~\ref{prop:g_twisted_axioms}.
\end{cor}

 \begin{cor}\label{cor:abelian-G-action}
If \(G\) is abelian, then for every \(g,h\in G\) the maps
\(
\mathbf C_{h,g}\colon {}^g\sh{(-)}_R \Rightarrow {}^g\sh{(-)}_R
\)
define a \(G\)-action on the \(g\)-twisted shadow.
Hence \({}^g\sh{(-)}_R\) canonically lifts to a functor
\(
{}^g\sh{(-)}_R\colon \mathcal B^{G\text{-tw}}(R,R)\to \mathcal T^G,\) where
\(\mathcal T^G:=\operatorname{Fun}(BG,\mathcal T)\).
\end{cor}

The next proposition identifies the twofold composite of twisted
cyclicity with the action of \(g\) given by the conjugation-transport
map \(\mathbf C_{g,g}\).

\begin{prop}\label{prop:twisted_2n_return}
Let \(g\in G\). For \(1\)-cells
\(M\colon R\to S\) and \(N\colon S\to R\) in
\(\mathcal B^{G\text{-tw}}\), the twofold twisted cyclicity operator
\(
\Psi
\coloneqq
{}^g\theta_{N,M}\circ{}^g\theta_{M,N}
\colon
{}^g\sh{M\odot N}_R
\xrightarrow{\cong}
{}^g\sh{M\odot N}_R
\)
coincides with
\(\mathbf C_{g,g}(M\odot N)\).

In particular, if \(g\) has finite order \(n\), then the \(2n\)-fold
composite of twisted cyclicity is the identity:
\(
\bigl({}^g\theta_{N,M}\circ{}^g\theta_{M,N}\bigr)^n
=
\id_{{}^g\sh{M\odot N}_R}.
\)
\end{prop}

\begin{proof}
Write
\(
T_R=\Theta_R(g^{-1})
\)
and
\(
T_S=\Theta_S(g^{-1}).
\)
Expanding both twisted cyclicity maps and suppressing associators,
\(\Psi\) is the composite
\begin{align*}
\sh{T_R\odot M\odot N}_R
&\xrightarrow{\theta}
\sh{N\odot T_R\odot M}_S
\xrightarrow{\sh{g_N\odot\id_M}}
\sh{T_S\odot N\odot M}_S \\
&\xrightarrow{\theta}
\sh{M\odot T_S\odot N}_R
\xrightarrow{\sh{g_M\odot\id_N}}
\sh{T_R\odot M\odot N}_R.
\end{align*}
By the ordinary shadow hexagon and the naturality of \(\theta\), the
two cyclicity maps combine to the ordinary cyclicity isomorphism
\(
\sh{T_R\odot(M\odot N)}_R
\xrightarrow{\theta}
\sh{(M\odot N)\odot T_R}_R.
\)
By Lemma~\ref{lem:gM-coherence}(2), the remaining transport maps combine
to
\(
g_{M\odot N}\colon
(M\odot N)\odot T_R
\xrightarrow{\cong}
T_R\odot(M\odot N).
\)
Thus \(\Psi\) is the composite
\[
\sh{T_R\odot(M\odot N)}_R
\xrightarrow{\theta}
\sh{(M\odot N)\odot T_R}_R
\xrightarrow{\sh{g_{M\odot N}}}
\sh{T_R\odot(M\odot N)}_R,
\]
which is precisely
\(\mathbf C_{g,g}(M\odot N)\) after applying the canonical unit and
multiplication identifications in the definition of
\(\mathbf C_{g,g}\).

If \(g\) has order \(n\), then
\(\langle g\rangle\subseteq C_G(g)\). By
Corollary~\ref{cor:centralizer-action},
\(
\Psi^n
=
\mathbf C_{g,g}(M\odot N)^n
=
\mathbf C_{g^n,g}(M\odot N)
=
\mathbf C_{e,g}(M\odot N)
=
\id.
\)
\end{proof}

\subsection{\(G\)-Morita invariance}

Morita invariance is a fundamental property of shadows. In the present
setting, the appropriate notion is Morita equivalence internal to the
bicategory \(\mathcal B^{G\text{-tw}}\). We begin by comparing
dualizability in \(\mathcal B^{G\text{-tw}}\) with dualizability in the
underlying bicategory \(\mathcal B\). We then apply this comparison to
recognize \(G\)-Morita equivalences and prove that the induced
\(g\)-twisted shadow is invariant under \(G\)-Morita equivalence.

We first record how twisting interacts with duality in the underlying bicategory.

\begin{lemma}\label{lem:dual-of-twist}
Let \(\mathcal B\) be a bicategory equipped with \(G\)-twisting data, and let \(M:R\to S\) be a left dualizable \(1\)-cell in \(\mathcal B\), with chosen left dual \({}^*M:S\to R\). Then, for every \(g\in G\), the left-twisted \(1\)-cell 
\(
{}^gM=\Theta_R(g^{-1})\odot M
\)
is left dualizable in \(\mathcal B\), and 
\(
({}^*M)^g = {}^*M\odot\Theta_R(g)
\)
is canonically a left dual of \({}^gM\). Consequently, there is a canonical isomorphism in \(\mathcal B(S,R)\)
\(
{}^*({}^gM)\xrightarrow{\cong}({}^*M)^g.
\)
\end{lemma}
\begin{proof}
Put
\(u=\Theta_R(g^{-1})
\) and \(v=\Theta_R(g).\)
The strong monoidal structure of \(\Theta_R\) gives \(v\) the structure
of a left dual of \(u\). 
Explicitly, the coevaluation $\bar{\eta}_u$ and evaluation $\bar{\epsilon}_u$ are defined by
\begin{align*}
\bar{\eta}_u&\colon
U_R
\xrightarrow{\eta^R}
\Theta_R(e)
\xrightarrow{(\mu^R_{g,g^{-1}})^{-1}}
\Theta_R(g)\odot\Theta_R(g^{-1})
=
v\odot u, \\
\bar{\epsilon}_u&\colon
u\odot v
=
\Theta_R(g^{-1})\odot\Theta_R(g)
\xrightarrow{\mu^R_{g^{-1},g}}
\Theta_R(e)
\xrightarrow{(\eta^R)^{-1}}
U_R.\end{align*}
The triangle identities follow from the coherence axioms for the
strong monoidal functor \(\Theta_R\).

Let \({}^*M\) be the chosen left dual of \(M\), with coevaluation
\(\bar{\eta}_M\) and evaluation \(\bar{\epsilon}_M\). The composite
\(
{}^*M\odot v
=
{}^*M\odot\Theta_R(g)
=
({}^*M)^g
\)
is a left dual of
\(
{}^gM=u\odot M.
\)
Indeed, its coevaluation is
\[
U_S
\xrightarrow{\bar{\eta}_M}
{}^*M\odot M
\cong
{}^*M\odot U_R\odot M
\xrightarrow{
  \id_{{}^*M}\odot\bar{\eta}_u\odot\id_M
}
{}^*M\odot v\odot u\odot M,
\]
and its evaluation is
\(
u\odot M\odot{}^*M\odot v
\xrightarrow{
  \id_u\odot\bar{\epsilon}_M\odot\id_v
}
u\odot U_R\odot v
\cong
u\odot v
\xrightarrow{\bar{\epsilon}_u}
U_R.
\)
The triangle identities follow from those for
\((M,{}^*M)\) and \((u,v)\). Thus \(({}^*M)^g\) is a left dual of
\({}^gM\). By uniqueness of left duals, there is a canonical
isomorphism
\(
{}^*({}^gM)\xrightarrow{\cong}({}^*M)^g
\)
compatible with the duality data.
\end{proof}

Dualizability can be lifted to \(\mathcal B^{G\text{-tw}}\).

\begin{prop}\label{prop:Gtw-dualizable-recognition}
Let
\(
U:\mathcal B^{G\text{-tw}}\to\mathcal B
\)
be the forgetful pseudofunctor.
A \(G\)-twisted \(1\)-cell
\((M,\tau):R\to S\)
is left dualizable in
\(\mathcal B^{G\text{-tw}}\)
if and only if its underlying \(1\)-cell \(M\) is left dualizable
in \(\mathcal B\).

More precisely, if \({}^*M:S\to R\) is a chosen left dual of \(M\),
then the twisting constraints on \(M\) canonically induce twisting
constraints on \({}^*M\).
\end{prop}

\begin{proof}
The forward implication follows immediately by applying the forgetful pseudofunctor \(U\).

Conversely, suppose that the underlying \(1\)-cell \(M\) is left dualizable in \(\mathcal B\), with chosen left dual \(N={}^*M\). Let \(\gamma_{g,M}\colon M\xrightarrow{\cong}{}^gM^g\) be the bi-twist isomorphism corresponding to the twisting constraint \(\tau_{g,M}\). By Lemma~\ref{lem:dual-of-twist} and the standard reverse-order formula for left duals of composites, the bi-twist \({}^gM^g = \Theta_R(g^{-1})\odot M\odot\Theta_S(g)\) is left dualizable, with left dual canonically isomorphic to \({}^*({}^gM^g) \cong \Theta_S(g^{-1})\odot N\odot\Theta_R(g) = {}^gN^g\). 

Let \({}^*\gamma_{g,M}\colon {}^*({}^gM^g)\xrightarrow{\cong}{}^*M=N\) denote the left mate of \(\gamma_{g,M}\). Using the canonical identification above, define \(\gamma_{g,N}\colon N \xrightarrow{({}^*\gamma_{g,M})^{-1}} {}^*({}^gM^g) \xrightarrow{\cong} {}^gN^g\).
The mate correspondence is compatible with identities, vertical
composition, and horizontal composition. Hence the unit and
multiplicativity conditions for \(\gamma_{g,M}\) transpose to the
corresponding conditions for \(\gamma_{g,N}\). Moreover, the defining
mate equations say precisely that the coevaluation and evaluation
maps for \((M,N)\) are compatible with the induced twisting
constraints. Thus they define \(2\)-cells in
\(\mathcal B^{G\text{-tw}}\). Their triangle identities are the
underlying triangle identities in \(\mathcal B\), so
\((M,N)\) is a dual pair in \(\mathcal B^{G\text{-tw}}\).
\end{proof}

\begin{cor}\label{cor:algebraic_dualizable_recognition}
Let \(M\colon R\to S\) be a 1-cell in \(\mathbf{Mod}(\mathcal V)^G\). Then \(M\) is left dualizable if and only if its underlying bimodule is left dualizable in \(\mathbf{Mod}(\mathcal V)\). 
\end{cor}
\begin{proof}
The forward implication follows by forgetting the \(G\)-equivariant
structure. Conversely, suppose that the underlying bimodule \(M\) is
left dualizable in \(\mathbf{Mod}(\mathcal V)\), with canonical left
dual
\(
{}^*M=\Hom_R(M,R).
\)
By Proposition~\ref{prop:Gtw-dualizable-recognition}, the equivariant
structure on \(M\) canonically induces an equivariant structure on
\({}^*M\), and the evaluation and coevaluation maps are equivariant.

Under the equivalence of
Corollary~\ref{cor:equiv-mod-G}, the induced structure on \({}^*M\) is
the conjugate \(G\)-action
\(
\phi^{\,{}^*M}_g
=
\Hom_R(\phi^M_{g^{-1}},\alpha_g):
\Hom_R(M,R)\to\Hom_R(M,R).
\)
Thus \({}^*M\) is a left dual of \(M\) in
\(\mathbf{Mod}(\mathcal V)^G\).
\end{proof}

\begin{example}\label{ex:dualizable_equivariant_module}
Consider the bicategory of $G$-twists $\mathbf{Bimod}/\mathcal{R}\mathrm{ing}^G$ for an abelian group $G$ (as in Example~\ref{ex:semilinear-G-twists}
\textup{(\ref{case:ring-bimod})}). An equivariant 1-cell $(M,\phi) \colon (R,\alpha) \to (S,\beta)$ is left dualizable if and only if its underlying module $M$ is a finitely generated projective left $R$-module. In this case, Corollary~\ref{cor:algebraic_dualizable_recognition} guarantees that the canonical left dual ${}^*M \coloneqq \Hom_R(M,R)$ lifts to a left dual in $\mathbf{Bimod}/\mathcal{R}\mathrm{ing}^G$ when equipped with the conjugate $G$-action: $(g\cdot f)(m) \coloneqq \alpha_g\!\Big(f\big(\phi^M_{g^{-1}}(m)\big)\Big), $ for $ g\in G,\ f\in {}^*M,\ m\in M.$

The semilinear \(G\)-action is additional structure on the underlying
\((R,S)\)-bimodule \(M\). Horizontal composition in
\(\mathbf{Bimod}/\mathcal{R}\mathrm{ing}^G\) is still the relative
tensor product over \(S\), namely \(M\otimes_S N\). Therefore the
dualizability condition concerns the underlying left \(R\)-module
structure on \(M\).
\end{example}

\begin{example}\label{ex:dualizable_green_functor}
Let $\mathbf{Bimod}/\mathbf{Green}^{G}$ be the bicategory of $G$-twists for a finite abelian group $G$, as in Example~\ref{ex:green-functors}. By Proposition~\ref{prop:dualizable-bimodule}, an equivariant $1$-cell $(\underline M,\phi^{\underline M})\in \mathbf{Bimod}/\mathbf{Green}^{G}(\underline R,\underline S)$ is left dualizable if and only if its underlying left \(\underline R\)-module is finitely generated
projective. In this case, Corollary~\ref{cor:algebraic_dualizable_recognition} guarantees that the canonical left dual ${}^*\underline M \coloneqq \underline{\mathrm{Hom}}_{\underline R}(\underline M,\underline R)$ lifts to a valid left dual in $\mathbf{Bimod}/\mathbf{Green}^{G}$ when equipped with the conjugate $G$-equivariant structure. Explicitly, evaluated at a finite $G$-set $X$, an element $f \in {}^*\underline M(X) \cong \mathbf{Mod}_{\underline R}(\underline{M} \square \underline{A}_X, \underline R)$ is mapped to: $\big(\phi_g^{{}^*\underline M}\big)_X(f) = \alpha_g \circ f \circ (\phi^{\underline M}_{g^{-1}} \square \id_{\underline{A}_X}).$
\end{example}

  A dual pair whose evaluation and coevaluation $2$-cells are isomorphisms is called a Morita equivalence (Definition~\ref{def:morita_equivalence}) between the corresponding $0$-cells. We formalize this for bicategories of $G$-twists.  

\begin{definition}\label{def:G_morita_equivalence}
Let \(\mathcal{B}^{G\text{-tw}}\) be a bicategory of \(G\)-twists. A \emph{\(G\)-Morita equivalence} between \(R\) and \(S\) is a Morita equivalence internal to \(\mathcal{B}^{G\text{-tw}}\); that is, a pair of \(1\)-cells \(M\in\mathcal{B}^{G\text{-tw}}(R,S)\) and \(N\in\mathcal{B}^{G\text{-tw}}(S,R)\), together with invertible 2-cells \(\bar\eta\colon U_S \xrightarrow{\cong} N\odot M\) and \(\bar\epsilon\colon M\odot N \xrightarrow{\cong} U_R\) in \(\mathcal{B}^{G\text{-tw}}\) satisfying the triangle identities to exhibit \((N,M)\) as a dual pair. In this case, we say \(R\) and \(S\) are \emph{\(G\)-Morita equivalent}, written \(R\simeq_G S\).
\end{definition}

The preceding recognition result for dualizability extends directly to Morita equivalences in any bicategory of \(G\)-twists.

\begin{cor}\label{cor:general_morita_recognition}
Let \(\mathcal{B}^{G\text{-tw}}\) be a bicategory of \(G\)-twists, and let \((M,\tau) \in \mathcal{B}^{G\text{-tw}}(R,S)\) be a \(G\)-twisted \(1\)-cell. Then \((M,\tau)\) provides a \(G\)-Morita equivalence between \(R\) and \(S\) if and only if its underlying \(1\)-cell \(M\) provides a Morita equivalence in the underlying bicategory \(\mathcal{B}\).
\end{cor}

\begin{proof}
By Definition~\ref{def:G_morita_equivalence}, a \(G\)-Morita equivalence is precisely a dual pair whose evaluation and coevaluation maps are isomorphisms. By Proposition~\ref{prop:Gtw-dualizable-recognition}, the entire dual pair structure (including the \(2\)-cells) lifts from \(\mathcal{B}\) to \(\mathcal{B}^{G\text{-tw}}\). Furthermore, a \(2\)-cell in \(\mathcal{B}^{G\text{-tw}}\) is an isomorphism if and only if its underlying \(2\)-cell is an isomorphism in \(\mathcal{B}\). The result follows immediately.
\end{proof}

We illustrate this general recognition principle through biproducts in closed symmetric monoidal categories.

\begin{example}\label{ex:general_morita_biproduct}
Suppose that the closed symmetric monoidal category \(\mathcal V\) has finite biproducts \(\oplus\). Let \(k\geq 1\), let \(R\) be a \(G\)-monoid, and set \(P\coloneqq R^{\oplus k}\) and \(S\coloneqq \End_R(P)^{\mathrm{op}}\). The diagonal \(G\)-action on \(P\) induces the conjugation \(G\)-action on \(S\), making \(S\) a \(G\)-monoid in \(\mathcal V\). In the underlying bicategory \(\mathbf{Mod}(\mathcal V)\), the bimodules \(P\colon R\to S\) and \(Q\coloneqq\Hom_R(P,R)\colon S\to R\) form a Morita equivalence with canonical isomorphisms \(P\odot_S Q\cong R\) and \(Q\odot_R P\cong S\). By Corollary~\ref{cor:general_morita_recognition}, the \(G\)-twisted structure on \(P\) canonically induces one on \(Q\), yielding a \(G\)-Morita equivalence \(R \simeq_G S\).
\end{example}

This general framework specializes naturally across different settings:

\begin{enumerate}
\item  For a \(C_n\)-Green functor \(\underline{R}\), taking biproducts of Mackey functors yields a \(C_n\)-Morita equivalence between \(\underline{R}\) and \(\underline{\End}_{\underline{R}}(\underline{R}^{\oplus k})^{\mathrm{op}}\) in \(\mathbf{Bimod}/\mathbf{Green}^{C_n}\).
	
\item  Let \(G=C_n\), and let \(R\) be a ring spectrum with a naive \(C_n\)-action. The finite free \(R\)-module \(P\coloneqq R^{\vee k}\) carries the diagonal action, and its endomorphism ring spectrum \(S\coloneqq \End_R(P)^{\mathrm{op}}\) inherits the conjugation action. By Corollary~\ref{cor:general_morita_recognition}, this defines a \(C_n\)-Morita equivalence between \(R\) and \(S\) in the twisted homotopy bicategory \((\Ho\mathbf{Bimod}/\mathcal{R}_{\mathbf{Sp}})^{C_n\text{-tw}}\), modeling Borel \(C_n\)-equivariant ring spectra and semilinear bimodule spectra (see Example~\ref{ex:semilinear-G-twists}
\textup{(\ref{case:borel-twisted-bimodules})}).
\end{enumerate}

We now prove that \(g\)-twisted shadows are  $G$-Morita invariant.

\begin{prop}\label{prop:morita-invariance}
Let
\((\mathcal B,\sh{-},\mathcal T)\)
be a \(\mathcal T\)-shadowed bicategory equipped with
\(G\)-twisting data, and  
\({}^g\sh{-}\)
denote the induced \(g\)-twisted shadow on
\(\mathcal B^{G\text{-tw}}\).
If \(R\) and \(S\) are \(G\)-Morita equivalent, then
\(
{}^g\sh{U_R}\cong{}^g\sh{U_S}.
\)
\end{prop}

\begin{proof}
By \(G\)-Morita equivalence, there is a dual pair \((N,M)\) with \(N\in\mathcal{B}^{G\text{-tw}}(S,R)\) and \(M\in\mathcal{B}^{G\text{-tw}}(R,S)\), equipped with invertible 2-cells \(\bar\eta\colon U_S \xrightarrow{\cong} N\odot M\) and \(\bar\epsilon\colon M\odot N \xrightarrow{\cong} U_R\). Applying the twisted shadow functor \({}^g\sh{-}\) and the twisted cyclic isomorphism yields the composite isomorphism
\[
{}^g\sh{U_R} \xrightarrow{{}^g\sh{\bar\epsilon^{-1}}} {}^g\sh{M \odot N} \xrightarrow{{}^g\theta} {}^g\sh{N \odot M} \xrightarrow{{}^g\sh{\bar\eta^{-1}}} {}^g\sh{U_S}
.\]
Therefore, \({}^g\sh{U_R}\cong {}^g\sh{U_S}\).
\end{proof}

\subsection{Twisted THH and twisted Hochschild constructions}
Having established the structural theory of the 
$g$-twisted shadow, we now show that the motivating example, twisted 
THH, fits precisely into this framework.

\begin{thm}\label{thm:twisted-thh-shadow}
The assignment
\(
(R,M)\mapsto\THH_{C_n}(R;M)
\)
defines a \(g\)-twisted
\(\Ho\mathbf{Sp}^{C_n}\)-valued shadow on
\(\Ho\mathbf{Bimod}/\mathbf{Sp}^{C_n}\).
\end{thm}

\begin{proof}
By Remark~\ref{rem:twisted-bar-equivalence}, there is a natural
isomorphism
\(
\THH_{C_n}(R;M)
\cong
\THH(R;{}^gM).
\) The right-hand side is obtained by applying the ordinary
\(\THH\)-shadow to the left-twisted \(1\)-cell
\({}^gM=\Theta_R(g^{-1})\odot_R M\). Hence
\(\THH_{C_n}(R;-)\) is naturally isomorphic to the induced
\(g\)-twisted shadow of Definition~\ref{def:g-twisted-shadow}.
\end{proof}

The same construction applies to monoids in a symmetric monoidal model category.  

\begin{definition}\label{def:twisted-HH-construction}
Let \(\mathcal V\) be a symmetric monoidal simplicial model category
satisfying the hypotheses of Remark~\ref{rem:HoModV}. Let \(g\in G\),
and let \(A\) be an object of
\(\Ho\mathbf{Mod}(\mathcal V)^{G\text{-tw}}\). The \emph{\(g\)-twisted Hochschild construction}
is the functor
\[
\HH_{\mathcal V}^g(A;-):
\Ho\mathbf{Mod}(\mathcal V)^{G\text{-tw}}(A,A)
\longrightarrow
\Ho\mathcal V
\]
defined on a \(1\)-cell \((M,\tau)\) by
\(
\HH_{\mathcal V}^g(A;(M,\tau))
:=
\left|B_\bullet^{\mathrm{cyc}}(A;{}^gM)\right|,
\)
where \(|-|\) denotes the geometric realization in the Hochschild construction
recalled in Remark~\ref{rem:HH-construction}.
On a morphism
\(
f:(M,\tau)\Rightarrow (N,\sigma)
\)
in \(\Ho\mathbf{Mod}(\mathcal V)^{G\text{-tw}}(A,A)\), we set
\(
\HH_{\mathcal V}^g(A;f)
:=
\left|B_\bullet^{\mathrm{cyc}}(A;{}^gf)\right|,
\)
where \({}^gf=\Theta_A(g^{-1})\odot_A f\).
When the twisting constraint is clear from context, we write
\(\HH_{\mathcal V}^g(A;M)\).
\end{definition}

\begin{remark}
Definition~\ref{def:twisted-HH-construction} agrees with the definition of
\emph{twisted Hochschild homology} given in \cite[Def.~4.1.3]{AGHKK} when the automorphism
$\varphi$ arises from the action of an element $g \in G$.
Similar categorical trace formalisms also appear in algebraic geometry:
for example, the categorical approach developed in \cite{Hemo} defines the
$\phi$-twisted Hochschild homology
$\HH(\mathcal A, {}^\phi\mathcal A)$ of a monoidal $\infty$-category,
where $\phi$  is a monoidal endofunctor; Frobenius provides an
important example in applications.
\end{remark}

We can now show that, in general, $\HH_{\mathcal{V}}^g$ is a $g$-twisted shadow on $\Ho\mathbf{Mod}(\mathcal{V})^{G\text{-tw}}$. The twisted Hochschild construction explicitly realizes the abstract $g$-twisted shadow of Definition~\ref{def:g-twisted-shadow} applied to the ordinary Hochschild shadow.

\begin{thm}\label{thm:twisted-HH-shadow}
The assignment
\(
(A,M)\mapsto\HH_{\mathcal V}^g(A;M)
\)
defines a \(g\)-twisted \(\Ho\mathcal V\)-valued shadow on the bicategory of $G$-twists 
$\Ho\mathbf{Mod}(\mathcal V)^{G\text{-tw}}$.
\end{thm}

\begin{proof}
By Definition~\ref{def:twisted-HH-construction}, evaluating the geometric realization of the ordinary Hochschild bar construction on the twisted $1$-cell ${}^gM$ yields $\HH_{\mathcal{V}}^g(A; M)$. Since the ordinary Hochschild construction defines a valid shadow on the underlying bicategory $\Ho\mathbf{Mod}(\mathcal{V})$ (Remark~\ref{rem:HH-construction}), 
applying Proposition~\ref{prop:g_twisted_axioms} canonically equips
\(\HH_{\mathcal V}^g\) with the twisted cyclicity isomorphisms, ensuring it satisfies both the twisted hexagon and twisted triangle axioms.
\end{proof}

Taking
\(
\mathcal V=s\mathbf{Mack}_{C_n}
\)
in Theorem~\ref{thm:twisted-HH-shadow}, for \(g\in C_n\) and
\(i\geq 0\) we set
\(
\underline{\HH}^{g}_i
:=
H_i\circ\HH^g_{s\mathbf{Mack}_{C_n}},
\)
where
\(
H_i\colon
\Ho(s\mathbf{Mack}_{C_n})
\to
\mathbf{Mack}_{C_n}
\)
is the \(i\)-th homology functor.

When \(g=e^{2\pi i/n}\) is the chosen generator, we write
\(
\underline{\HH}^{C_n}_i
:=
\underline{\HH}^{g}_i.
\)
On Green functors and bimodule Mackey functors regarded as constant
simplicial objects, this agrees with the \(C_n\)-twisted Hochschild
homology of Definition~\ref{def:twisted-HH-Green}; see
\cite[Definition~4.3.3]{AGHKK}.

\begin{cor}\label{cor:twisted-HH-Green}
For every \(g\in C_n\) and \(i\geq 0\), the functor
\(\underline{\HH}^{g}_i\) defines a \(g\)-twisted
\(\mathbf{Mack}_{C_n}\)-valued shadow on
\(\Ho\mathbf{Mod}(s\mathbf{Mack}_{C_n})^{C_n\text{-tw}}.
\)
In particular, for the chosen generator \(g=e^{2\pi i/n}\), this
gives the \(g\)-twisted shadow
\(\underline{\HH}^{C_n}_i\).
\end{cor}

We next give a concrete description of the degree-zero construction.

\begin{definition}\label{def:twisted-HH0}
Let \(\mathcal V\) be a symmetric monoidal category admitting the
relevant coequalizers.
Let \((R,\alpha)\) be a \(G\)-monoid in \(\mathcal V\), and let \(N\) be an \((R,R)\)-bimodule representing an endomorphism \(1\)-cell in \(\mathbf{Mod}(\mathcal V)^{G\text{-tw}}\). 

By Lemma~\ref{lem:base-change-twists}, the left twist \({}^gN=R^{g^{-1}}\odot_RN\) is identified with the bimodule whose underlying object is \(N\), whose right \(R\)-action is unchanged, and whose left \(R\)-action is
\(
{}^g\lambda = \lambda_N\circ(\alpha_g\otimes\id_N).
\)

The \emph{zeroth \(g\)-twisted Hochschild homology} of \(R\) with coefficients in \(N\), denoted by \(\HH_0^g(R;N)\), is defined as the coequalizer
\begin{equation}\label{eq:twisted-HH0-qg}
R\otimes N
\mathrel{
  \begin{matrix}
    \xrightarrow{\mathmakebox[\widthof{$\scriptstyle \rho\circ \mathfrak{s}_{R,N}$}]{\scriptstyle {}^g\lambda}} \\[-1.5ex]
    \xrightarrow[\scriptstyle \rho\circ \mathfrak{s}_{R,N}]{} 
  \end{matrix}
}
N
\xrightarrow{q^{(g)}}
\HH_0^g(R;N).
\end{equation}
In particular, for \(N=R\), we simply write \(\HH_0^g(R) \coloneqq \HH_0^g(R; R)\).
\end{definition}

\begin{example}\label{ex:twisted_HH0_concrete}
Twisted zeroth Hochschild homology admits the following concrete descriptions.

(1) Consider the bicategory of $G$-twists \( \mathbf{Bimod}/\mathcal{R}\mathrm{ing}^{G} \).
Let \( R \) be a ring equipped with a $G$-action by automorphisms (a $G$-ring), 
and let \( M \) be an \( (R,R) \)-bimodule in this bicategory. 
Then the associated \( g \)-twisted zeroth Hochschild homology is
\(
\HH_0^g(R;M)
  = M\,\big/\,
    \bigl\langle\, g(r)\cdot m - m\cdot r \;\bigm|\; r\in R,\ m\in M \,\bigr\rangle,
\)
the quotient of the underlying abelian group of 
$M$ by the subgroup generated by these differences.
Furthermore, if $G$ is abelian, Corollary~\ref{cor:abelian-G-action} implies that this quotient inherits a well-defined $G$-action induced from $M$, canonically lifting $\HH_0^g(R;M)$ to a $G$-module.

(2) Consider the bicategory with $C_n$-twists $\mathbf{Bimod}/\mathbf{Green}^{C_n}.$ Let $\underline{R}$ be a $C_n$-Green functor 
and $\underline{M}$ an $(\underline{R},\underline{R})$-bimodule Mackey functor. 
Then the $g$-twisted zeroth Hochschild homology is the $C_n$-Mackey functor $\underline{\HH}_0^g(\underline{R}; \underline{M})$, which is given by the coequalizer diagram \eqref{eq:twisted-HH0-qg} evaluated in the category of Mackey functors, replacing the tensor product $\otimes$ with the box product $\square$.
\end{example}

\section{Twisted bicategorical trace}
In the ordinary (untwisted) setting, bicategorical shadows give rise to bicategorical traces that recover classical traces. For instance, the bicategorical trace of an endomorphism of a dualizable $1$-cell in $\mathbf{Bimod}/\mathcal R\mathrm{ing}$ recovers the Hattori--Stallings trace. After applying algebraic $K$-theory, the bicategorical trace construction
recovers the degree zero Dennis trace map
\(
K_0(R) \to \HH_0(R).
\)

We begin by defining the \(g\)-twisted bicategorical trace in a shadowed bicategory of \(G\)-twists. We then introduce the \(g\)-twisted Hattori--Stallings trace in the algebraic settings of \(G\)-rings and \(C_n\)-Green functors, and prove that, in \(\mathbf{Bimod}/\mathcal R\mathrm{ing}^G\) and \(\mathbf{Bimod}/\mathbf{Green}^{C_n}\), the \(g\)-twisted bicategorical trace recovers a \(g\)-twisted Hattori--Stallings trace. Finally, we use this identification to obtain a categorical realization of the degree \(0\) twisted equivariant Dennis trace.

 \subsection{The $g$-twisted bicategorical traces and $g$-twisted Hattori--Stallings traces}

 We now define the trace associated to a \(g\)-twisted shadow on
\(\mathcal B^{G\text{-tw}}\). Given a left dualizable \(1\)-cell in
\(\mathcal B^{G\text{-tw}}\), the \(g\)-twisted trace is defined by the same trace composite as in the untwisted case, with the cyclicity
replaced by the twisted cyclic isomorphism of the \(g\)-twisted shadow.

\begin{definition}\label{def:twisted-bicat-trace}
Let $(\mathcal{B}, \sh{-}, \mathcal{T})$ be a $\mathcal{T}$-shadowed bicategory equipped with $G$-twisting data, and let $\mathcal{B}^{G\text{-tw}}$ be an associated bicategory of $G$-twists. Fix \(g\in G\). Let
\(M\in \mathcal B^{G\text{-tw}}(R,S)\) be left dualizable with left dual
\(N\in \mathcal B^{G\text{-tw}}(S,R)\).

The $g$-twisted bicategorical trace of a 2-cell 
$f : M \odot Q \Rightarrow P \odot M$ (where $Q: S \to S$ and $P: R \to R$), denoted $\operatorname{tr}^g(f)$, is the composite:
\[
\begin{aligned}
{}^g\sh{Q}_S 
&\cong {}^g\sh{U_S\odot Q}_S 
\xrightarrow{{}^g\sh{\bar{\eta}\odot \id_Q}_S} 
{}^g\sh{N\odot M\odot Q}_S 
\xrightarrow{{}^g\sh{\id_N\odot f}_S} 
{}^g\sh{N\odot P\odot M}_S \\
&\qquad \xrightarrow{{}^g\theta} 
{}^g\sh{M\odot N\odot P}_R 
\xrightarrow{{}^g\sh{\bar{\epsilon}\odot \id_P}_R} 
{}^g\sh{U_R\odot P}_R 
\cong {}^g\sh{P}_R .
\end{aligned}
\]
where $\bar{\eta}: U_S \Rightarrow N \odot M$ and $\bar{\epsilon}: M \odot N \Rightarrow U_R$ are the coevaluation and evaluation maps for $(M,N)$, and ${}^g\theta$ is the twisted cyclic isomorphism.

Similarly, for a 2-cell $h : Q \odot N \Rightarrow N \odot P$ (where $Q: S \to S$ and $P: R \to R$), the trace $\operatorname{tr}^g(h)$ is the composite:
\[
\begin{aligned}
{}^g\sh{Q}_S
&\cong {}^g\sh{Q\odot U_S}_S
\xrightarrow{{}^g\sh{\id_Q\odot \bar{\eta}}_S}
{}^g\sh{Q\odot N\odot M}_S
\xrightarrow{{}^g\sh{h\odot \id_M}_S}
{}^g\sh{N\odot P\odot M}_S \\
&\xrightarrow{{}^g\theta}
{}^g\sh{P\odot M\odot N}_R
\xrightarrow{{}^g\sh{\id_P\odot \bar{\epsilon}}_R}
{}^g\sh{P\odot U_R}_R
\cong {}^g\sh{P}_R .
\end{aligned}
\]
\end{definition}

In the untwisted setting, the bicategorical trace of the identity 
2-cell recovers the classical Hattori--Stallings trace (Example~\ref{ex:HS-as-bicat-trace}). We now define a
$g$-twisted analogue for an arbitrary
$R$-linear map from a finitely generated
projective module to its 
$g$-twist.

Let $R$ be a ring equipped with a $G$-action, and let $g \in G$. For any finitely generated projective left $R$-module $P$, the tensor--Hom adjunction together with dualizability yields a canonical isomorphism
$
\delta_{P,{}^{g}P}: {}^*P\otimes_R {}^{g}P \xrightarrow{\cong} \Hom_R(P,{}^{g}P).
$

\begin{definition}\label{def:twisted_trace}
Let \(P\) be a finitely generated projective left \(R\)-module. The \emph{\(g\)-twisted Hattori--Stallings trace} on \(P\) is the map
\(
\operatorname{tr}_{\mathrm{HS}}^g \colon \operatorname{Hom}_R(P, {}^gP) \to \mathrm{HH}_0^g(R)
\)
defined by sending any \(R\)-linear map \(\psi \colon P \to {}^gP\) to
\(
\operatorname{tr}_{\mathrm{HS}}^g(\psi) := \pi^{(g)}(\delta_{P, {}^gP}^{-1}(\psi)),
\)
where \(\pi^{(g)} \colon {}^\ast P \otimes_R {}^gP \to \mathrm{HH}_0^g(R)\) is the canonical map given by \(\pi^{(g)}(\xi \otimes p) = [\xi(p)]\).
\end{definition}

\begin{remark}
The map $\pi^{(g)}$ is well-defined: for $\xi \in {}^*P$, $p \in P$, and $r \in R$, we have the following sequence of equalities:
\(
\pi^{(g)}(\xi\cdot r\otimes p)
=
[\xi(p)r]
=
[g(r)\xi(p)]
=
[\xi(g(r)p)]
=
\pi^{(g)}(\xi\otimes r\cdot_{{}^gP}p).
\)
The first equality follows from the definition of the right $R$-action on the dual module ${}^*P$. The second equality holds in $\HH_0^g(R)$ due to the defining relation. The third equality uses the \(R\)-linearity of \(\xi\), and
the last equality uses the definition of the left \(R\)-action on
\({}^gP\). Hence \(\pi^{(g)}\) respects the balancing relation and descends
to a well-defined map on \({}^*P\otimes_R{}^gP\).
\end{remark}

 \begin{remark}
The map \(\operatorname{tr}_{\mathrm{HS}}^g \colon \operatorname{Hom}_R(P, {}^gP) \to \mathrm{HH}_0^g(R)\) constructed in Definition \ref{def:twisted_trace} is analogous to the twisted trace introduced by \cite[Section 2.4.6]{BPW}, who define a map $\operatorname{tr}^\varphi \colon \mathcal{R}\text{ep}(A) \to A/[A,A]_\varphi$ for an algebra endomorphism $\varphi$. In our setting, the twist is induced by the group action $g \in G$, and the target space \(\mathrm{HH}_0^g(R)\) serves as the corresponding space of twisted coinvariants.
\end{remark}

To establish the cyclicity property of the twisted Hattori--Stallings trace, we must frequently manipulate the canonical tensor--Hom isomorphism $\delta$ under both pre- and post-composition. We record these fundamental naturality properties in the following lemma.

\begin{lemma}\label{lem:delta-naturality}
Let $P$ and $Q$ be finitely generated projective left $R$-modules, and let $M$ and $N$ be left $R$-modules.
\begin{enumerate}
\item If $f \colon Q \to N$ and $h \colon P \to Q$ are $R$-linear maps, then
\(
\delta_{P,N}^{-1}(f \circ h) = (h^* \otimes \id_N)\delta_{Q,N}^{-1}(f),
\)
where $h^* \colon {}^*Q \to {}^*P$ is the precomposition map induced by $h$.
\item If $f \colon P \to M$ and $\alpha \colon M \to N$ are $R$-linear maps, then
\(
\delta_{P,N}^{-1}(\alpha \circ f) = (\id_{{}^*P} \otimes \alpha)\delta_{P,M}^{-1}(f).
\)
\end{enumerate}
\end{lemma}

Now we show that twisted Hattori--Stallings trace satisfies analogues of the usual additivity and cyclicity properties.

\begin{prop}[Twisted additivity]\label{prop:twisted-additivity}
Let $P=P_1\oplus P_2$ be a direct sum of finitely generated projective left $R$-modules. If $\psi_i: P_i \to {}^gP_i$ are $R$-linear maps for $i=1,2$, then their direct sum $\psi_1 \oplus \psi_2: P \to {}^gP$ satisfies:
\(
\operatorname{tr}_{\mathrm{HS}}^{g}(\psi_1\oplus \psi_2) = \operatorname{tr}_{\mathrm{HS}}^{g}(\psi_1) + \operatorname{tr}_{\mathrm{HS}}^{g}(\psi_2).
\)
\end{prop}

\begin{proof}
This follows directly from the compatibility of the tensor--Hom isomorphism $\delta$ and the canonical map $\pi^{(g)}$ with direct sum decompositions.
\end{proof}

\begin{prop}[Twisted cyclicity]\label{prop:algebraic-sliding}
Let $P$ and $Q$ be finitely generated projective left $R$-modules. For any $R$-linear maps $\psi: Q \to {}^gP$ and $h: P \to Q$, we have:
\(
\operatorname{tr}_{\mathrm{HS}}^g(\psi \circ h) = \operatorname{tr}_{\mathrm{HS}}^g({}^g h \circ \psi).
\)
\end{prop}

\begin{proof}
By Lemma \ref{lem:delta-naturality}, the pre- and post-composition compatibilities of the tensor--Hom isomorphism yield
\[
\delta_{P,{}^gP}^{-1}(\psi\circ h)
=
(h^*\otimes \id_{{}^gP})\delta_{Q,{}^gP}^{-1}(\psi), \quad
\text{and} \quad 
\delta_{Q,{}^gQ}^{-1}({}^gh\circ\psi)
=
(\id_{{}^*Q}\otimes{}^gh)\delta_{Q,{}^gP}^{-1}(\psi),
\]
where \(h^*:{}^*Q\to{}^*P\) is precomposition by \(h\).

It therefore suffices to show that the two maps
\(
\pi_P^{(g)}\circ(h^*\otimes \id_{{}^gP})\) and 
\(\pi_Q^{(g)}\circ(\id_{{}^*Q}\otimes{}^gh)
\)
from \({}^*Q\otimes_R{}^gP\) to \(\HH_0^g(R)\) agree. Since both maps are
additive, it is enough to check this on pure tensors. Let
\(\xi\otimes p\in{}^*Q\otimes_R{}^gP\). Viewing \(p\) as an element of the
underlying abelian group of \(P\), we have
\[
\pi_P^{(g)}\bigl((h^*\otimes\id_{{}^gP})(\xi\otimes p)\bigr)
=
\pi_P^{(g)}\bigl((\xi\circ h)\otimes p\bigr)
=
[(\xi\circ h)(p)]
=
[\xi(h(p))].
\]
On the other hand,
\(
\pi_Q^{(g)}\bigl((\id_{{}^*Q}\otimes{}^gh)(\xi\otimes p)\bigr)
=
\pi_Q^{(g)}\bigl(\xi\otimes{}^gh(p)\bigr)
=
[\xi({}^gh(p))].
\)
The map \({}^gh:{}^gP\to{}^gQ\) has the same underlying
abelian-group map as \(h:P\to Q\). Hence
\({}^gh(p)=h(p)\), and the two expressions agree in
\(\HH_0^g(R)\).
\end{proof}

\begin{remark}
In either the ring or Green-functor setting, whenever
\((P,\gamma)\colon R\to S\) is a left dualizable \(1\)-cell, its
twisting structure gives a bi-twist isomorphism
\(\gamma_g\colon P\xrightarrow{\cong}{}^gP^g\).
After forgetting the right \(S\)-module structure, we use the same
notation for the induced isomorphism of left \(R\)-modules
\(\gamma_g\colon P\xrightarrow{\cong}{}^gP\).
Accordingly, for a left \(R\)-linear endomorphism \(f\colon P\to P\),
we write
\(\operatorname{tr}_{\mathrm{HS},\gamma}^{\,g}(f)
:=
\operatorname{tr}_{\mathrm{HS}}^{\,g}(\gamma_g\circ f)\).
\end{remark}

Let \((P,\gamma)\colon R\to S\) be a left dualizable \(1\)-cell,
let \(f\colon P\to P\) be left \(R\)-linear, and let
\(h\colon(P,\gamma)\Rightarrow(P,\gamma)\) be an endomorphism
\(2\)-cell. The equivariance condition gives
\(\gamma_g\circ h={}^gh\circ\gamma_g\).
Applying Proposition~\ref{prop:algebraic-sliding} to
\(\gamma_g\circ f\colon P\to{}^gP\) and \(h\colon P\to P\) yields
\(
\operatorname{tr}_{\mathrm{HS},\gamma}^{\,g}(f\circ h)
=
\operatorname{tr}_{\mathrm{HS},\gamma}^{\,g}(h\circ f).
\)

As in the untwisted case
(Example~\ref{ex:HS-as-bicat-trace}), the following theorem identifies
the \(g\)-twisted bicategorical trace with the
\(g\)-twisted Hattori--Stallings trace.

\begin{thm}\label{thm:twHS-recovery-general}
In the bicategory of $G$-twists $\mathbf{Bimod}/\mathcal{R}\mathrm{ing}^{G\text{-tw}}$,
let $R$ and $S$ be $G$-rings, let ${}_R P_S$ be a left dualizable $1$-cell,
and let $f\colon P \Rightarrow P$ be a $2$-cell. For each $s\in S$, let
\(
\rho_s\colon P\to P
\)
denote right multiplication by $s$. Then
\(\operatorname{tr}^g(f)([s])
=
\operatorname{tr}_{\mathrm{HS},\gamma}^{\,g}(f\circ\rho_s)
=
\operatorname{tr}_{\mathrm{HS},\gamma}^{\,g}(\rho_s\circ f).\)
\end{thm}

\begin{proof}
Following the fixed notation, we regard the bi-twist isomorphism as a left $R$-linear map $\gamma_g\colon P \to {}^gP$. Let $q$ denote the ordinary Hochschild quotient map for the $S$-bimodule ${}^*P \otimes_R {}^gP$. We evaluate $\operatorname{tr}^g(f)$ via the following commutative diagram in $\mathbf{Ab}$:

\begin{equation}\label{eq:trace-diagram}
\begin{tikzcd}[column sep=1.7em, row sep=2em, every label/.append style={font=\scriptsize}]
\HH_0^g(S)
  \arrow[r, "\HH_0^g(\bar{\eta})"]
& \HH_0^g(S;{}^*P\otimes_R P)
  \arrow[r, "\HH_0^g(\id\otimes{f})"]
& \HH_0^g(S;{}^*P\otimes_R P)
  \arrow[r, "{}^g\theta"]
& \HH_0^g(R;P\otimes_S {}^*P)
  \arrow[r, "\HH_0^g(\bar{\epsilon})"]
& \HH_0^g(R)
\\
S\arrow[u, "q^{(g)}"]
& {}^*P\otimes_R P
  \arrow[ur, "q^{(g)}"']
& {} 
& \HH_0(S;{}^*P \otimes_R {}^gP)\arrow[u,"\theta","\cong"']
& {}
\\
S
  \arrow[u, equal]
  \arrow[r, "f\circ \rho_{(\cdot)}"]
& \End_R(P)
  \arrow[u, "\delta^{-1}_{P,P}","\cong"']
  \arrow[r, "(\gamma_g)_*"]
& \Hom_R(P,{}^gP)
  \arrow[r, "\delta^{-1}_{P,{}^gP}", "\cong"']
& {}^*P\otimes_R {}^gP
  \arrow[u, " q"]
  \arrow[uur, "\pi^{(g)}"']
& {}
\end{tikzcd}
\end{equation}
The composition along the top row represents the bicategorical trace $\operatorname{tr}^g(f)$, while evaluating $s \in S$ along the bottom row and up the right side computes the Hattori--Stallings trace $\operatorname{tr}_{\mathrm{HS},\gamma}^{\,g}(f \circ \rho_s)$. The desired equality follows from the commutativity of this diagram.

To verify the commutativity of the leftmost region, note that since the quotient map $S \to \HH_0^g(S)$ is surjective, it suffices to check the equality on elements $s \in S$. Following the top path, $s$ maps to $[\bar{\eta}(s)] \in \HH_0^g(S; {}^*P \otimes_R P)$ and then to $[(\id_{{}^*P}\otimes f)\bar{\eta}(s)]$. Following the bottom path, 
$s$ maps to $f \circ \rho_s$ as a left
\(R\)-linear endomorphism of \(P\), which lifts to $\delta^{-1}_{P,P}(f\circ \rho_s) \in {}^*P \otimes_R P$. By explicitly evaluating the coevaluation map on the dual basis of the finitely generated projective left $R$-module $P$, and utilizing the right $S$-linearity of $\bar{\eta}$, we obtain the identity in ${}^*P \otimes_R P$:
$\delta_{P,P}^{-1}(f\circ\rho_s) = (\id_{{}^*P}\otimes f)\bar{\eta}(s)$.
Applying the canonical map to twisted Hochschild homology yields the equality of the classes.

 To verify that the middle region commutes, we trace an element $h \in \End_R(P)$ along both paths. By Lemma~\ref{lem:delta-naturality}, the bottom path yields the identity  in \({}^*P \otimes_R {}^gP:\)
 \(\delta^{-1}_{P,{}^gP}(\gamma_g \circ h) = (\id_{{}^*P} \otimes \gamma_g)\bigl(\delta^{-1}_{P,P}(h)\bigr). 
\)
Projecting this element to homology and applying the untwisted cyclic isomorphism $\theta$ (the right-hand upward path) coincides exactly with applying the twisted cyclic isomorphism ${}^g\theta$ to the homology class of $\delta^{-1}_{P,P}(h)$ (the top path), by the  definition of ${}^g\theta$.

Note that \({}^*P\otimes_R{}^gP\) is an \(S\)-bimodule, and \(\pi^{(g)}((s\cdot\xi)\otimes p) = [\xi(p\cdot s)] = \pi^{(g)}(\xi\otimes(p\cdot s))\). Hence \(\pi^{(g)}\) coequalizes the two \(S\)-actions and induces a unique map \(\overline{\pi}^{(g)}\colon \HH_0(S;{}^*P\otimes_R{}^gP)\to\HH_0^g(R)\) such that \(\overline{\pi}^{(g)}\circ q=\pi^{(g)}\). To verify that the rightmost triangle commutes, it remains to show that \(\HH_0^g(\bar\epsilon)\circ\theta=\overline{\pi}^{(g)}\). After precomposition with \(q\), we have \((\HH_0^g(\bar\epsilon)\circ\theta\circ q)(\xi\otimes p) = \HH_0^g(\bar\epsilon)([p\otimes\xi]) = [\xi(p)] = \pi^{(g)}(\xi\otimes p)\). Since \(q\) is an epimorphism, the desired equality follows.

Tracing an element $s \in S$ along the bottom path: $s$ maps to $f\circ \rho_s$, then to $(\gamma_g)_*(f\circ \rho_s) = \gamma_g \circ f\circ \rho_s$, which lifts to $\delta^{-1}_{P,{}^gP}(\gamma_g \circ f\circ \rho_s)$ and projects via $\pi^{(g)}$ to $\operatorname{tr}_{\mathrm{HS},\gamma}^{g}(f\circ \rho_s)$. By the commutativity of the diagram, this equals the image along the top row, establishing:
\(
\operatorname{tr}^g(f)([s]) = \operatorname{tr}_{\mathrm{HS},\gamma}^{g}(f\circ \rho_s).
\)

Since $f$ is right $S$-linear, we have $f\circ \rho_s=\rho_s\circ f$, and the
second equality follows.
\end{proof}

Since \(\mathbb Z\) is regarded as a \(G\)-ring with trivial \(G\)-action, its
\(g\)-twisted Hochschild quotient is canonically identified with \(\mathbb Z\):
\(
\HH_0^g(\mathbb Z)\cong \mathbb Z.
\)
Under this identification, \(q^{(g)}(1)=[1]\).

\begin{cor}\label{cor:twHS-recovery-Z}
In the setting of Theorem~\ref{thm:twHS-recovery-general}, suppose
$P\colon R\to \mathbb Z$. Then
\(
\operatorname{tr}^g(\id_P)([1])
=
\operatorname{tr}_{\mathrm{HS},\gamma}^{\,g}(\id_P).
\)
\end{cor}

We next construct the $C_n$--Mackey functor analogue of the twisted Hattori--Stallings trace.

Recall from Example~\ref{ex:twisted_HH0_concrete} that for a $C_n$-Green functor $\underline R$ and an $\underline R$-bimodule $\underline M$, the Mackey functor $\underline{\HH}_0^g(\underline R;\underline M)$ is the coequalizer of the twisted left action and the standard right action. Specializing the monoidal product in \eqref{eq:twisted-HH0-qg} to the box product $\square$, let
\begin{equation}\label{eq:qM-g}
q^{(g)} \colon \underline{M} \longrightarrow \underline{\HH}_0^g(\underline{R};\underline{M})
\end{equation}
be the canonical quotient map. Under the identification $\underline{\HH}_0^g(\underline R;\underline M)\cong \underline{\HH}_0(\underline R,{}^{g}\underline M)$, we  also write
\begin{equation}\label{eq:qgM}
q_{{}^{g}\underline M}\colon {}^{g}\underline M \longrightarrow \underline{\HH}_0^g(\underline R;\underline M)
\end{equation}
for the corresponding quotient map. In particular, for any finite $C_n$--set $X$, there is an identification of underlying abelian groups $\underline M(X) = {}^{g}\underline M(X)$, and for $m\in \underline M(X)$ we write
$$[m]\coloneqq q^{(g)}_X(m)=(q_{{}^{g}\underline M})_X(m) \in \underline{\HH}_0^{\,g}(\underline R;\underline M)(X).$$

Let $\underline P$ be a finitely generated projective left $\underline R$-module. As established in Example~\ref{ex:dualizable_green_functor}, the finite generation and projectivity of $\underline P$ ensure that it is left dualizable, with its canonical left dual ${}^*\underline{P} := \underline{\Hom}_{\underline{R}}(\underline{P}, \underline{R})$. Consequently, we obtain the canonical tensor--Hom isomorphism:
\(
\delta_{\underline{P},{}^g\underline{P}} \colon {}^*\underline{P} \square_{\underline{R}} {}^g\underline{P} \xrightarrow{\cong} \underline{\Hom}_{\underline{R}}(\underline{P}, {}^g\underline{P}).\)

Recall from Definition~\ref{def:relative-box-equalizer} that ${}^{*}\underline{P}\square_{\underline R}{}^{g}\underline{P}$ is the coequalizer of the right $\underline R$-action on ${}^{*}\underline{P}$ and the (twisted) left $\underline R$-action on ${}^{g}\underline{P}$.
Let $\bar{\epsilon}\colon \underline{P}\square {}^{*}\underline{P}$
$\to \underline{R}$ be the  evaluation. Define the $g$-twisted evaluation $\bar{\epsilon}^{(g)}\colon {}^{*}\underline{P}\square {}^{g}\underline{P}\to {}^{g}\underline{R}$ to be the composite
\begin{equation}\label{eq:PgP-to-gR}{}^{*}\underline{P}\square {}^{g}\underline{P} \xrightarrow{\;\mathfrak{s}\;} {}^{g}\underline{P}\square {}^{*}\underline{P} \xrightarrow{\;\cong\;} {}^{g}\underline{R}\square_{\underline R}(\underline{P}\square {}^{*}\underline{P}) \xrightarrow{\;\mathrm{id}\square \bar{\epsilon}\;} {}^{g}\underline{R}\square_{\underline R}\underline{R} \xrightarrow{\;\cong\;} {}^{g}\underline{R}.\end{equation}

The map \(\bar{\epsilon}^{(g)}\) sends the two  parallel maps  ${}^{*}\underline{P}\square \underline{R}\square {}^{g}\underline{P}\rightrightarrows{}^{*}\underline{P}\square {}^{g}\underline{P}$ to the
twisted left and ordinary right \(\underline R\)-actions on
\({}^g\underline R\). 
Since
\(q_{{}^g\underline R}:{}^g\underline R\to
\underline{\HH}_0^g(\underline R)\)
coequalizes the twisted left and ordinary right
\(\underline R\)-actions, the composite
\(q_{{}^g\underline R}\circ\bar\epsilon^{(g)}\) coequalizes the two
parallel arrows defining
\({}^*\underline P\square_{\underline R}{}^g\underline P\).
Hence it induces a unique morphism
\(\pi^{(g)}\) making the following diagram commute:
\begin{equation}\label{diag:twisted-pi-definition}
\begin{tikzcd}[column sep=3em,row sep=1.5em]
{}^{*}\underline{P}\square \underline R \square {}^g\underline{P}
  \ar[r,shift left=.55ex,"\mathrm{id}\square \lambda_{{}^{g}\underline{P}}"]
  \ar[r,shift right=.55ex,swap,"\rho_{{}^{*}\underline{P}}\square\mathrm{id}"]
&
{}^{*}\underline{P}\square {}^g\underline{P}
\ar[r,"b"]
  \ar[d,"\bar{\epsilon}^{(g)}"']
&
{}^{*}\underline{P}\square_{\underline R}{}^g\underline{P}
  \ar[d,dashed,"\pi^{(g)}"]
\\
{} &
{}^{g}\underline R
  \ar[r,"q"]
&
\underline{\HH}_0^{\,g}(\underline{R})
\end{tikzcd}
\end{equation}

\begin{definition}\label{def:twisted-trace-hom-morphism}
Let $\underline{P}$ be a finitely generated projective left $\underline{R}$--module. The \emph{$g$--twisted Hattori--Stallings trace} of $\underline P$ is the morphism
$\operatorname{tr}_{\mathrm{HS}}^{g} \;\coloneqq\; \pi^{(g)} \circ \delta_{\underline{P},\,{}^{g}\underline{P}}^{-1} \;:\; \underline{\Hom}_{\underline{R}}(\underline{P}, {}^{g}\underline{P}) \to \underline{\HH}_0^{g}(\underline{R}).$
Evaluated at a finite $C_n$--set $X$ for an element $\psi\in \underline{\Hom}_{\underline{R}}(\underline{P}, {}^{g}\underline{P})(X)$, this gives:
$
\bigl(\operatorname{tr}_{\mathrm{HS}}^g\bigr)_X(\psi) = \pi_X^{(g)}\!\bigl( (\delta_{\underline{P},\,{}^{g}\underline{P}}^{-1})_X(\psi) \bigr) \in \underline{\HH}_0^{g}(\underline{R})(X).
$
\end{definition}

We now establish the Mackey-functor-valued analogue of
Proposition~\ref{prop:algebraic-sliding}. Let \(\underline P\) and
\(\underline Q\) be finitely generated projective left
\(\underline R\)-modules. Bouc's associative pairing
(Example~\ref{ex:endomorphism-green}) gives a morphism of Mackey
functors
\(
\hat\circ\colon
\underline{\Hom}_{\underline R}(\underline Q,{}^g\underline P)
\square
\underline{\Hom}_{\underline R}(\underline P,\underline Q)
\to
\underline{\Hom}_{\underline R}(\underline P,{}^g\underline P).
\)

To formulate twisted cyclicity for this pairing, we first construct
the morphism on internal Homs induced by twisting. The construction
rests on the following compatibility between twisting and
tensor--Hom duality.

\begin{lemma}\label{lem:twisted-dual-box}
Let $\underline{P}$ and $\underline{Q}$ be finitely generated projective left $\underline{R}$-modules. For any $g \in C_n$, there is a natural isomorphism $\kappa_{g}\colon {}^*({}^g\underline{P}) \square_{\underline{R}} {}^g\underline{Q} \xrightarrow{\cong} {}^*\underline{P} \square_{\underline{R}} \underline{Q}$, given as the composite \({}^*({}^g\underline{P}) \square_{\underline{R}} {}^g\underline{Q} \xrightarrow{\cong} ({}^*\underline{P})^g \square_{\underline{R}} {}^g\underline{Q} \xrightarrow{\mathfrak{m}} {}^*\underline{P} \square_{\underline{R}} \underline{Q}.\)
\end{lemma}

\begin{proof}
The first isomorphism is the instance of Lemma~\ref{lem:dual-of-twist} in the bicategory $\mathbf{Bimod}/\mathbf{Green}^{C_n}$, applied to the left dualizable $1$-cell $\underline{P}$. The second isomorphism is the twist cancellation $\mathfrak{m}\colon ({}^*\underline{P})^g \square_{\underline{R}} {}^g\underline{Q} \xrightarrow{\cong} {}^*\underline{P} \square_{\underline{R}} \underline{Q}$ from Lemma~\ref{lem:twist_cancel}, applied to $M = {}^*\underline{P}$ and $N = \underline{Q}$. 
\end{proof}

Via the tensor--Hom adjunctions, Lemma~\ref{lem:twisted-dual-box} induces a natural morphism of Mackey functors $\mathrm{tw}_g\colon \underline{\Hom}_{\underline R}(\underline P,\underline Q) \to \underline{\Hom}_{\underline R}({}^{g}\underline P,{}^{g}\underline Q)$ defined by the composite \[\underline{\Hom}_{\underline R}(\underline P,\underline Q) \xrightarrow{\delta^{-1}_{\underline P,\underline Q}} {}^{*}\underline P \square_{\underline R} \underline Q \xrightarrow{\kappa_{g}^{-1}} {}^{*}({}^{g}\underline P)\square_{\underline R} {}^{g}\underline Q \xrightarrow{\delta_{{}^{g}\underline P,{}^{g}\underline Q}} \underline{\Hom}_{\underline R}({}^{g}\underline P,{}^{g}\underline Q).\] Equivalently, $\mathrm{tw}_g$ is characterized by the identity
\begin{equation}\label{eq:twg-characterization}
\delta^{-1}_{{}^{g}\underline P,{}^{g}\underline Q}\circ \mathrm{tw}_g = \kappa_{g}^{-1}\circ \delta^{-1}_{\underline P,\underline Q}.
\end{equation}

Evaluated on a finite \(C_n\)-set \(X\), for \(h\in \underline{\Hom}_{\underline R}(\underline P,\underline Q)(X)\), let \(h_X\colon \underline P_X\to \underline Q\) denote the corresponding \(\underline R\)-module morphism under the canonical identification \(\underline{\Hom}_{\underline R}(\underline P,\underline Q)(X) \cong \mathbf{Mack}_{C_n}(\underline A_X,\underline{\Hom}_{\underline R}(\underline P,\underline Q)) \cong \mathbf{Mod}_{\underline R}(\underline A_X\square \underline P,\underline Q)\), given by the Yoneda lemma and the tensor--Hom adjunction (Proposition~\ref{prop:bicategorical-tensor-Hom}).
Put
\(
L\coloneqq\Theta_{\underline R}(g^{-1}),
\)
so that
\(
{}^g\underline P=L\square_{\underline R}\underline P
\)
and similarly for \(\underline Q\).

\begin{lemma}\label{lem:twisting-compatible-with-yoneda}
Under the canonical identification \(({}^g\underline P)_X \cong L\square_{\underline R}\underline P_X\), the element \((\mathrm{tw}_g)_X(h)\) corresponds to \({}^g h_X\coloneqq \id_L\square_{\underline R}h_X \colon L\square_{\underline R}\underline P_X \to L\square_{\underline R}\underline Q = {}^g\underline Q\).
\end{lemma}
\begin{proof}
Under the Yoneda identification and the tensor--Hom isomorphism, \(h\) corresponds to a morphism \(u_X\colon \underline A_X \to {}^*\underline P\square_{\underline R}\underline Q\). By \eqref{eq:twg-characterization}, the morphism corresponding to \((\mathrm{tw}_g)_X(h)\) is \(v_X=\kappa_g^{-1}\circ u_X\colon \underline A_X \to {}^*({}^g\underline P) \square_{\underline R}{}^g\underline Q\). Under the canonical identifications for duals of composites, \(\kappa_g^{-1}\) inserts the coevaluation \(\bar\eta_L\). Evaluating \(v_X\) via the evaluation map for \(L\square_{\underline R}\underline P\) introduces the composite \((\bar\epsilon_L\square_{\underline R}\id_L) \circ (\id_L\square_{\underline R}\bar\eta_L) = \id_L\) by the triangle identity. Hence the resulting morphism is precisely \(\id_L\square_{\underline R}h_X = {}^gh_X\).
\end{proof}

By Lemma~\ref{lem:twisting-compatible-with-yoneda}, for every finite
\(C_n\)-set \(X\), the \(X\)-component of \(\mathrm{tw}_g\)
corresponds to \(h_X\mapsto{}^gh_X\). Hence the internal
Mackey-functor analogue of
Proposition~\ref{prop:algebraic-sliding} takes the following form.

\begin{prop}[Equivariant twisted cyclicity]
\label{prop:eq-algebraic-sliding}
Let \(\underline P\) and \(\underline Q\) be finitely generated
projective left \(\underline R\)-modules, and let
\[
\mathfrak{s}\colon
\underline{\Hom}_{\underline R}(\underline Q,{}^g\underline P)
\square
\underline{\Hom}_{\underline R}(\underline P,\underline Q)
\xrightarrow{\cong}
\underline{\Hom}_{\underline R}(\underline P,\underline Q)
\square
\underline{\Hom}_{\underline R}(\underline Q,{}^g\underline P)
\]
be the symmetry isomorphism. Write
\(\operatorname{tr}_{\mathrm{HS},\underline P}^g\) and
\(\operatorname{tr}_{\mathrm{HS},\underline Q}^g\)
for the twisted Hattori--Stallings traces associated to
\(\underline P\) and \(\underline Q\), respectively. Then
\(\operatorname{tr}_{\mathrm{HS},\underline P}^g\circ\hat\circ
=
\operatorname{tr}_{\mathrm{HS},\underline Q}^g
\circ\hat\circ
\circ(\mathrm{tw}_g\square\id)
\circ\mathfrak{s}.
\)
\end{prop}

\begin{proof}
By the universal property of the box product, it suffices to compare
the associated componentwise pairings at every finite \(C_n\)-set
\(X\). Fix
\(
\psi\in
\underline{\Hom}_{\underline R}
(\underline Q,{}^g\underline P)(X),
h\in
\underline{\Hom}_{\underline R}
(\underline P,\underline Q)(X).
\)
Under the Yoneda identifications and the inverse tensor--Hom
isomorphisms, let
\(
u_\psi\colon
\underline A_X
\to
{}^*\underline Q\square_{\underline R}{}^g\underline P,
u_h\colon
\underline A_X
\to
{}^*\underline P\square_{\underline R}\underline Q
\)
be the morphisms corresponding to \(\psi\) and \(h\), respectively. Let \(m_{\underline P,\underline Q}\colon ({}^*\underline P\square_{\underline R}\underline Q) \square \allowbreak ({}^*\underline Q\square_{\underline R}{}^g\underline P) \to {}^*\underline P\square_{\underline R}{}^g\underline P\) be the contraction morphism evaluating the  \(\underline Q\) factors, and  \(m^g_{\underline Q,\underline P}\colon ({}^*({}^g\underline P)\square_{\underline R}{}^g\underline Q) \square \allowbreak ({}^*\underline Q\square_{\underline R}{}^g\underline P) \allowbreak \to {}^*\underline Q\square_{\underline R}{}^g\underline Q\) be the corresponding twisted contraction evaluating the \({}^g\underline P\) factors (which requires rearranging the tensor factors via symmetry).

By the description of Bouc's componentwise pairing in Example~\ref{ex:endomorphism-green}, and using the cocommutativity of $\Delta_X$, we have
\(
(\delta^{-1}_{\underline P,{}^g\underline P})_X (\psi\mathbin{\hat\circ}h) = m_{\underline P,\underline Q} \circ (u_h\square u_\psi) \circ \Delta_X.
\)
For the right-hand composite, \eqref{eq:twg-characterization} identifies the morphism
corresponding to $(\mathrm{tw}_g)_X(h)$ with $\kappa_g^{-1}\circ u_h$, yielding \[(\delta^{-1}_{\underline Q,{}^g\underline Q})_X ((\mathrm{tw}_g)_X(h)\mathbin{\hat\circ}\psi) = m^g_{\underline Q,\underline P} \circ (\kappa_g^{-1}\square\id) \circ (u_h\square u_\psi) \circ \Delta_X.\] It therefore remains to prove the equality
\begin{equation}\label{eq:pi_contraction}
\pi_{\underline P}^{(g)} \circ m_{\underline P,\underline Q} = \pi_{\underline Q}^{(g)} \circ m^g_{\underline Q,\underline P} \circ (\kappa_g^{-1}\square\id) 
\end{equation}
as maps $({}^*\underline P\square_{\underline R}\underline Q) \square ({}^*\underline Q\square_{\underline R}{}^g\underline P) \to \underline{\HH}_0^g(\underline R)$.

Let
\(
p\colon
{}^*\underline P\square\underline Q
\square{}^*\underline Q\square{}^g\underline P
\to
({}^*\underline P\square_{\underline R}\underline Q)
\square
({}^*\underline Q\square_{\underline R}{}^g\underline P)
\)
be the canonical quotient map. Since relative box products are
defined by coequalizers and \(\square\) preserves coequalizers in
each variable, \(p\) is an epimorphism. Thus equality of two
morphisms from its target may be checked after precomposition
with \(p\).
Let
\(
\sigma\colon
{}^*\underline P\square\underline Q
\square{}^*\underline Q\square{}^g\underline P
\xrightarrow{\cong}
({}^*\underline P\square{}^g\underline P)
\square
(\underline Q\square{}^*\underline Q)
\)
be the canonical permutation of the factors, and set
\(s
:=
(\bar\epsilon_{\underline P}^{(g)}
 \square\bar\epsilon_{\underline Q})
\circ\sigma
\colon
{}^*\underline P\square\underline Q
\square{}^*\underline Q\square{}^g\underline P
\to
{}^g\underline R\square\underline R.
\)
Here \(\bar\epsilon_{\underline P}^{(g)}\) is the twisted evaluation
and \(\bar\epsilon_{\underline Q}\) is the ordinary evaluation.
Unwinding the definitions of the contraction maps and of
\(\kappa_g\), and using the triangle identities, the composites of
the two sides of \eqref{eq:pi_contraction} with \(p\) are
\(
q\circ\mu_R\circ s
\) and
\(q\circ\mu_L^{(g)}\circ s,
\)
where \(\mu_R, \mu_L^{(g)}\colon {}^g\underline{R}\square\underline{R}\to{}^g\underline{R}\) denote the ordinary right action and the twisted left action (precomposed with the symmetry isomorphism \(\mathfrak{s}\)), respectively.
 Since
\(q\colon{}^g\underline R\to\underline{\HH}_0^g(\underline R)\)
coequalizes these actions, the two composites agree. Since \(p\) is
epimorphic, \eqref{eq:pi_contraction} follows.

Applying \eqref{eq:pi_contraction} to
\((u_h\square u_\psi)\circ\Delta_X\), and using the definition of the
twisted Hattori--Stallings trace, gives
\(
\bigl(\operatorname{tr}_{\mathrm{HS},\underline P}^g\bigr)_X
(\psi\mathbin{\hat\circ}h)
=
\bigl(\operatorname{tr}_{\mathrm{HS},\underline Q}^g\bigr)_X
\bigl((\mathrm{tw}_g)_X(h)\mathbin{\hat\circ}\psi\bigr).
\)
Since this holds for every finite \(C_n\)-set \(X\), the asserted
morphisms of Mackey functors are equal.
\end{proof}

Let \(\underline R\) and \(\underline S\) be \(C_n\)-Green functors,
and 
\(
(\underline P,\gamma)\colon\underline R\to\underline S
\)
be a left dualizable \(1\)-cell in
\(\mathbf{Bimod}/\mathbf{Green}^{C_n}\).
Postcomposition with
\(\gamma_g\colon\underline P\to{}^g\underline P\) defines
\(
(\gamma_g)_*\colon
\underline{\End}_{\underline R}(\underline P)
\to
\underline{\Hom}_{\underline R}
(\underline P,{}^g\underline P).
\)
Set
\(
\operatorname{tr}_{\mathrm{HS},\gamma}^{\,g}
:=
\operatorname{tr}_{\mathrm{HS}}^{\,g}\circ(\gamma_g)_*
\colon\)
\(\underline{\End}_{\underline R}(\underline P)
\to
\underline{\HH}_0^{\,g}(\underline R).
\)

The right \(\underline S\)-action on \(\underline P\) has, under the
tensor--Hom adjunction, an adjoint morphism of Mackey functors
\(
\varrho\colon
\underline S\longrightarrow
\underline{\End}_{\underline R}(\underline P).
\)
Explicitly, for a finite \(C_n\)-set \(X\), write
\(\underline P_X=\underline P\square\underline A_X\) and
\(\underline S_X=\underline S\square\underline A_X\).
If \(s\in\underline S(X)\), let
\(\widehat s\colon\underline S_X\to\underline S\)
be the corresponding morphism of left \(\underline S\)-modules.
Under the Yoneda identification, \(\varrho_X(s)\) is represented by
the composite
\begin{equation}\label{eq:rhoXs-def}
\varrho_{X,s}\colon
\underline P_X
\xrightarrow{\cong}
\underline P\square_{\underline S}\underline S_X
\xrightarrow{\id_{\underline P}\square_{\underline S}\widehat s}
\underline P\square_{\underline S}\underline S
\xrightarrow{\mathfrak r_{\underline P}}
\underline P .
\end{equation}
Postcomposition with an endomorphism \(2\)-cell \(f\colon\underline P\Rightarrow\underline P\) yields an element in \(\underline{\End}_{\underline R}(\underline P)(X)\), which is represented by \(f\circ\varrho_{X,s}\).

The following theorem is the Mackey-functor-valued analogue of
Theorem~\ref{thm:twHS-recovery-general}.

 \begin{thm}\label{thm:tw-eq-bicat-vs-HS}
With the notation above, let $f\colon \underline P\Rightarrow \underline P$ be a $2$-cell. Then the following diagram in $\mathbf{Mack}_{C_n}$ commutes:
\[
\begin{tikzcd}[column sep=3em, row sep=1.5em]
\underline S
\arrow[r,"f_*\circ\varrho"]
\arrow[d,"q^{(g)}"']
&
\underline{\End}_{\underline R}(\underline P)
\arrow[d,"\operatorname{tr}_{\mathrm{HS},\gamma}^{\,g}"]\\
\underline{\HH}_0^{\,g}(\underline S) \arrow[r, "\operatorname{tr}^g(f)"'] & \underline{\HH}_0^{\,g}(\underline R)
\end{tikzcd}
\]
where $f_*=\underline{\Hom}_{\underline R}(\underline P,f) \colon \underline{\operatorname{End}}_{\underline{R}}(\underline{P}) \longrightarrow \underline{\operatorname{End}}_{\underline{R}}(\underline{P})$ is the morphism of Mackey functors induced by post-composition with the $2$-cell $f$.
Equivalently, after evaluating at a finite $C_n$-set $X$, for every $s\in\underline S(X)$, one has
\[
(\operatorname{tr}^g(f))_X(q_X^{(g)}(s))
=
\bigl(\operatorname{tr}_{\mathrm{HS},\gamma}^{\,g}\bigr)_X
(f\circ\varrho_{X,s})
=
\bigl(\operatorname{tr}_{\mathrm{HS},\gamma}^{\,g}\bigr)_X
(\varrho_{X,s}\circ f_X).
\]
\end{thm}

\begin{proof}
The proof is the internal Mackey-functor analogue of
Diagram~\ref{eq:trace-diagram}. Namely, replace modules, tensor
products, and Hom groups in that diagram by Mackey-functor modules,
relative box products, and internal Hom Mackey functors, respectively.

The resulting diagram commutes for the same  reasons:
the left-hand region follows from the tensor--Hom adjunction and the fact that
\(\varrho\colon\underline S\to
\underline{\End}_{\underline R}(\underline P)\)
is the adjunct of the right \(\underline S\)-action on
\(\underline P\);
the middle region follows from the naturality of \(\delta\) and the
definition of the twisted cyclicity map \({}^g\theta\); and the
right-hand triangle follows from the universal property defining
\(\pi^{(g)}\) in
Diagram~\ref{diag:twisted-pi-definition}. Consequently, \(
\operatorname{tr}^g(f)\circ q^{(g)}
=
\operatorname{tr}_{\mathrm{HS},\gamma}^{\,g}
\circ f_*\circ\varrho.
\)

Evaluating this identity at a finite \(C_n\)-set \(X\) and applying it
to \(s\in\underline S(X)\) gives
\(
(\operatorname{tr}^g(f))_X(q_X^{(g)}(s))
=
\bigl(\operatorname{tr}_{\mathrm{HS},\gamma}^{\,g}\bigr)_X
(f\circ\varrho_{X,s}).
\)
Finally, by the definition of \(\varrho_{X,s}\) in \eqref{eq:rhoXs-def}, and
under the canonical identification
\(
\underline P_X\cong \underline P\square_{\underline S}\underline S_X,
\)
the map \(f_X=f\square\id_{\underline A_X}\) corresponds to
\(f\square_{\underline S}\id_{\underline S_X}\). Since \(f\) is a morphism of
right \(\underline S\)-modules, it commutes with the right unitor
\(
\mathfrak r_{\underline P}:
\underline P\square_{\underline S}\underline S\to \underline P.
\)
Therefore
\(
f\circ\varrho_{X,s}
=
\varrho_{X,s}\circ f_X.
\)
\end{proof}

\subsection{The degree-zero twisted Dennis trace}

Having shown that the \(g\)-twisted bicategorical trace recovers the
\(g\)-twisted Hattori--Stallings trace, we now show that, on classes
arising from \(G\)-twisted \(1\)-cells, the degree-zero twisted Dennis
trace introduced in \cite[Section~6.1]{AGHKK} is realized by the
\(g\)-twisted bicategorical trace.

 Following \cite[Remark~3.2.11]{AGHKK}, a perfect left $R$-module is a cofibrant, finitely generated left $R$-module $P$ such that, for every left $R$-module $N$, the canonical morphism 
\[
\underline{\Hom}_R(P,R)\otimes_R^{\mathbb L}N
\longrightarrow
\underline{\Hom}_R(P,N)
\]
is an isomorphism in $\Ho(\mathcal V)$. This is the tensor--Hom criterion ensuring that $P$, regarded as a $1$-cell 
\(
P\colon R\to I,
\)
is left dualizable in $\Ho\mathbf{Mod}(\mathcal V)$.

For module spectra, this recovers the standard notion of perfect modules; for \(\mathcal V=s\mathbf{Ab}\), it corresponds via the Dold--Kan correspondence to bounded complexes of finitely generated projective modules (and in particular, to finitely generated projective modules concentrated in degree zero).

For an $R$-bimodule $M$ in $\mathcal V$, let $\End(R,M)$ denote the category of $M$-parametrized endomorphisms. Its objects are pairs $(P,u)$, where $P$ is a perfect left $R$-module and 
\(
u\colon P\to M\otimes_R P
\)
is a morphism of left $R$-modules. A morphism 
\(
a\colon(P,u)\to(Q,v)
\)
is a morphism of left $R$-modules satisfying 
\(
v\circ a = (\id_M\otimes_R a)\circ u.
\)

Write 
\(
K(R;M) := K(\End(R,M)).
\)
The generalized Dennis trace construction recalled in the proof of \cite[Theorem~6.1.1]{AGHKK} gives a map 
\(
\operatorname{trc}_{R;M}\colon
K(R;M)\longrightarrow\HH_{\mathcal V}(R;M).
\)
For rings and ring spectra, this is the construction of \cite[Definition~6.16]{CLMPI}. On degree zero, it sends the class of an $M$-parametrized endomorphism to the homotopy class of its bicategorical trace:
\[
\pi_0(\operatorname{trc}_{R;M})([(P,u)])
= [\operatorname{tr}([u])]
\in
\pi_0\HH_{\mathcal V}(R;M),
\]
where 
\([u]\colon
P\longrightarrow M\otimes_R^{\mathbb L}P
\)
is the $2$-cell in $\Ho\mathbf{Mod}(\mathcal V)$ represented by $u$; see \cite[Example~6.26]{CLMPI}. Since $P$ is perfect, it is left dualizable in this homotopy bicategory, and hence $\operatorname{tr}([u])$ is defined.

The trace above vanishes on the image of the zero-endomorphism inclusion 
\(
K(R)\to K(R;M)
\)
and therefore factors through the reduced endomorphism $K$-theory $\widetilde{K}(R;M)$; see \cite[Remark~6.21]{CLMPI}. The reduced factorization is the map appearing in \cite[Theorem~1.2]{CLMPI}.  

Now let $R$ be a cofibrant monoid in $\mathcal V$ equipped with a $G$-action 
\(
\alpha\colon
G\to
\operatorname{Aut}_{\operatorname{Mon}(\mathcal V)}(R),
\)
and fix $g\in G$. Following \cite[Section~6.1]{AGHKK}, let 
\(
\Perf_{(R,g)}
:=
\Perf_{(R,\alpha_g)}.
\)
An object of $\Perf_{(R,g)}$ is a perfect left $R$-module $P$ equipped with an $\alpha_g$-semilinear automorphism 
\(
\phi\colon P\xrightarrow{\cong}P.
\)
By Lemma~\ref{lem:gamma-vs-semilinear}, such an object may equivalently be written as a pair $(P,u)$, where 
\(
u\colon P\xrightarrow{\cong}{}^gP
\)
is an isomorphism of left $R$-modules in $\mathcal V$. A morphism 
\(
a\colon(P,u)\longrightarrow(Q,v)
\)
is a morphism of left $R$-modules satisfying 
\(
v\circ a={}^g a\circ u.
\)

Let 
\(
\Aut(R,M)\subseteq\End(R,M)
\)
denote the full subcategory spanned by the pairs $(P,u)$ for which $u$ is an isomorphism. The canonical identification 
\(
{}^gR\otimes_R P\cong{}^gP
\)
gives 
\(
\Perf_{(R,g)}
\cong
\Aut(R,{}^gR)
\subseteq
\End(R,{}^gR).
\)
Write 
\(
K(R,g) := K(\Perf_{(R,g)})\) and \(
K_0(R,g) := \pi_0K(R,g).
\)
Composing the inclusion above with the generalized Dennis trace gives the twisted Dennis trace of \cite[Theorem~6.1.1]{AGHKK}:
\[
K(R,g)
\longrightarrow
K(R;{}^gR)
\xrightarrow{\operatorname{trc}_{R;{}^gR}}
\HH_{\mathcal V}(R;{}^gR).
\]
Passing to degree zero gives 
\(
\operatorname{tr}_{\mathrm{Dennis}}^g\colon
K_0(R,g)
\to
\pi_0\HH_{\mathcal V}(R;{}^gR).
\)
For $(P,u)\in\Perf_{(R,g)}$, let 
\(
[u]\colon P\xrightarrow{\cong}{}^gP
\)
denote the $2$-cell in $\Ho\mathbf{Mod}(\mathcal V)$ represented by $u$. Under the identification 
\(
{}^gP
\cong
{}^gR\otimes_R^{\mathbb L}P,
\)
this is the $2$-cell corresponding to the ${}^gR$-parametrized endomorphism $u$. By \cite[Remark~6.1.3]{AGHKK} and \cite[Example~6.26]{CLMPI},
\[
\operatorname{tr}_{\mathrm{Dennis}}^g([(P,u)])
= [\operatorname{tr}([u])]
\in
\pi_0\HH_{\mathcal V}(R;{}^gR).
\]
We now realize this value as a $g$-twisted bicategorical trace.

\begin{prop}\label{prop:algebraic_dennis_trace_recovery}
Let $(P,u)\in\Perf_{(R,g)}$. Suppose that the $1$-cell $P\colon R\to I$ in $\Ho\mathbf{Mod}(\mathcal V)$ is equipped with twisting constraints $\tau$, so that 
\(
(P,\tau)\colon R\to I
\)
is a $1$-cell in $\Ho\mathbf{Mod}(\mathcal V)^{G\text{-tw}}$. Using the canonical identification $P^g\cong P$, let 
\(
\gamma_{g,P}\colon
P\xrightarrow{\cong}{}^gP
\)
denote its $g$-bi-twist isomorphism. If 
\(
f\colon(P,\tau)\Rightarrow(P,\tau)
\)
is a $2$-cell satisfying 
\(
[u]=\gamma_{g,P}\circ f
\)
in the homotopy category of left $R$-modules, then 
\[
\operatorname{tr}_{\mathrm{Dennis}}^g([(P,u)])
= [\operatorname{tr}^g(f)]
\in
\pi_0\HH_{\mathcal V}(R;{}^gR).
\]
\end{prop}
\begin{proof}
Since $P$ is perfect, its underlying $1$-cell in $\Ho\mathbf{Mod}(\mathcal V)$ is left dualizable. Therefore Proposition~\ref{prop:Gtw-dualizable-recognition} implies that $(P,\tau)$ is left dualizable in $\Ho\mathbf{Mod}(\mathcal V)^{G\text{-tw}}$, so $\operatorname{tr}^g(f)$ is defined.

By definition of the degree-zero twisted Dennis trace, we have
\(
\operatorname{tr}_{\mathrm{Dennis}}^g([(P,u)]) = [\operatorname{tr}([u])].
\)
The hypothesis $[u]=\gamma_{g,P}\circ f$ allows us to replace $[u]$ by $\gamma_{g,P}\circ f$ in the trace composite defining $\operatorname{tr}([u])$. The coevaluation, the map induced by $f$, and the evaluation are then precisely the corresponding maps in Definition~\ref{def:twisted-bicat-trace}. Moreover, by \eqref{def:twisted-cyclicity}, the cyclicity map together with the transport determined by $\gamma_{g,P}$ is exactly the $g$-twisted cyclicity isomorphism ${}^g\theta$. Consequently, 
\(
\operatorname{tr}([u])
= \operatorname{tr}^g(f)\colon
I\to
\HH_{\mathcal V}(R;{}^gR).
\)
Passing to homotopy classes proves the result.
\end{proof}
The same argument gives the coefficient version.

\begin{cor}\label{cor:generalized-twisted-trace-comparison}
Let 
\(
(M,\tau_M)\colon R\to R\) and \(
(P,\tau_P)\colon R\to I
\)
be $1$-cells in $\Ho\mathbf{Mod}(\mathcal V)^{G\text{-tw}}$, and suppose that $P$ is perfect. Let 
\(
f\colon
(P,\tau_P)
\Rightarrow
(M,\tau_M)\otimes_R^{\mathbb L}(P,\tau_P)
\)
be a $2$-cell. Define 
\[
\widehat f\colon
P
\xrightarrow{f}
M\otimes_R^{\mathbb L}P
\xrightarrow{\gamma_{g,M\otimes_R^{\mathbb L}P}}
{}^g\bigl(M\otimes_R^{\mathbb L}P\bigr)
\cong
{}^gM\otimes_R^{\mathbb L}P.
\]
Suppose that $(P,u)\in\End(R,{}^gM)$ and that 
\([u]=\widehat f
\)
in the homotopy category of left $R$-modules. Then 
\[
\pi_0(\operatorname{trc}_{R;{}^gM})([(P,u)])
= [\operatorname{tr}^g(f)]
\in
\pi_0\HH_{\mathcal V}(R;{}^gM).
\]
\end{cor}
\begin{proof}
By the degree-zero description of the generalized Dennis trace, 
\(
\pi_0(\operatorname{trc}_{R;{}^gM})([(P,u)]) = [\operatorname{tr}([u])].
\)
Substituting $[u]=\widehat f$ into this trace composite and using \eqref{def:twisted-cyclicity} identifies the resulting cyclicity and transport maps with ${}^g\theta$. The resulting composite is therefore the one defining $\operatorname{tr}^g(f)$.
\end{proof}

\begin{remark}\label{rem:connection-to-MM}
Suppose that $G=C_n$, that $R$ is a $C_n$-ring spectrum,
and that $g$ is the chosen generator. Let
$\Perf_R^{C_n}$ denote the category of perfect
left $R$-modules equipped with a coherent semilinear $C_n$-action.
By \cite[Theorem~1.3]{MM}, there is an equivalence
\(
K_{C_n}(R)^{C_n}
\simeq
K\bigl(\Perf_R^{C_n}\bigr).
\)

Restriction of a semilinear $C_n$-action to the chosen generator
defines an exact functor
\(F_g:
\Perf_R^{C_n}
\to
\Perf_{(R,g)}.
\)
The composite
\[
K_{C_n}(R)^{C_n}
\xrightarrow{\simeq}
K\bigl(\Perf_R^{C_n}\bigr)
\xrightarrow{K(F_g)}
K(R,g)
\xrightarrow{\operatorname{tr}_{\mathrm{Dennis}}^g}
\THH(R;{}^gR)
\cong
\THH_{C_n}(R)
\]
agrees with the twisted Dennis trace of
\cite[Corollary~6.2.1]{AGHKK}.
For an object
\(
(P,\tau)\in\Perf_R^{C_n},
\)
restriction to the chosen generator gives
\(
F_g(P,\tau)
=
(P,u_g)\in\Perf_{(R,g)},
\)
where the twisting constraint on \(P\) induces an identification
\(
[u_g]=\gamma_{g,P}\colon
P\xrightarrow{\cong}{}^gP
\)
in the homotopy category of left \(R\)-modules.
Therefore
Proposition~\ref{prop:algebraic_dennis_trace_recovery}
applies with \(f=\id_P\). Consequently, the degree-zero
value of the composite above on the class of \((P,\tau)\) is
\(
[\operatorname{tr}^g(\id_P)]
\in
\pi_0\THH(R;{}^gR)
\cong
\pi_0\THH_{C_n}(R).
\)
\end{remark}

 \subsection{Twisted equivariant Dennis traces} The degree-zero twisted Dennis trace discussed in Section~5.2 is a
homomorphism of abelian groups. For a \(G\)-Green functor \(\underline R\) where \(G\) is abelian, we construct a Mackey-functor-valued refinement of
this trace. At the orbit \(G/G\),
this morphism recovers the degree-zero generalized Dennis trace
associated to the automorphism induced by \(g\), as in
\cite[Section~6.1]{AGHKK}.

Throughout this subsection, put $G = C_n$, fix $g \in G$, and let $\underline{R}$ be a $G$-Green functor. For $H \le G$, write $\underline{R}_H := \operatorname{Res}_H^G \underline{R}$. Similarly, for an $\underline{R}$-bimodule $\underline{M}$, write $\underline{M}_H := \operatorname{Res}_H^G \underline{M}$. Recall from Example~\ref{ex:green-functors} that the action of $g$ induces an automorphism $\alpha_g\colon \underline{R} \xrightarrow{\cong} \underline{R}$ of the $G$-Green functor. For $H \le G$, we write $\alpha_g^H := \operatorname{Res}_H^G(\alpha_g)\colon \underline{R}_H \xrightarrow{\cong} \underline{R}_H$ for the induced automorphism of the restricted $H$-Green functor. If $\underline{P}$ is a left $\underline{R}_H$-module, we write ${}^g\underline{P}$ for the left module obtained by precomposing the left action with $\alpha_g^H$.

\begin{remark}\label{rem:restricted-twisting-data}
Although \(\underline R_H\) is an
\(H\)-Green functor, the twisting used below is still indexed by the
chosen element \(g\in C_n\), not necessarily by an element of \(H\). The \(C_n\)-twisting
data on \(\underline R\) restrict to a group homomorphism
\(
\alpha^H:C_n\to
\operatorname{Aut}_{\mathrm{Green}_H}(\underline R_H),
a\mapsto
\alpha_a^H:=\operatorname{Res}^{C_n}_H(\alpha_a).
\)
Thus the superscript \(g\) in
\(\underline{\HH}_0^g(\underline R_H;\underline M_H)\)
does not mean \(g\) is regarded as an element of \(H\) but denotes the \(g\)-twisted Hochschild construction of
Definition~\ref{def:twisted-HH-construction}, applied in
\(\mathbf{Mack}_H\) using the restricted \(C_n\)-twisting data \(\alpha_g^H\).
\end{remark}

Following \cite[Section~6.1]{AGHKK}, adapted to Green functors, let
\(\operatorname{Perf}_{(\underline R_H,g)}\)
be the exact category whose objects are pairs
\((\underline P,\gamma_{g,\underline P}),
\)
where \(\underline P\) is a finitely generated projective left
\(\underline R_H\)-module and
\(\gamma_{g,\underline P}\colon
\underline P\xrightarrow{\cong}{}^g\underline P\) is an isomorphism of left
\(\underline R_H\)-modules.  A morphism
\(
f\colon(\underline P,\gamma_{g,\underline P})
\to
(\underline Q,\gamma_{g,\underline Q})
\)
is a module map \(f\colon \underline P\to\underline Q\) such that
\(
{}^g f\circ \gamma_{g,\underline P}
=
\gamma_{g,\underline Q}\circ f .
\)
The exact structure is inherited from the exact category of perfect
\(\underline R_H\)-modules: a sequence is exact if
and only if its underlying sequence of modules is exact.

We first record that the change-of-groups functors preserve perfect
modules and are compatible with the additional \(g\)-twisted structure.

\begin{lemma}\label{lem:twisted-functor-extends}
Let \(K\leq H\leq C_n\). Restriction and induction induce exact functors
\[
\operatorname{Res}^H_K:
\operatorname{Perf}_{(\underline R_H,g)}
\longrightarrow
\operatorname{Perf}_{(\underline R_K,g)},
\qquad
\operatorname{Ind}^H_K:
\operatorname{Perf}_{(\underline R_K,g)}
\longrightarrow
\operatorname{Perf}_{(\underline R_H,g)}.
\]
For every \(a\in C_n\), conjugation induces an exact functor
\(
c_a^*:
\operatorname{Perf}_{(\underline R_H,g)}
\longrightarrow
\operatorname{Perf}_{(\underline R_{aHa^{-1}},aga^{-1})}.
\)
Since \(C_n\) is abelian, \(aHa^{-1}=H\) and \(aga^{-1}=g\), so the
target is \(\operatorname{Perf}_{(\underline R_H,g)}\).
\end{lemma}

\begin{proof}
Restriction and conjugation are exact strong symmetric monoidal functors,
so they preserve finite free modules and their retracts. For induction,
the projection formula gives
\(
\operatorname{Ind}_K^H
\bigl(\underline R_K\square\underline A_Y\bigr)
\cong
\underline R_H\square\underline A_{H\times_KY}
\)
for every finite \(K\)-set \(Y\). Hence induction also preserves finite
free modules and their retracts. Thus all three change-of-groups functors
preserve finitely generated projective modules.

It remains to verify compatibility with the chosen \(g\)-twist.
For \(K\leq H\), the maps \(\ell_g^H\) and \(\ell_g^K\) are compatible
with restriction. Hence the induced automorphisms of restricted Green
functors satisfy
\(\operatorname{Res}^H_K(\alpha_g^H)=\alpha_g^K\). This gives a canonical identification
\(
\Phi_{\operatorname{Res}}\colon
\operatorname{Res}^H_K({}^g\underline P)
\xrightarrow{\cong}
{}^g(\operatorname{Res}^H_K\underline P).
\)
Thus restriction sends \((\underline P,\gamma_{g})\) to
\(
\left(
\operatorname{Res}^H_K\underline P,\,
\Phi_{\operatorname{Res}}\circ
\operatorname{Res}^H_K(\gamma_{g})
\right).
\)

For induction, the projection formula gives a natural identification
\(
\Phi_{\operatorname{Ind}}\colon
\operatorname{Ind}^H_K({}^g\underline Q)
\xrightarrow{\cong}
{}^g(\operatorname{Ind}^H_K\underline Q).
\)
Thus induction sends \((\underline Q,\gamma_{g})\) to
\(
\left(
\operatorname{Ind}^H_K\underline Q,\,
\Phi_{\operatorname{Ind}}\circ
\operatorname{Ind}^H_K(\gamma_{g})
\right).
\)

For conjugation, there is a canonical identification
\(
c_a^*(\underline R_H)
\cong
\underline R_{aHa^{-1}}
\)
under which
\(
c_a^*(\alpha_g^H)
=
\alpha_{aga^{-1}}^{aHa^{-1}}.
\)
Consequently, strong monoidality of \(c_a^*\) gives a canonical
isomorphism
\(
\Phi_{c_a}\colon
c_a^*({}^g\underline P)
\xrightarrow{\cong}
{}^{aga^{-1}}(c_a^*\underline P).
\)
Since \(C_n\) is abelian, \(aHa^{-1}=H\) and \(aga^{-1}=g\), so this
becomes
\(
\Phi_{c_a}\colon
c_a^*({}^g\underline P)
\xrightarrow{\cong}
{}^g(c_a^*\underline P).
\)
Thus \(c_a^*\) sends \((\underline P,\gamma_{g})\) to
\(
\left(
c_a^*\underline P,\,
\Phi_{c_a}\circ c_a^*(\gamma_{g})
\right).
\)

On morphisms, the functors are given by applying the underlying
restriction, induction, or conjugation functor. Naturality of \(\Phi_{\operatorname{Res}}\), \(\Phi_{\operatorname{Ind}}\),
and \(\Phi_{c_a}\) gives compatibility with the defining squares for
morphisms in the categories \(\operatorname{Perf}_{(-,g)}\). Exactness follows from the underlying exact functors,
since exact sequences in \(\operatorname{Perf}_{(\underline R_H,g)}\) are
defined on the underlying perfect modules.
\end{proof}

\begin{definition}\label{def:twisted-equivariant-K0}
For \(H\leq C_n\), define
\(
\underline K_0^g(\underline R)(C_n/H)
\coloneqq
K_0\bigl(\operatorname{Perf}_{(\underline R_H,g)}\bigr).
\)

For \(K\leq H\), the restriction and transfer maps are induced by the exact functors of Lemma~\ref{lem:twisted-functor-extends}; explicitly,
\(
\res_K^H[(\underline P,\gamma_{g})]
=
\left[
\left(
\operatorname{Res}_K^H\underline P,\,
\Phi_{\operatorname{Res}}\circ
\operatorname{Res}_K^H(\gamma_{g})
\right)
\right],
\)
and
\(
\tr_K^H[(\underline Q,\gamma_{g})]
=
\left[
\left(
\operatorname{Ind}_K^H\underline Q,\,
\Phi_{\operatorname{Ind}}\circ
\operatorname{Ind}_K^H(\gamma_{g})
\right)
\right].
\)
For \(a\in C_n\), the conjugation map is induced by
\(
c_a^*[(\underline P,\gamma_{g})]
=
\left[
\left(
c_a^*\underline P,\,
\Phi_{c_a}\circ c_a^*(\gamma_{g})
\right)
\right].
\)
\end{definition}

\begin{prop}\label{prop:twisted-K0-Mackey}
The assignment in Definition~\ref{def:twisted-equivariant-K0} defines a
\(C_n\)-Mackey functor. We call it the \emph{\(g\)-twisted
\(K_0\)-Mackey functor} of \(\underline R\).
\end{prop}

\begin{proof}
By Lemma~\ref{lem:twisted-functor-extends}, restriction, transfer, and
conjugation are induced by exact functors between the exact categories \(\operatorname{Perf}_{(\underline R_H,g)}\). Hence they induce homomorphisms on
Grothendieck groups.

For the underlying perfect-module categories, the identity, transitivity,
conjugation, and double-coset relations are induced by the standard
canonical natural isomorphisms among the corresponding composites of
restriction, induction, and conjugation functors. These functors preserve
perfect modules by the first paragraph of the proof of
Lemma~\ref{lem:twisted-functor-extends}.

We claim that these canonical natural isomorphisms lift to the categories
\(\operatorname{Perf}_{(\underline R_H,g)}\). Let \(F\) be any composite of
restriction, induction, and conjugation functors appearing in these
Mackey identities. By Lemma~\ref{lem:twisted-functor-extends}, together
with the projection formula in the induction case, the functor \(F\)
comes equipped with a canonical natural identification
\(
\chi_F(\underline P)\colon
F({}^g\underline P)\xrightarrow{\cong}{}^gF(\underline P),
\)
obtained by composing the corresponding maps
\(\Phi_{\operatorname{Res}}\), \(\Phi_{\operatorname{Ind}}\), and
\(\Phi_{c_a}\).

If \(\eta\colon F_1\xRightarrow{\cong}F_2\) is one of the canonical
natural isomorphisms appearing in the Mackey identities, then the following diagram commutes:
\[\begin{tikzcd}[column sep=3.5em, row sep=1.5em]
F_1(\underline P) 
  \ar[r,"F_1(\gamma_{g})"] 
  \ar[d,"\eta_{\underline P}"'] 
& 
F_1({}^g\underline P) 
  \ar[r,"\chi_{F_1}(\underline P)"] 
  \ar[d,"\eta_{{}^g\underline P}"] 
& 
{}^gF_1(\underline P) 
  \ar[d,"{}^g\eta_{\underline P}"] 
\\
F_2(\underline P) 
  \ar[r,"F_2(\gamma_{g})"'] 
& 
F_2({}^g\underline P) 
  \ar[r,"\chi_{F_2}(\underline P)"'] 
& 
{}^gF_2(\underline P)
\end{tikzcd}\]
The top square commutes by the naturality of \(\eta\) evaluated on the
morphism \(\gamma_{g,\underline P}\). The bottom square commutes because
the standard change-of-groups isomorphisms, including the double-coset
isomorphism, are compatible with the automorphisms \(\alpha_g\) and with
the projection-formula identifications used to define \(\chi\). Hence the
outer rectangle commutes, which precisely means that
\(\eta_{\underline P}\) defines an isomorphism
\(
\bigl(F_1(\underline P),
\chi_{F_1}(\underline P)\circ F_1(\gamma_{g})\bigr)
\cong
\bigl(F_2(\underline P),
\chi_{F_2}(\underline P)\circ F_2(\gamma_{g})\bigr)
\)
in the corresponding category \(\operatorname{Perf}_{(-,g)}\).

Consequently, the canonical natural isomorphisms proving the Mackey
identities for the underlying perfect module categories also prove the
same identities in the exact categories
\(\operatorname{Perf}_{(\underline R_H,g)}\). Passing to Grothendieck
groups gives the Mackey functor identities for
\(\underline K_0^g(\underline R)\).
\end{proof}

 We next record the corresponding compatibility for the twisted
degree-zero Hochschild homology.

\begin{lemma}\label{lem:twisted-HH0-restriction}
Let \(H\leq C_n\), and let \(\underline M\) be an
\((\underline R,\underline R)\)-bimodule.  There is a canonical isomorphism
of \(H\)-Mackey functors
\(
\underline{\HH}_0^g
\bigl(\underline R_H;\underline M_H\bigr)
\cong
\operatorname{Res}^{C_n}_H
\underline{\HH}_0^g(\underline R;\underline M),
\)
natural in both \(\underline R\) and \(\underline M\).  In particular, taking
\(\underline M=\underline R\) gives
\(
\underline{\HH}_0^g(\underline R_H)
\cong
\operatorname{Res}^{C_n}_H\underline{\HH}_0^g(\underline R).
\)
\end{lemma}

\begin{proof}
By the definition of \(g\)-twisted Hochschild homology, there is an
identification
\(
\underline{\HH}_0^g(\underline R;\underline M)
=
\underline{\HH}_0(\underline R;{}^g\underline M).
\)
Since
\(
\operatorname{Res}_H^G(\alpha_g)=\alpha_g^H,
\)
there is a canonical identification of \(\underline R_H\)-bimodules
\(
\operatorname{Res}_H^G({}^g\underline M)
\cong
{}^g(\underline M_H).
\)
The restriction functor
\(\operatorname{Res}_H^G:\mathbf{Mack}_G\to\mathbf{Mack}_H\)
preserves colimits and is strong symmetric monoidal. Applying it to the
coequalizer defining
\(\underline{\HH}_0(\underline R;{}^g\underline M)\) therefore gives
the coequalizer defining
\(\underline{\HH}_0(\underline R_H;{}^g\underline M_H)\). Hence
\(
\underline{\HH}_0
\bigl(\underline R_H;
\operatorname{Res}_H^G({}^g\underline M)
\bigr)
\cong
\operatorname{Res}_H^G
\underline{\HH}_0(\underline R;{}^g\underline M).
\)
Using the preceding identification of coefficients, this becomes
\(
\underline{\HH}_0^g
\bigl(\underline R_H;\underline M_H\bigr)
\cong
\operatorname{Res}_H^G
\underline{\HH}_0^g(\underline R;\underline M).
\)
The case \(\underline M=\underline R\) follows immediately. Naturality
in \(\underline R\) and \(\underline M\) follows from the naturality of
the defining coequalizer and of the identification
\(\operatorname{Res}_H^G({}^g\underline M)\cong{}^g(\underline M_H)\).
\end{proof}

\begin{remark}\label{rem:twisted-orbit-eval}
The natural isomorphism in Lemma~\ref{lem:twisted-HH0-restriction} can
be evaluated explicitly on orbits. For \(K\le H\le C_n\), evaluating at the
\(H\)-orbit \(H/K\) gives a canonical isomorphism
\[
\vartheta_K^{H,g}\colon
\underline{\HH}_0^g(\underline R_H;\underline M_H)(H/K)
\xrightarrow{\cong}
\underline{\HH}_0^g(\underline R;\underline M)(C_n/K),
\] where
\(
\underline{\HH}_0^g(\underline R_H;\underline M_H)(H/K)
\cong
\bigl(
\operatorname{Res}^{C_n}_H
\underline{\HH}_0^g(\underline R;\underline M)
\bigr)(H/K)
\cong
\underline{\HH}_0^g(\underline R;\underline M)(C_n/K)
\).

In particular, when \(K=H\), we write
\[
\vartheta_H^g:=\vartheta_H^{H,g}\colon
\underline{\HH}_0^g(\underline R_H;\underline M_H)(H/H)
\xrightarrow{\cong}
\underline{\HH}_0^g(\underline R;\underline M)(C_n/H).
\]

Applying the same restriction argument as in
Lemma~\ref{lem:twisted-HH0-restriction} to the inclusion \(K\le H\), and using
the canonical identifications
\(
\underline R_K\cong \operatorname{Res}_K^H\underline R_H\) and
\(\underline M_K\cong \operatorname{Res}_K^H\underline M_H,
\)
we obtain a canonical isomorphism
\[
\iota_K^{H,g}\colon
\underline{\HH}_0^g(\underline R_K;\underline M_K)(K/K)
\xrightarrow{\cong}
\underline{\HH}_0^g(\underline R_H;\underline M_H)(H/K).
\]

These maps are compatible with the canonical orbit identifications. Namely,
the composite
\[
\underline{\HH}_0^g(\underline R_K;\underline M_K)(K/K)
\xrightarrow{\ \iota_K^{H,g}\ }
\underline{\HH}_0^g(\underline R_H;\underline M_H)(H/K)
\xrightarrow{\ \vartheta_K^{H,g}\ }
\underline{\HH}_0^g(\underline R;\underline M)(C_n/K)
\]
is the  map
\(
\vartheta_K^g:=\vartheta_K^{K,g}\colon
\underline{\HH}_0^g(\underline R_K;\underline M_K)(K/K)
\xrightarrow{\cong}
\underline{\HH}_0^g(\underline R;\underline M)(C_n/K).
\)
Equivalently,
\(
\vartheta_K^g=\vartheta_K^{H,g}\circ \iota_K^{H,g}.
\)
\end{remark}

Since the isomorphism in Lemma~\ref{lem:twisted-HH0-restriction} is an
isomorphism of \(H\)-Mackey functors, the maps
\(\vartheta_K^{H,g}\) are compatible with restriction and transfer in the
\(H\)-Mackey functor variables.  They are also compatible with conjugation,
where conjugation sends the \(g\)-twist to the \(aga^{-1}\)-twist; since
\(C_n\) is abelian, this is again the \(g\)-twist.

We now define the degree-zero \(g\)-twisted equivariant Dennis trace.
For \(H\leq C_n\) and a finitely generated projective left
\(\underline R_H\)-module \(\underline P\), write
\(
\operatorname{tr}_{\mathrm{HS},\underline P}^{g,H}\colon
\underline{\Hom}_{\underline R_H}
   (\underline P,{}^g\underline P)
\to
\underline{\HH}_0^g(\underline R_H)
\)
for the \(H\)-Mackey-functor-valued \(g\)-twisted
Hattori--Stallings trace.  For \(L\leq H\), we denote its component at
\(H/L\) by
\(
\bigl(
\operatorname{tr}_{\mathrm{HS},\underline P}^{g,H}
\bigr)_{H/L}.
\)
We suppress \(\underline P\) from the notation when it is clear from
the context.

If
\(
(\underline P,\gamma_g)\in
\operatorname{Perf}_{(\underline R_H,g)},
\)
then
\(
\gamma_g\in
\underline{\Hom}_{\underline R_H}
(\underline P,{}^g\underline P)(H/H),
\)
and therefore
\(
\bigl(
\operatorname{tr}_{\mathrm{HS},\underline P}^{g,H}
\bigr)_{H/H}(\gamma_g)
\in
\underline{\HH}_0^g(\underline R_H)(H/H).
\)
Composing with \(\vartheta_H^g\) gives an element of
\(\underline{\HH}_0^g(\underline R)(C_n/H)\).

\begin{definition}\label{def:twisted-equivariant-Dennis}
For each orbit \(C_n/H\), the \emph{\(g\)-twisted equivariant Dennis trace} at \(C_n/H\) is the map
\[
\bigl(\tr_{\mathrm{Dennis}}^g\bigr)_{C_n/H}\colon
\underline K_0^g(\underline R)(C_n/H)
\longrightarrow
\underline{\HH}_0^g(\underline R)(C_n/H)
\]
defined on a class represented by a \(g\)-twisted perfect
\(\underline R_H\)-module
\((\underline P,\gamma_{g})\) by
\[
\bigl(\tr_{\mathrm{Dennis}}^g\bigr)_{C_n/H}
\bigl([(\underline P,\gamma_{g})]\bigr)
:=
\vartheta_H^g
\left(
(\tr_{\mathrm{HS}}^{g,H})_{H/H}
(\gamma_{g})
\right).
\]
This is well-defined on \(K_0\) by
\cite[Theorem~6.1.1 and Remark~6.1.3]{AGHKK}, applied in
\(\mathcal V=s\mathbf{Mack}_H\), since the displayed formula is the
degree-zero value of the twisted Dennis trace constructed there.

\end{definition}

To show that the maps in Definition~\ref{def:twisted-equivariant-Dennis}
assemble to a morphism of Mackey functors, it remains to compare the
twisted Hattori--Stallings classes under restriction, induction, and
conjugation of \(g\)-twisted perfect modules.

The compatibility with induction is not formal from strong symmetric
monoidality, since induction is only oplax monoidal. We first record the
finite-free case.

\begin{lemma}\label{lem:twisted-HS-trace-ind-free}
Let \(K\leq H\leq C_n\), let \(Y\) be a finite \(K\)-set, and set
\(
\underline F:=\underline R_K\square\underline A_Y.
\)
For every morphism \(f\) in \(\underline{\Hom}_{\underline R_K}(\underline F,{}^g\underline F)(K/K)\), the following diagram commutes:
\[
\begin{tikzcd}[column sep=2.7em,row sep=2em]
\underline{\Hom}_{\underline R_K}
(\underline F,{}^g\underline F)(K/K)
  \ar[r,"{(\tr_{\mathrm{HS}}^{g,K})_{K/K}}"]
  \ar[d,"{f\mapsto
  \Phi_{\operatorname{Ind}}\circ\operatorname{Ind}_K^H(f)}"']
&
\underline{\HH}_0^g(\underline R_K)(K/K)
  \ar[d,"{\tr_K^H\circ\iota_K^{H,g}}"]
\\
\underline{\Hom}_{\underline R_H}
\bigl(\operatorname{Ind}_K^H\underline F,
{}^g\operatorname{Ind}_K^H\underline F\bigr)(H/H)
  \ar[r,"{(\tr_{\mathrm{HS}}^{g,H})_{H/H}}"']
&
\underline{\HH}_0^g(\underline R_H)(H/H).
\end{tikzcd}
\]
\end{lemma}

\begin{proof}
The representable Mackey functor \(\underline A_Y\) is self-dual in
\(\mathbf{Mack}_K\), and hence the finite free module
\(\underline F=\underline R_K\square\underline A_Y\) is left
dualizable over \(\underline R_K\), with left dual
\(
{}^*\underline F
\cong
\underline A_Y\square\underline R_K.
\)
By the projection formula,
\(
\operatorname{Ind}_K^H\underline F
\cong
\underline R_H\square\underline A_{H\times_KY},
\)
which is likewise finite free and left dualizable over
\(\underline R_H\).

In the Burnside category \(\mathcal B_K\), the coevaluation and
evaluation exhibiting \(Y\) as self-dual are represented by
\(
K/K\xleftarrow{\ }Y\xrightarrow{\Delta}Y\times Y,
Y\times Y\xleftarrow{\Delta}Y\xrightarrow{\ }K/K.
\)
Applying \(H\times_K-\), composing with \(H/K\to H/H\), and using the
canonical map
\(
H\times_K(Y\times Y)\to
(H\times_KY)\times(H\times_KY),
[h,(y_1,y_2)]\mapsto([h,y_1],[h,y_2]),
\)
gives the coevaluation and evaluation used to compute the trace after
induction. The twist changes only the coefficient action, and the
identification
\(
\Phi_{\operatorname{Ind}}\colon
\operatorname{Ind}_K^H({}^g\underline F)
\xrightarrow{\cong}
{}^g(\operatorname{Ind}_K^H\underline F)
\)
identifies the induced twisted coefficient action with the
\(H\)-level one. Consequently, the same span calculation, followed by
the twisted Hochschild quotient, gives
\(
(\tr_{\mathrm{HS}}^{g,H})_{H/H}
\bigl(\Phi_{\operatorname{Ind}}\circ
\operatorname{Ind}_K^H(f)\bigr)
=
\tr_K^H\left(
\iota_K^{H,g}\left(
(\tr_{\mathrm{HS}}^{g,K})_{K/K}(f)
\right)\right).
\)
This is the asserted commutativity.
\end{proof}

\begin{prop}\label{prop:twisted-HS-Mackey-operations}
Let \(K\le H\le C_n\). The \(g\)-twisted Hattori--Stallings trace is
compatible with restriction, transfer, and conjugation in the following sense.

If
\(
(\underline P,\gamma_{g,\underline P})\in
\Perf_{(\underline R_H,g)},
\)
then
\[
\res_K^H
\left(
(\tr_{\mathrm{HS}}^{g,H})_{H/H}
(\gamma_{g,\underline P})
\right)
=
\iota_K^{H,g}
\left(
(\tr_{\mathrm{HS}}^{g,K})_{K/K}
\left(
\Phi_{\operatorname{Res}}\circ
\operatorname{Res}_K^H(\gamma_{g,\underline P})
\right)
\right).
\]

If
\(
(\underline Q,\gamma_{g,\underline Q})\in
\Perf_{(\underline R_K,g)},
\)
then
\[
\tr_K^H
\left(
\iota_K^{H,g}
\left(
(\tr_{\mathrm{HS}}^{g,K})_{K/K}
(\gamma_{g,\underline Q})
\right)
\right)
=
(\tr_{\mathrm{HS}}^{g,H})_{H/H}
\left(
\Phi_{\operatorname{Ind}}\circ
\operatorname{Ind}_K^H(\gamma_{g,\underline Q})
\right).
\]

Moreover, for \(a\in C_n\), if
\(
(\underline P,\gamma_{g,\underline P})\in
\Perf_{(\underline R_H,g)},
\)
then
\[
c_a^*
\left(
(\tr_{\mathrm{HS}}^{g,H})_{H/H}
(\gamma_{g,\underline P})
\right)
=
(\tr_{\mathrm{HS}}^{g,aHa^{-1}})_{aHa^{-1}/aHa^{-1}}
\left(
\Phi_{c_a}\circ c_a^*(\gamma_{g,\underline P})
\right).
\]
Since \(C_n\) is abelian, \(aHa^{-1}=H\).
\end{prop}

\begin{proof}
The restriction functor \(\operatorname{Res}_K^H\) is strong symmetric
monoidal and is compatible with the chosen \(g\)-twists through
\(\Phi_{\operatorname{Res}}\). It therefore preserves the duality maps
and the twisted Hochschild coequalizer used to define the
Hattori--Stallings trace. Evaluating the resulting comparison at \(K/K\)
and using \(\iota_K^{H,g}\) gives the restriction formula.

We next prove the transfer formula. Choose a finite free
\(\underline R_K\)-module
\(
\underline F=\underline R_K\square\underline A_Y
\)
and a splitting
\(
\underline Q\xrightarrow{i}\underline F
\xrightarrow{p}\underline Q, p\circ i=\id_{\underline Q}.
\)
Define
\(
\widetilde\gamma_g
:={}^{g}i\circ\gamma_{g,\underline Q}\circ p
\colon
\underline F\longrightarrow{}^g\underline F.
\)
By evaluating the equivariant twisted cyclicity (Proposition~\ref{prop:eq-algebraic-sliding}) at the orbit $K/K$ for the module morphisms
\(
p:\underline F\to\underline Q
\)
and
\(
{}^gi\circ\gamma_{g,\underline Q}:
\underline Q\to{}^g\underline F,
\)
gives
\(
(\tr_{\mathrm{HS}}^{g,K})_{K/K}(\gamma_{g,\underline Q})
=
(\tr_{\mathrm{HS}}^{g,K})_{K/K}(\widetilde\gamma_g),
\)
because
\(
{}^gp\circ{}^gi\circ\gamma_{g,\underline Q}
=\gamma_{g,\underline Q}.
\)
After applying induction, the naturality and coherence of
\(\Phi_{\operatorname{Ind}}\), together with the induced splitting,
give in the same way
\(
(\tr_{\mathrm{HS}}^{g,H})_{H/H}
\bigl(\Phi_{\operatorname{Ind}}\circ
\operatorname{Ind}_K^H(\gamma_{g,\underline Q})\bigr)
=
(\tr_{\mathrm{HS}}^{g,H})_{H/H}
\bigl(\Phi_{\operatorname{Ind}}\circ
\operatorname{Ind}_K^H(\widetilde\gamma_g)\bigr).
\)
Applying Lemma~\ref{lem:twisted-HS-trace-ind-free} to
\(\widetilde\gamma_g\) now gives
\(
(\tr_{\mathrm{HS}}^{g,H})_{H/H}
\bigl(\Phi_{\operatorname{Ind}}\circ
\operatorname{Ind}_K^H(\gamma_{g,\underline Q})\bigr)=
\tr_K^H\left(
\iota_K^{H,g}\left(
(\tr_{\mathrm{HS}}^{g,K})_{K/K}
(\gamma_{g,\underline Q})
\right)\right),
\)
which is the transfer formula.

Finally, conjugation by \(a\) is a strong symmetric monoidal
equivalence and is compatible with the twisting through
\(\Phi_{c_a}\). It therefore transports the twisted
Hattori--Stallings trace over \(\underline R_H\) to the trace over
\(\underline R_{aHa^{-1}}\), with twist \(aga^{-1}\). Since \(C_n\) is
abelian, this is again the \(g\)-twist, giving the conjugation formula.
\end{proof}

\begin{thm}\label{thm:twisted-equivariant-Dennis}
The homomorphisms of Definition~\ref{def:twisted-equivariant-Dennis} assemble
to a morphism of \(C_n\)-Mackey functors
\(
\tr_{\mathrm{Dennis}}^g\colon
\underline K_0^g(\underline R)
\to
\underline{\HH}_0^g(\underline R).
\)
We call this morphism the \emph{\(g\)-twisted equivariant Dennis trace in
degree zero}.
\end{thm}

\begin{proof}
For each subgroup \(H\le C_n\), let
\(
\tr_{\mathrm{Dennis}}^{g,H}\colon
K_0(\Perf_{(\underline R_H,g)})
\to
\underline{\HH}_0^g(\underline R_H)(H/H)
\)
denote the homomorphism induced by
\(
[(\underline P,\gamma_{g})]
\mapsto
(\tr_{\mathrm{HS}}^{g,H})_{H/H}
(\gamma_{g}).
\)
This is well-defined on \(K_0\) by the exact-sequence additivity of
the \(g\)-twisted Hattori--Stallings trace
(Proposition~\ref{prop:twisted-additivity}).

By Definition~\ref{def:twisted-equivariant-Dennis}, the orbitwise map is precisely
\(
\bigl(\tr_{\mathrm{Dennis}}^g\bigr)_{C_n/H} = \vartheta_H^g\circ\tr_{\mathrm{Dennis}}^{g,H}.
\)
Proposition~\ref{prop:twisted-HS-Mackey-operations}, together with the restriction, transfer, and conjugation compatibility of the canonical identifications \(\vartheta_H^g\), shows that these maps commute with all Mackey structure maps. Hence they assemble to the asserted morphism of \(C_n\)-Mackey functors.
\end{proof}

The following result identifies the \(g\)-twisted equivariant Dennis trace
with the \(g\)-twisted bicategorical trace of the identity \(2\)-cell on
representatives arising from the chosen \(C_n\)-twist structure. Thus it
provides a shadow-theoretic interpretation, at degree \(0\), of the
twisted trace map of \cite{AGHKK}.

\begin{cor}
    \label{corollary:twDennis-equals-twTrace-levelwise}
Let \(H\le C_n\), and let
\(
(\underline P,\gamma_{g})
\in
\Perf_{(\underline R_H,g)}
\)
be a \(g\)-twisted perfect \(\underline R_H\)-module such that
\(
\gamma_{g}\colon
\underline P\xrightarrow{\cong}{}^g\underline P
\)
is the isomorphism corresponding, under
Proposition~\ref{prop:gM-vs-gammaM}, to the \(g\)-component of a chosen
\(C_n\)-twist structure on \(\underline P\).  Then the following equality holds
in \(\underline{\HH}^{\,g}_0(\underline R)(C_n/H)\):
\[
\bigl(\tr_{\mathrm{Dennis}}^{\,g}\bigr)_{C_n/H}
\bigl([(\underline P,\gamma_{g})]\bigr)
=
\vartheta_H^{\,g}
\left(
\bigl(\tr^{g,H}(\id_{\underline P})\bigr)_{H/H}([1])
\right).
\]
Here \(\tr^{g,H}(\id_{\underline P})\) denotes the \(g\)-twisted
bicategorical trace computed in the bicategory of bimodules over
\(H\)-Green functors.
\end{cor}

\begin{proof}
This follows by applying
Theorem~\ref{thm:tw-eq-bicat-vs-HS} in
\(\mathbf{Mack}_H\) to \(f=\id_{\underline P}\), and then composing
with \(\vartheta_H^g\).
\end{proof}

\subsection{Reduced twisted Dennis traces and Teichm\"uller maps}
In this subsection, we extend the construction of the preceding
subsection to define a reduced degree-zero twisted equivariant Dennis
trace and show that it recovers the \(m=1\)
Teichm\"uller maps of \cite[Theorem~6.13]{BGHL}.

For an ordinary commutative ring \(A\) and \(a\in A\), let
\(m_a\colon A\to A\), \(x\mapsto ax\), be the \(A\)-linear
endomorphism given by multiplication by \(a\). For every \(d\geq 1\),
one has
\(m_a^d=m_{a^d}\) and
\(\operatorname{tr}_{\mathrm{HS}}(m_a^d)=a^d\).
Thus the Hattori--Stallings traces of the iterates of \(m_a\) recover
the ghost components of the Teichm\"uller element
\([a]\in{\mathbb W}(A)\) in the big Witt ring.
This observation has a twisted equivariant analogue.

Recall that a \(C_n\)-Tambara functor is a commutative
\(C_n\)-Green functor equipped with multiplicative norm maps satisfying
the Tambara compatibility and distributivity relations
\cite{Tambara}. Let \(\underline R\) be a \(C_n\)-Tambara functor,
and fix the generator \(g=e^{2\pi i/n}\) of \(C_n\). For
\(H\leq C_n\), write
\(N_H^{C_n}\colon\underline R(C_n/H)\to
\underline R(C_n/C_n)\)
for the internal Tambara norm.

Specializing \cite[Definition~6.2]{BGHL} to \(m=1\) gives
\(
\underline{\mathbb W}_{C_n}(\underline R)
=
\underline{\HH}^{C_n}_{C_n}(\underline R)_0.
\)
For the chosen generator \(g\), the right-hand side is canonically
identified with \(\underline{\HH}_0^g(\underline R)\).
For each \(H\leq C_n\), let
\(
\tau_H\colon
\underline R(C_n/H)
\longrightarrow
\underline{\mathbb W}_{C_n}(\underline R)(C_n/C_n)
\)
denote the Teichm\"uller map constructed in
\cite[Theorem~6.13]{BGHL} from the external norm of
\cite[Proposition~6.12]{BGHL}. Composing the external norm with the
Tambara structure map gives the internal norm \(N_H^{C_n}\).
Consequently, under the preceding identification,
\begin{equation}\label{eq:BGHL-Teichmuller}
\tau_H(r)
=
\bigl[N_H^{C_n}(r)\bigr].
\end{equation}
Here the brackets denote the image under the component at
\(C_n/C_n\) of the quotient map \(q^{(g)}\) from
\eqref{eq:qM-g}, specialized to \(\underline M=\underline R\).

For \(K\leq C_n\) and \(x\in\underline R(C_n/K)=\underline R_K(K/K)\), let \(f_x^K\colon\underline R_K\to{}^g\underline R_K\) be the left $\underline R_K$-module map corresponding to \(x\) under the canonical isomorphism
\begin{equation*}
\underline{\Hom}_{\underline R_K}(\underline R_K,{}^g\underline R_K)(K/K)
\cong
{}^g\underline R_K(K/K)
\cong
\underline R(C_n/K).
\end{equation*}
Thus \((f_x^K)_{K/K}(1)=x\), where \(1\) is the multiplicative unit of \(\underline R_K(K/K)\). More generally, if \(L\leq K\) and \(a\in\underline R_K(K/L)\), then \((f_x^K)_{K/L}(a) = \alpha_g^K(a)\res_L^K(x)\).  In particular, for
\(r\in\underline R(C_n/H)\), the element \(N_H^{C_n}(r)\) determines
the \(\underline R\)-module map
\(
f_{N_H^{C_n}(r)}^{C_n}\colon
\underline R\to{}^g\underline R.
\)

The map \(f_x^K\) need not be invertible and hence does not in general define an object of \(\Perf_{(\underline R_K,g)}\). We thus extend the trace to arbitrary maps into the twist. For $H\leq C_n$, let $\Perf(\underline R_H)$ denote the exact category of finitely generated projective left $\underline R_H$-modules. Define $\End^g(\underline R_H)$ to be the exact category whose objects are pairs $(\underline P,f)$, where $\underline P\in\Perf(\underline R_H)$ and $f\colon\underline P\to{}^g\underline P$. Morphisms $u\colon(\underline P,f)\to(\underline Q,h)$ in this category are $\underline R_H$-module maps satisfying ${}^g u\circ f=h\circ u$. A sequence in $\End^g(\underline R_H)$ is exact if its underlying sequence in $\Perf(\underline R_H)$ is exact. 

Under the canonical identification ${}^g\underline R_H \square_{\underline R_H}\underline P \cong {}^g\underline P$, the category $\End^g(\underline R_H)$ is precisely the degree-zero Green-functor analogue of the parametrized endomorphism category $\End(\underline R_H,{}^g\underline R_H)$ from Section~5.2. The zero-endomorphism functor \(z_H\colon \Perf(\underline R_H)\to\End^g(\underline R_H)\), given by \(\underline P\mapsto(\underline P,0)\), is split by the forgetful functor. We define the associated reduced Grothendieck group as the cokernel:
\begin{equation*}
\widetilde K_0\bigl(\End^g(\underline R_H)\bigr)
:=
\coker\!\left(
K_0\bigl(\Perf(\underline R_H)\bigr)
\xrightarrow{\,K_0(z_H)\,}
K_0\bigl(\End^g(\underline R_H)\bigr)
\right).
\end{equation*}
Orbitwise, we set \(\widetilde{\underline K}_0(\End^g(\underline R))(C_n/H) := \widetilde K_0(\End^g(\underline R_H))\).

\begin{prop}\label{prop:reduced-twisted-Dennis-Mackey}
The groups
\(
\widetilde{\underline K}_0
(\End^g(\underline R))(C_n/H)
\),
for \(H\leq C_n\), assemble to a \(C_n\)-Mackey functor, and the maps
\begin{equation*}
\begin{aligned}
\bigl(\widetilde{\tr}_{\mathrm{Dennis}}^g\bigr)_{C_n/H}\colon
\widetilde{\underline K}_0
(\End^g(\underline R))(C_n/H)
&\longrightarrow
\underline{\HH}_0^g(\underline R)(C_n/H),\quad 
[(\underline P,f)]
&\longmapsto
\vartheta_H^g
\left(
(\tr_{\mathrm{HS}}^{g,H})_{H/H}(f)
\right)
\end{aligned}
\end{equation*}
form a morphism of \(C_n\)-Mackey functors.
\end{prop}

\begin{proof}
For each restriction, induction, or conjugation functor \(F\), let
\(
\Phi_F\colon F({}^g\underline P)\xrightarrow{\cong}{}^gF(\underline P)
\)
be the  twist isomorphism from
Lemma~\ref{lem:twisted-functor-extends}. This defines an exact functor on parametrized endomorphisms via \(F(\underline P,f) = \bigl(F(\underline P),\Phi_F\circ F(f)\bigr)\). The proof of
Proposition~\ref{prop:twisted-K0-Mackey} applies without requiring
\(f\) to be invertible and gives a \(C_n\)-Mackey functor
\(
H\mapsto K_0(\End^g(\underline R_H)).
\)
The maps \(K_0(z_H)\) commute with restriction, induction, and
conjugation, and hence form a morphism of \(C_n\)-Mackey functors.
Their cokernels therefore form the asserted reduced
\(C_n\)-Mackey functor.

By the exact-sequence additivity of the twisted
Hattori--Stallings trace
(Proposition~\ref{prop:twisted-additivity}), the orbitwise traces
induce homomorphisms
\[
K_0\bigl(\End^g(\underline R_H)\bigr)
\longrightarrow
\underline{\HH}_0^g(\underline R_H)(H/H).
\]
They factor through the reduced groups because
\(
\tr_{\mathrm{HS}}^{g,H}(0)=0.
\) 
The proof of
Proposition~\ref{prop:twisted-HS-Mackey-operations} applies without change
to arbitrary maps \(f\colon\underline P\to{}^g\underline P\).
Together with the compatibility of \(\vartheta_H^g\), these formulas
show that the orbitwise maps form a morphism of Mackey functors.
\end{proof}

\begin{prop}\label{prop:teichmuller-as-twisted-Dennis}
Let \(H\leq C_n\) and
\(r\in\underline R(C_n/H)\). The reduced \(g\)-twisted Dennis trace
of the norm-induced map
\(
f_{N_H^{C_n}(r)}^{C_n}\colon
\underline R\to{}^g\underline R
\)
is the Teichm\"uller class \(\tau_H(r)\):
\(
\bigl(\widetilde{\tr}_{\mathrm{Dennis}}^g\bigr)_{C_n/C_n}
\left[
\left(
\underline R,
f_{N_H^{C_n}(r)}^{C_n}
\right)
\right]
=
\tau_H(r).
\)
\end{prop}

\begin{proof}
For \(K\leq C_n\) and
\(x\in\underline R(C_n/K)\), consider the map
\(f_x^K\colon\underline R_K\to{}^g\underline R_K\).
Under the identification
\(
\bigl(
{}^*\underline R_K
\square_{\underline R_K}
{}^g\underline R_K
\bigr)(K/K)
\cong
{}^g\underline R_K(K/K),
\) one has
\(
\bigl(
\delta_{\underline R_K,{}^g\underline R_K}^{-1}
\bigr)_{K/K}(f_x^K)
=
x.
\)
For the module \(\underline P=\underline R_K\), the map
\(\pi^{(g)}\) in the definition of the twisted
Hattori--Stallings trace is the twisted Hochschild quotient map.
Consequently,
\(
\bigl(\tr_{\mathrm{HS}}^{g,K}\bigr)_{K/K}(f_x^K)
=
[x].
\)
Taking \(K=C_n\) and \(x=N_H^{C_n}(r)\), and using
\(\vartheta_{C_n}^g=\id\), gives
\[
\bigl(\widetilde{\tr}_{\mathrm{Dennis}}^g\bigr)_{C_n/C_n}
\left[
\left(
\underline R,
f_{N_H^{C_n}(r)}^{C_n}
\right)
\right]
=
[N_H^{C_n}(r)].
\]
By \eqref{eq:BGHL-Teichmuller}, the right-hand side is
\(\tau_H(r)\).
\end{proof}

\section{Properties of the Twisted Bicategorical Trace}

In this section, we establish properties of the \(g\)-twisted
bicategorical trace, including twisted tightening, cyclicity, additivity,
and multiplicativity. These results generalize the corresponding
properties of the twisted Hattori--Stallings trace.

Throughout this section, we fix a group \(G\) and an element \(g\in G\).
Let \((\mathcal B,\sh{-})\) be a shadowed bicategory equipped with
\(G\)-twisting data, and let
\(
\mathcal B^{G\text{-tw}}
\)
be the associated bicategory of \(G\)-twists
(Definition~\ref{def:bicat_of_G_twists}). We work in
\(\mathcal B^{G\text{-tw}}\), equipped with the induced \(g\)-twisted
shadow
$
{}^g\!\sh{-}$
of Definition~\ref{def:g-twisted-shadow}.

\begin{prop}[Twisted tightening]\label{prop:tw-tightening}
Let \(M\in \mathcal B^{G\text{-tw}}(R,S)\) be a left dualizable
\(1\)-cell, with chosen left dual denoted by
\({}^*M\in \mathcal B^{G\text{-tw}}(S,R)\). Let
\(
Q,Q'\in \mathcal B^{G\text{-tw}}(S,S),
P,P'\in \mathcal B^{G\text{-tw}}(R,R)
\)
be \(1\)-cells, and let
\(
f\colon M\odot Q \Rightarrow P\odot M,
h\colon Q'\Rightarrow Q,
k\colon P\Rightarrow P'
\)
be \(2\)-cells in \(\mathcal B^{G\text{-tw}}\). Then the \(g\)-twisted
trace satisfies
\[
{}^g\sh{k}\circ \tr^g(f)\circ {}^g\sh{h}
=
\tr^g\!\left(
(k\odot \id_M)\circ f\circ(\id_M\odot h)
\right).
\]
\end{prop}
\begin{proof}
Expand the definition of \(\tr^g(f)\).  Functoriality of
\({}^g\!\sh{-}\), together with naturality of the coevaluation, evaluation,
and twisted cyclicity isomorphism, gives
\begin{align*}
{}^g\sh{k}\circ \tr^g(f)\circ{}^g\sh{h}
&={}^{g}\sh{\bar\epsilon\odot\id_{P'}}
  \circ{}^g\theta
  \circ{}^g\sh{
     \id_{{}^*M}\odot
     \bigl((k\odot\id_M)\circ f\circ(\id_M\odot h)\bigr)}
  \circ{}^g\sh{\bar\eta\odot\id_{Q'}}.
\end{align*}
This is the defining composite for the trace on the right-hand side.
\end{proof}

To prove the twisted sliding property, we first recall the mate correspondence for 2-cells in a bicategory of $G$-twists, which arises from the evaluation and coevaluation maps of dualizable 1-cells.

\begin{lemma}
\label{lem:equivariant_mate_correspondence}
Let
\(M\in \mathcal B^{G\text{-tw}}(R,S)\) be left dualizable, with chosen
left dual
\(
{}^*M\in \mathcal B^{G\text{-tw}}(S,R).\)
For any \(1\)-cells
\(
Q\in \mathcal B^{G\text{-tw}}(S,T),
P\in \mathcal B^{G\text{-tw}}(R,U),
N\in \mathcal B^{G\text{-tw}}(U,T),
\)
there is a natural bijection
\[
\operatorname{Hom}_{\mathcal B^{G\text{-tw}}(R,T)}(M\odot Q,\;P\odot N)
\;\xrightarrow{\;\cong\;}\;
\operatorname{Hom}_{\mathcal B^{G\text{-tw}}(S,T)}
  (Q,\;{}^*M\odot P\odot N).
\]
It sends a \(2\)-cell \(f\colon M\odot Q \Rightarrow P\odot N\) to its
left mate
\(
f^\sharp
:=
(\id_{{}^*M}\odot f)\circ(\bar\eta\odot\id_Q),
\)
and the inverse sends a \(2\)-cell
\(a\colon Q\Rightarrow{}^*M\odot P\odot N\) to
\(
(\bar\epsilon\odot\id_{P\odot N})\circ(\id_M\odot a).
\)
\end{lemma}

\begin{proof}
The two assignments are inverse by the triangle identities for the dual pair
\((M,{}^*M)\), and naturality follows from functoriality of horizontal
composition.
\end{proof}

\begin{prop}[Twisted sliding]
\label{prop:twisted_sliding_general}
Let $M\in\mathcal B^{G\text{-tw}}(R,S)$ and
$N\in\mathcal B^{G\text{-tw}}(U,T)$ be left dualizable $1$-cells,
with chosen left duals ${}^*M$ and ${}^*N$.
Let
\(
Q\in\mathcal B^{G\text{-tw}}(S,T),
P\in\mathcal B^{G\text{-tw}}(R,U),
K\in\mathcal B^{G\text{-tw}}(T,S),\) and
\(L\in\mathcal B^{G\text{-tw}}(U,R)
\) be 1-cells, and let
\(
f:M\odot Q\Rightarrow P\odot N,
h:N\odot K\Rightarrow L\odot M
\)
be $2$-cells.
Suppressing associativity constraints, define
\[
F := (\id_L\odot f)\circ (h\odot \id_Q)
:\;
N\odot (K\odot Q)\Rightarrow (L\odot P)\odot N,
\]
and
\[
 H := (\id_P\odot h)\circ (f\odot \id_K)
:\;
M\odot (Q\odot K)\Rightarrow (P\odot L)\odot M.
\]
Then the square
\[
\begin{tikzcd}[row sep=1.5em, column sep=3em]
{}^g\sh{K \odot Q}_T 
  \arrow[r, "\operatorname{tr}_{\mathcal{B},N}^g(F)"] 
  \arrow[d, "{}^g\theta_{K,Q}"'] 
& {}^g\sh{L \odot P}_U
  \arrow[d, "{}^g\theta_{L,P}"] 
\\
{}^g\sh{Q \odot K}_S 
  \arrow[r, "\operatorname{tr}_{\mathcal{B},M}^g( H)"'] 
& {}^g\sh{P \odot L}_R
\end{tikzcd}
\]
commutes.
\end{prop}

\begin{proof} 
To prove commutativity, we expand both paths and reduce them to an identical sequence of 2-cells. Throughout the proof, we suppress associativity constraints.

Let $M$ and $N$ be left dualizable 1-cells  with chosen left duals ${}^*M$ and ${}^*N$. By Lemma~\ref{lem:equivariant_mate_correspondence}, the 2-cells $f$ and $h$ uniquely correspond to their left mates:
\begin{align*}
f^\sharp := (\id_{{}^*M} \odot f) \circ (\bar{\eta}_M \odot \id_Q) \;\colon\; Q \Rightarrow {}^*M \odot P \odot N, \quad
h^\sharp := (\id_{{}^*N} \odot h) \circ (\bar{\eta}_N \odot \id_K) \;\colon\; K \Rightarrow {}^*N \odot L \odot M
\end{align*}
and can be recovered via the substitution relations:
\begin{equation}\label{eq:mate_substitutions}
\begin{aligned}
f &= (\bar{\epsilon}_M \odot \id_{P \odot N}) \circ (\id_M \odot f^\sharp), \quad
\quad h = (\bar{\epsilon}_N \odot \id_{L \odot M}) \circ (\id_N \odot h^\sharp).
\end{aligned}
\end{equation}

Let $P_1$ denote the composite ${}^g\theta_{L,P} \circ \operatorname{tr}_{\mathcal{B},N}^g(F)$. Expanding the trace of $F = (\id_L \odot f) \circ (h \odot \id_Q)$ via Definition~\ref{def:twisted-bicat-trace} yields:
\[
P_1 = {}^g\theta_{L,P} \circ {}^g\sh{\bar{\epsilon}_N \odot \id_{L \odot P}} \circ {}^g\theta_{{}^*N \odot L \odot P, N} \circ {}^g\sh{\id_{{}^*N} \odot F} \circ {}^g\sh{\bar{\eta}_N \odot \id_{K \odot Q}}.
\]
Substitute $F$, and recognize that the sequence inserting $N \odot {}^*N$ and applying $h$ is exactly the mate $h^\sharp$. After substituting $f$ via \eqref{eq:mate_substitutions}, we obtain:
\[
P_1 = {}^g\theta_{L,P} \circ {}^g\sh{\bar{\epsilon}_N \odot \id_{L \odot P}} \circ {}^g\theta_{{}^*N \odot L \odot P, N} \circ {}^g\sh{\id_{{}^*N \odot L} \odot \bar{\epsilon}_M \odot \id_{P \odot N}} \circ {}^g\sh{h^\sharp \odot f^\sharp}.
\]
By the naturality of ${}^g\theta$, we slide both evaluations to the very end. To cycle $N$ to the front \emph{before} evaluating $M \odot {}^*M$, we apply ${}^g\theta_{{}^*N \odot L \odot M \odot {}^*M \odot P, N}$. To cycle $P$ to the front \emph{before} evaluating $N \odot {}^*N$, we apply ${}^g\theta_{N \odot {}^*N \odot L \odot M \odot {}^*M, P}$. This gives:
\[
P_1 = {}^g\sh{\id_P \odot \bar{\epsilon}_N \odot \id_L \odot \bar{\epsilon}_M} \circ {}^g\theta_{N \odot {}^*N \odot L \odot M \odot {}^*M, P} \circ {}^g\theta_{{}^*N \odot L \odot M \odot {}^*M \odot P, N} \circ {}^g\sh{h^\sharp \odot f^\sharp}.
\]
Applying the twisted hexagon axiom (Proposition~\ref{prop:g_twisted_axioms}) to the two adjacent cyclic permutations, they combine into a single shift of the block $P \odot N$:
\begin{equation}\label{eq:path1_abstract}
P_1 = {}^g\sh{\id_P \odot \bar{\epsilon}_N \odot \id_L \odot \bar{\epsilon}_M} \circ {}^g\theta_{{}^*N \odot L \odot M \odot {}^*M, P \odot N} \circ {}^g\sh{h^\sharp \odot f^\sharp}.
\end{equation}

Let $P_2$ denote the composite $\operatorname{tr}_{\mathcal{B},M}^g(H) \circ {}^g\theta_{K,Q}$. Expanding the trace of $H = (\id_P \odot h) \circ (f \odot \id_K)$ and identifying the mates gives:
\[
P_2 = {}^g\sh{\bar{\epsilon}_M \odot \id_{P \odot L}} \circ {}^g\theta_{{}^*M \odot P \odot L, M} \circ {}^g\sh{\id_{{}^*M \odot P} \odot \bar{\epsilon}_N \odot \id_{L \odot M}} \circ {}^g\sh{f^\sharp \odot h^\sharp} \circ {}^g\theta_{K,Q}.
\]
Pushing $f^\sharp \odot h^\sharp$ rightward through ${}^g\theta_{K,Q}$ via naturality swaps their order and updates the cyclic operator to ${}^g\theta_{{}^*N \odot L \odot M, {}^*M \odot P \odot N}$. Sliding the evaluations to the end requires cycling $M$ to the front \emph{before} evaluating $N \odot {}^*N$, yielding:
\[
P_2 = {}^g\sh{\bar{\epsilon}_M \odot \id_P \odot \bar{\epsilon}_N \odot \id_L} \circ {}^g\theta_{{}^*M \odot P \odot N \odot {}^*N \odot L, M} \circ {}^g\theta_{{}^*N \odot L \odot M, {}^*M \odot P \odot N} \circ {}^g\sh{h^\sharp \odot f^\sharp}.
\]
Applying the Twisted Hexagon axiom to these two permutations combines them into a single shift of the block $M \odot {}^*M \odot P \odot N$:
\begin{equation}\label{eq:path2_abstract}
P_2 = {}^g\sh{\bar{\epsilon}_M \odot \id_P \odot \bar{\epsilon}_N \odot \id_L} \circ {}^g\theta_{{}^*N \odot L, M \odot {}^*M \odot P \odot N} \circ {}^g\sh{h^\sharp \odot f^\sharp}.
\end{equation}

We show that \eqref{eq:path1_abstract} and \eqref{eq:path2_abstract} are equivalent. By the naturality of ${}^g\theta$ with respect to the evaluation map $\bar{\epsilon}_M$, sliding the evaluation across the cyclic operator yields the same expression for both paths:
\begin{align*}
P_1 &= {}^g\sh{\id_P \odot \bar{\epsilon}_N \odot \id_L} \circ {}^g\theta_{{}^*N \odot L, P \odot N} \circ {}^g\sh{\id_{{}^*N \odot L} \odot \bar{\epsilon}_M \odot \id_{P \odot N}} \circ {}^g\sh{h^\sharp \odot f^\sharp} \\
P_2 &= {}^g\sh{\id_P \odot \bar{\epsilon}_N \odot \id_L} \circ {}^g\theta_{{}^*N \odot L, P \odot N} \circ {}^g\sh{\id_{{}^*N \odot L} \odot \bar{\epsilon}_M \odot \id_{P \odot N}} \circ {}^g\sh{h^\sharp \odot f^\sharp}
\end{align*}
Therefore, $P_1 = P_2$, and the square commutes.
\end{proof}

\begin{cor}[Twisted cyclicity]\label{cor:twisted_cyclicity}
Let \(M,N\in\mathcal B^{G\text{-tw}}(R,S)\) be left dualizable
\(1\)-cells.  For \(2\)-cells
\(f\colon M\Rightarrow N\) and \(h\colon N\Rightarrow M\), one has
\(
\tr^g(f\circ h)=\tr^g(h\circ f).
\)
\end{cor}

\begin{proof}
Apply Proposition~\ref{prop:twisted_sliding_general} with
\(T=S\), \(U=R\), \(Q=K=U_S\), and \(P=L=U_R\), and suppress the
unit constraints.  The two vertical twisted cyclicity maps are then the
canonical unit isomorphisms.
\end{proof}

\begin{prop}[Twisted additivity]\label{prop:tw-additivity}
Assume that each hom-category \(\mathcal B^{G\text{-tw}}(A,B)\) is
additive, that horizontal composition distributes over biproducts, and that
the target \(\mathcal T\) of \({}^g\!\sh{-}\) is additive and the twisted
shadow preserves biproducts.  If
\(X,Y\in\mathcal B^{G\text{-tw}}(A,B)\) are left dualizable and
\(f_X\colon X\Rightarrow X\), \(f_Y\colon Y\Rightarrow Y\), then
\(X\oplus Y\) is left dualizable and
\(
\tr^g(f_X\oplus f_Y)=\tr^g(f_X)+\tr^g(f_Y).
\)
\end{prop}

\begin{proof}
The left dual of \(X\oplus Y\) is \({}^*X\oplus{}^*Y\).  If
\(i_X,p_X\) and \(i_Y,p_Y\) are the biproduct injections and projections,
and \(j_X,q_X\) and \(j_Y,q_Y\) denote the corresponding maps for the
dual, then
\[
\bar\eta_{X\oplus Y}
=(j_X\odot i_X)\circ\bar\eta_X
 +(j_Y\odot i_Y)\circ\bar\eta_Y,
\qquad
\bar\epsilon_{X\oplus Y}
=\bar\epsilon_X\circ(p_X\odot q_X)
 +\bar\epsilon_Y\circ(p_Y\odot q_Y).
\]
Substitute these expressions and
\(f_X\oplus f_Y=i_Xf_Xp_X+i_Yf_Yp_Y\) into the definition of the
twisted trace.  By naturality of \({}^g\theta\), every mixed term contains
an orthogonal composite such as \(p_Xi_Y\) or \(q_Xj_Y\), and hence
vanishes.  The two diagonal terms reduce, using
\(p_Xi_X=q_Xj_X=\id\) and the analogous identities for \(Y\), to
\(\tr^g(f_X)\) and \(\tr^g(f_Y)\), respectively.
\end{proof}

\begin{prop}[Twisted multiplicativity]\label{prop:tw-multiplicativity}
Let
\(X\in\mathcal B^{G\text{-tw}}(A,B)\) and
\(Y\in\mathcal B^{G\text{-tw}}(B,C)\) be left dualizable, and let
\(f_X\colon X\Rightarrow X\) and \(f_Y\colon Y\Rightarrow Y\).  Then
\(X\odot Y\) is left dualizable and
\[
\tr^g(f_X\odot f_Y)
=\tr^g(f_X)\circ\tr^g(f_Y)
\colon{}^g\!\sh{U_C}\longrightarrow{}^g\!\sh{U_A}.
\]
\end{prop}

\begin{proof}
Choose left duals \({}^*X\) and \({}^*Y\).  Then \(X\odot Y\)
has left dual \({}^*Y\odot{}^*X\), with
\[
\bar\eta_{X\odot Y}
=
(\id_{{}^*Y}\odot\bar\eta_X\odot\id_Y)
\circ\bar\eta_Y
\quad \text{and} \quad
\bar\epsilon_{X\odot Y}
=
\bar\epsilon_X\circ
(\id_X\odot\bar\epsilon_Y\odot\id_{{}^*X}).
\]
Suppressing associativity and unit constraints, put
\(
a_X
:=
(\id_{{}^*X}\odot f_X)\circ\bar\eta_X
\colon
U_B\Rightarrow{}^*X\odot X.
\)
After substituting the above duality maps into the definition of the
twisted trace, we obtain
\begin{align*}
\tr^g(f_X\odot f_Y)
={}&
{}^g\!\sh{\bar\epsilon_X}
\circ
{}^g\!\sh{\id_X\odot\bar\epsilon_Y\odot\id_{{}^*X}}
\circ
{}^g\theta_{{}^*Y\odot{}^*X,\,X\odot Y}
\circ
{}^g\!\sh{\id_{{}^*Y}\odot a_X\odot f_Y}
\circ
{}^g\!\sh{\bar\eta_Y}.
\end{align*}

The twisted hexagon axiom, applied to
\(
M={}^*Y\odot{}^*X, N=X, P=Y,
\)
gives
\[
{}^g\theta_{{}^*Y\odot{}^*X,\,X\odot Y}
=
{}^g\theta_{Y\odot{}^*Y\odot{}^*X,\,X}
\circ
{}^g\theta_{{}^*Y\odot{}^*X\odot X,\,Y}.
\]
Naturality of \({}^g\theta\) gives
\begin{align*}
&{}^g\theta_{{}^*Y\odot{}^*X\odot X,\,Y}
\circ
{}^g\!\sh{\id_{{}^*Y}\odot a_X\odot f_Y}
  =
{}^g\!\sh{\id_{Y\odot{}^*Y}\odot a_X}
\circ
{}^g\theta_{{}^*Y,Y}
\circ
{}^g\!\sh{\id_{{}^*Y}\odot f_Y},
\end{align*}
while naturality with respect to \(\bar\epsilon_Y\) gives
\[
{}^g\!\sh{\id_X\odot\bar\epsilon_Y\odot\id_{{}^*X}}
\circ
{}^g\theta_{Y\odot{}^*Y\odot{}^*X,\,X}
=
{}^g\theta_{{}^*X,X}
\circ
{}^g\!\sh{\bar\epsilon_Y\odot
\id_{{}^*X\odot X}}.
\]
Substituting these two identities and using functoriality of the
twisted shadow, we obtain
\begin{align*}
\tr^g(f_X\odot f_Y)
={}&
{}^g\!\sh{\bar\epsilon_X}
\circ
{}^g\theta_{{}^*X,X}
\circ
{}^g\!\sh{a_X}
\circ
{}^g\!\sh{\bar\epsilon_Y}
\circ
{}^g\theta_{{}^*Y,Y}
\circ
{}^g\!\sh{\id_{{}^*Y}\odot f_Y}
\circ
{}^g\!\sh{\bar\eta_Y}.
\end{align*}
Finally,
\(
{}^g\!\sh{a_X}
=
{}^g\!\sh{\id_{{}^*X}\odot f_X}
\circ
{}^g\!\sh{\bar\eta_X}.
\)
Therefore the preceding composite is exactly
\(
\tr^g(f_X)\circ\tr^g(f_Y).
\)
\end{proof}
\appendix

\section{Dualizable Bimodules over Green Functors}
\label{sec:appendix}

In this appendix we identify the left dualizable \(1\)-cells in the
bicategory of Green functors and bimodules.  Throughout, \(G\) is a finite
group, \(\mathbf{Mack}_G\) is equipped with the box product, and
\(\underline R\) and \(\underline S\) are not necessarily commutative
Green functors.

\begin{prop}\label{prop:bicategorical-tensor-Hom}\cite{lewis1981theory}
Let \(\underline M\) be an
\((\underline R,\underline S)\)-bimodule.  For
\(\underline L\in{}_{\underline S}\mathbf{Mod}\) and
\(\underline P\in{}_{\underline R}\mathbf{Mod}\), there is a natural
isomorphism
\begin{equation}\label{eq:boxS-HomR-adjunction}
\underline{\Hom}_{\underline R}
 \bigl(\underline M\square_{\underline S}\underline L,\underline P\bigr)
\cong
\underline{\Hom}_{\underline S}
 \bigl(\underline L,
       \underline{\Hom}_{\underline R}(\underline M,\underline P)\bigr).
\end{equation}
Thus
\(
\underline M\square_{\underline S}(-)
\;:\;
{}_{\underline S}\mathbf{Mod}
\rightleftarrows
{}_{\underline R}\mathbf{Mod}
\;:\;
\underline{\Hom}_{\underline R}(\underline M,-)
\)
is an adjoint pair.
\end{prop}

The category \({}_{\underline R}\mathbf{Mod}\) is abelian: kernels and
cokernels are created by the forgetful functor to
\(\mathbf{Mack}_G\).  Filtered colimits are likewise computed on the
underlying Mackey functors.

For a finite $G$-set $X$ and a left $\underline R$-module $\underline N$, let $\underline R_X:=\underline R\square\underline A_X$ and $\underline N_X:=\underline N\square\underline A_X.$
Since \(\underline A_X\) is canonically self-dual in
\(\mathbf{Mack}_G\), Proposition~\ref{prop:bicategorical-tensor-Hom}
gives natural isomorphisms
\begin{equation}\label{eq:Hom-RX-N}
\underline{\Hom}_{\underline R}(\underline R_X,\underline N)
\cong
\underline{\Hom}(\underline A_X,\underline N)
\cong
\underline N_X.
\end{equation}
Moreover, since $\underline R_X\square_{\underline R}\underline M \cong \underline M \square \underline A_X,$ we have
\begin{equation}\label{eq:Mod-RX-N}
{}_{\underline R}\mathbf{Mod}(\underline R_X,\underline N)
\cong
\underline N(X).
\end{equation}
Thus \(\underline R_X\) represents evaluation at \(X\) on
\({}_{\underline R}\mathbf{Mod}\). 
This representability observation leads to the following elementary criterion for finitely generated projective modules.

\begin{lemma}\label{lem:fgp-Hom-criterion}
A left \(\underline R\)-module \(\underline M\) is finitely generated
projective if and only if the ordinary functor
\(
{}_{\underline R}\mathbf{Mod}(\underline M,-)
\colon
{}_{\underline R}\mathbf{Mod}\longrightarrow\mathbf{Ab}
\)
is exact and preserves filtered colimits.
\end{lemma}

\begin{proof}
Suppose first that \(\underline M\) is finitely generated projective.
Then \(\underline M\) is a retract of \(\underline R_X\) for some
finite \(G\)-set \(X\).  By \eqref{eq:Mod-RX-N},
\(
{}_{\underline R}\mathbf{Mod}(\underline R_X,\underline N)
\cong
\underline N(X).
\)
Since exact sequences and filtered colimits of
\(\underline R\)-modules are computed on the underlying Mackey
functors, evaluation at \(X\) is exact and preserves filtered
colimits.  The same is therefore true for
\({}_{\underline R}\mathbf{Mod}(\underline M,-)\), since it is a
retract of
\({}_{\underline R}\mathbf{Mod}(\underline R_X,-)\).

Conversely, suppose that
\({}_{\underline R}\mathbf{Mod}(\underline M,-)\) is exact and
preserves filtered colimits.  Exactness implies that
\(\underline M\) is projective.  Write \(\underline M\) as the
filtered colimit of its finitely generated submodules:
\(
\underline M=\varinjlim_i\underline M_i.
\)
Preservation of this colimit gives
\(
{}_{\underline R}\mathbf{Mod}(\underline M,\underline M)
\cong
\varinjlim_i
{}_{\underline R}\mathbf{Mod}(\underline M,\underline M_i).
\)
Hence \(\id_{\underline M}\) factors through some inclusion
\(\underline M_i\hookrightarrow\underline M\).  This inclusion is
therefore both a monomorphism and a split epimorphism, and hence an
isomorphism.  Thus \(\underline M\) is finitely generated.  Being
both finitely generated and projective, it is a direct summand of a
finite free \(\underline R\)-module.
\end{proof}

The next lemma provides the tensor--Hom isomorphism required to construct the duality adjunction in our main theorem.

\begin{lemma}\label{lem:hom-tensor-green}
Let \(\underline M\) be an
\((\underline R,\underline S)\)-bimodule that is finitely generated
projective as a left \(\underline R\)-module, and set
\(
{}^*\underline M
:=
\underline{\Hom}_{\underline R}(\underline M,\underline R),
\)
with its induced \((\underline S,\underline R)\)-bimodule structure.
Then, for every left \(\underline R\)-module \(\underline N\), the
canonical  map
\(
\tau_{\underline N}\colon
{}^*\underline M\square_{\underline R}\underline N
\longrightarrow
\underline{\Hom}_{\underline R}(\underline M,\underline N)
\)
is an isomorphism of left \(\underline S\)-modules.
\end{lemma}

\begin{proof}
First suppose that \(\underline M=\underline R_X\) is finite free.
Using the self-duality of \(\underline A_X\), both sides identify
naturally with \(\underline N_X\):
\(
\underline{\Hom}_{\underline R}(\underline R_X,\underline R)
 \square_{\underline R}\underline N
\cong
\underline R_X\square_{\underline R}\underline N
\cong
\underline N_X
\cong
\underline{\Hom}_{\underline R}(\underline R_X,\underline N).
\)
Under these identifications, \(\tau_{\underline N}\) is the identity.

For general \(\underline M\), choose \(\underline R\)-linear maps
\(
\underline M\xrightarrow{i}\underline R_X
\xrightarrow{p}\underline M,
p \circ i=\id_{\underline M}.
\)
The tensor--Hom comparison is natural in its first variable.  Hence, after
forgetting the \(\underline S\)-action, \(\tau_{\underline N}\) is a
retract of the comparison map for \(\underline R_X\), and is therefore an
isomorphism in \(\mathbf{Mack}_G\).  Since \(\tau_{\underline N}\) is
\(\underline S\)-linear, it is an isomorphism of
\(\underline S\)-modules.
\end{proof}

\begin{prop}\label{prop:dualizable-bimodule}
Let \(\underline R\) and \(\underline S\) be Green functors.  In the
bicategory whose objects are Green functors, whose \(1\)-cells are
bimodules, and whose composition is the relative box product, an
\((\underline R,\underline S)\)-bimodule
\(\underline M\)
is left dualizable if and only if it is finitely generated projective as a
left \(\underline R\)-module.  In this case, a left dual is
\(
{}^*\underline M
=
\underline{\Hom}_{\underline R}(\underline M,\underline R),
\)
viewed as an \((\underline S,\underline R)\)-bimodule.
\end{prop}

\begin{proof}
Suppose first that \(\underline M\) is left dualizable, with left dual
\({}^*\underline M\).  The duality adjunction and
Proposition~\ref{prop:bicategorical-tensor-Hom} identify the two right
adjoints of
\(\underline M\square_{\underline S}(-)\), giving a natural isomorphism
\(
{}^*\underline M\square_{\underline R}(-)
\cong
\underline{\Hom}_{\underline R}(\underline M,-).
\)
The functor on the left preserves all colimits.  The functor on the right
is left exact, while the functor on the left is right exact; hence their
common value is exact and preserves filtered colimits.  Evaluating the
internal Hom at \(G/G\), we conclude that
\({}_{\underline R}\mathbf{Mod}(\underline M,-)\) is exact and preserves
filtered colimits.  By Lemma~\ref{lem:fgp-Hom-criterion},
$\underline M$ is finitely generated projective as a left
$\underline R$-module.

Conversely, suppose that \(\underline M\) is finitely generated projective
as a left \(\underline R\)-module and set
\(
{}^*\underline M
=
\underline{\Hom}_{\underline R}(\underline M,\underline R).
\)
By Lemma~\ref{lem:hom-tensor-green}, the tensor--Hom comparison gives a
natural isomorphism
\(
{}^*\underline M\square_{\underline R}(-)
\cong
\underline{\Hom}_{\underline R}(\underline M,-).
\)
Consequently, the tensor--Hom adjunction identifies
\(\underline M\square_{\underline S}(-)\) as a left adjoint of
\({}^*\underline M\square_{\underline R}(-)\).  Its unit and counit are
maps
\(
\bar\eta\colon
\underline S\to 
{}^*\underline M\square_{\underline R}\underline M,
\bar\epsilon\colon
\underline M\square_{\underline S}{}^*\underline M
\to
\underline R,
\)
and the adjunction triangle identities are precisely the duality triangle
identities.  Thus \(\underline M\) is left dualizable with left dual
\({}^*\underline M\).
\end{proof}

\bibliography{trace}{}

@article {PONSH,
    AUTHOR = {Ponto, Kate and Shulman, Michael},
     TITLE = {Shadows and traces in bicategories},
   JOURNAL = {J. Homotopy Relat. Struct.},
  FJOURNAL = {Journal of Homotopy and Related Structures},
    VOLUME = {8},
      YEAR = {2013},
    NUMBER = {2},
     PAGES = {151--200},
      ISSN = {2193-8407,1512-2891},
   MRCLASS = {18D05 (55M20)},
  MRNUMBER = {3095324},
MRREVIEWER = {Richard\ John\ Steiner},
       DOI = {10.1007/s40062-012-0017-0},
       URL = {https://doi.org/10.1007/s40062-012-0017-0},
}

@article {Tambara,
    AUTHOR = {Tambara, Daisuke},
     TITLE = {On multiplicative transfer},
   JOURNAL = {Comm. Algebra},
  FJOURNAL = {Communications in Algebra},
    VOLUME = {21},
      YEAR = {1993},
    NUMBER = {4},
     PAGES = {1393--1420},
      ISSN = {0092-7872},
   MRCLASS = {19A22 (20J06 55N91)},
  MRNUMBER = {1209937},
       DOI = {10.1080/00927879308824627},
}

@article{thevenaz1995structure,
  title={The structure of Mackey functors},
  author={Th{\'e}venaz, Jacques and Webb, Peter},
  journal={Transactions of the American Mathematical Society},
  volume={347},
  number={6},
  pages={1865--1961},
  year={1995}
}

@article{BPW,
  title={Quantum link homology via trace functor I},
  author={Beliakova, Anna and Putyra, Krzysztof K and Wehrli, Stephan M},
  journal={Inventiones mathematicae},
  volume={215},
  number={2},
  pages={383--492},
  year={2019},
  publisher={Springer}
}

@article {HA,
    AUTHOR = {Hattori, Akira},
     TITLE = {Rank element of a projective module},
   JOURNAL = {Nagoya Math. J.},
  FJOURNAL = {Nagoya Mathematical Journal},
    VOLUME = {25},
      YEAR = {1965},
     PAGES = {113--120},
      ISSN = {0027-7630,2152-6842},
   MRCLASS = {16.40 (18.00)},
  MRNUMBER = {175950},
MRREVIEWER = {I.\ Reiner},
       URL = {http://projecteuclid.org/euclid.nmj/1118801428},
}

@article{NikolausScholze,
  author    = {Thomas Nikolaus and Peter Scholze},
  title     = {On topological cyclic homology},
  journal   = {Annals of Mathematics},
  volume    = {187},
  number    = {2},
  year      = {2018},
  pages     = {361--502},
  doi       = {10.4007/annals.2018.187.2.4}
}

@phdthesis{Hemo,
  author       = {Hemo, Tamir},
  title        = {On the Categorical Approach to the {F}robenius Trace},
  school       = {California Institute of Technology},
  year         = {2023},
  month        = {May},
  note         = {},
  url          = {https://resolver.caltech.edu/CaltechTHESIS:05292023-123456789},
}

@article {PON1,
    AUTHOR = {Ponto, Kate},
     TITLE = {Fixed point theory and trace for bicategories},
   JOURNAL = {Ast\'erisque},
  FJOURNAL = {Ast\'erisque},
    NUMBER = {333},
      YEAR = {2010},
     PAGES = {xii+102},
      ISSN = {0303-1179,2492-5926},
      ISBN = {978-2-85629-293-8},
   MRCLASS = {55M20 (16D90 18D05)},
  MRNUMBER = {2741967},
MRREVIEWER = {R.\ H.\ Street},
}

@article {LEWMAN,
    AUTHOR = {Lewis, Jr., L. Gaunce and Mandell, Michael A.},
     TITLE = {Equivariant universal coefficient and {K}\"unneth spectral
              sequences},
   JOURNAL = {Proc. London Math. Soc. (3)},
  FJOURNAL = {Proceedings of the London Mathematical Society. Third Series},
    VOLUME = {92},
      YEAR = {2006},
    NUMBER = {2},
     PAGES = {505--544},
      ISSN = {0024-6115,1460-244X},
   MRCLASS = {55N91 (55P43 55U20 55U25)},
  MRNUMBER = {2205726},
MRREVIEWER = {R.\ E.\ Stong},
       DOI = {10.1112/S0024611505015492},
       URL = {https://doi.org/10.1112/S0024611505015492},
}

@article {AGHKK,
    AUTHOR = {Adamyk, Katharine and Gerhardt, Teena and Hess, Kathryn and
              Klang, Inbar and Kong, Hana Jia},
     TITLE = {A shadow perspective on equivariant {H}ochschild homologies},
   JOURNAL = {Int. Math. Res. Not. IMRN},
  FJOURNAL = {International Mathematics Research Notices. IMRN},
      YEAR = {2023},
    NUMBER = {18},
     PAGES = {15299--15357},
      ISSN = {1073-7928,1687-0247},
   MRCLASS = {18N40 (18F25 55P43)},
  MRNUMBER = {4644963},
MRREVIEWER = {Steffen\ Sagave},
       DOI = {10.1093/imrn/rnac250},
       URL = {https://doi.org/10.1093/imrn/rnac250},
}

@book {weibel:homological,
    AUTHOR = {Weibel, Charles A.},
     TITLE = {An introduction to homological algebra},
    SERIES = {Cambridge Studies in Advanced Mathematics},
    VOLUME = {38},
 PUBLISHER = {Cambridge University Press, Cambridge},
      YEAR = {1994},
     PAGES = {xiv+450},
      ISBN = {0-521-43500-5; 0-521-55987-1},
   MRCLASS = {18-01 (16-01 17-01 20-01 55Uxx)},
  MRNUMBER = {1269324},
MRREVIEWER = {Kenneth A. Brown},
       DOI = {10.1017/CBO9781139644136},
       URL = {https://doi-org.ezproxy.uky.edu/10.1017/CBO9781139644136},
}

@article{MMSS,
  author    = {Michael A. Mandell and J. Peter May and Stefan Schwede and Brooke Shipley},
  title     = {Model Categories of Diagram Spectra},
  journal   = {Proc. London Math. Soc. (3)},
  volume    = {82},
  number    = {2},
  year      = {2001},
  pages     = {441--512},
  doi       = {10.1112/S0024611501012692}
}

@article {BGHL,
    AUTHOR = {Blumberg, Andrew J. and Gerhardt, Teena and Hill, Michael A.
              and Lawson, Tyler},
     TITLE = {The {W}itt vectors for {G}reen functors},
   JOURNAL = {J. Algebra},
  FJOURNAL = {Journal of Algebra},
    VOLUME = {537},
      YEAR = {2019},
     PAGES = {197--244},
      ISSN = {0021-8693,1090-266X},
   MRCLASS = {55P42 (19D55 55N91)},
  MRNUMBER = {3990042},
MRREVIEWER = {Andrew\ J.\ Baker},
       DOI = {10.1016/j.jalgebra.2019.07.014},
       URL = {https://doi.org/10.1016/j.jalgebra.2019.07.014},
}

@article{CCM,
  title={A linearization map for genuine equivariant algebraic {K}-theory},
  author={Calle, Maxine  and Chan, David and Mejia, Andres },
  journal={Preprint arXiv:2309.08025},
  year={2023},
}

@article{boucproj,
  title={The {B}urnside dimension of projective Mackey functors (Algebraic Combinatorics)},
  author={Bouc, Serge},
  journal={RIMS K{\^o}ky{\^u}roku},
  volume={1440},
  pages={107--120},
  year={2005},
  publisher={Kyoto University, Research Institute for Mathematical Sciences}
}

@article {HHR,
    AUTHOR = {Hill, M. A. and Hopkins, M. J. and Ravenel, D. C.},
     TITLE = {On the nonexistence of elements of {K}ervaire invariant one},
   JOURNAL = {Ann. of Math. (2)},
  FJOURNAL = {Annals of Mathematics. Second Series},
    VOLUME = {184},
      YEAR = {2016},
    NUMBER = {1},
     PAGES = {1--262},
      ISSN = {0003-486X,1939-8980},
   MRCLASS = {55P91 (55N22 55P42 55Q45 55T15 55U35 57R15)},
  MRNUMBER = {3505179},
MRREVIEWER = {Paul\ G.\ Goerss},
       DOI = {10.4007/annals.2016.184.1.1},
       URL = {https://doi.org/10.4007/annals.2016.184.1.1},
}

@article {CP,
    AUTHOR = {Campbell, Jonathan A. and Ponto, Kate},
     TITLE = {Topological {H}ochschild homology and higher characteristics},
   JOURNAL = {Algebr. Geom. Topol.},
  FJOURNAL = {Algebraic \& Geometric Topology},
    VOLUME = {19},
      YEAR = {2019},
    NUMBER = {2},
     PAGES = {965--1017},
      ISSN = {1472-2747,1472-2739},
   MRCLASS = {16D90 (18D05 19D55 55M20 55R12)},
  MRNUMBER = {3924181},
MRREVIEWER = {Andrew\ J.\ Baker},
       DOI = {10.2140/agt.2019.19.965},
       URL = {https://doi.org/10.2140/agt.2019.19.965},
}

@article {PSSY,
    AUTHOR = {Ponto, Kate and Shulman, Michael},
     TITLE = {Traces in symmetric monoidal categories},
   JOURNAL = {Expo. Math.},
  FJOURNAL = {Expositiones Mathematicae},
    VOLUME = {32},
      YEAR = {2014},
    NUMBER = {3},
     PAGES = {248--273},
      ISSN = {0723-0869,1878-0792},
   MRCLASS = {18D10 (55M20)},
  MRNUMBER = {3253568},
MRREVIEWER = {Miguel\ Angel\ Garc\'ia-Mu\~noz},
       DOI = {10.1016/j.exmath.2013.12.003},
       URL = {https://doi.org/10.1016/j.exmath.2013.12.003},
}

@article {PSL,
    AUTHOR = {Ponto, Kate and Shulman, Michael},
     TITLE = {The linearity of traces in monoidal categories and
              bicategories},
   JOURNAL = {Theory Appl. Categ.},
  FJOURNAL = {Theory and Applications of Categories},
    VOLUME = {31},
      YEAR = {2016},
     PAGES = {Paper No. 23, 594--689},
      ISSN = {1201-561X},
   MRCLASS = {18D05 (18D10 18D20)},
  MRNUMBER = {3518981},
MRREVIEWER = {Ang\'elica\ M.\ Osorno},
}

@article {MM,
    AUTHOR = {Malkiewich, Cary and Merling, Mona},
     TITLE = {Equivariant {$A$}-theory},
   JOURNAL = {Doc. Math.},
  FJOURNAL = {Documenta Mathematica},
    VOLUME = {24},
      YEAR = {2019},
     PAGES = {815--855},
      ISSN = {1431-0635,1431-0643},
   MRCLASS = {19D10 (18D50 19C99 55P91)},
  MRNUMBER = {3982285},
MRREVIEWER = {No\'e\ B\'arcenas Torres},
}

@article {ABGHLM,
    AUTHOR = {Angeltveit, Vigleik and Blumberg, Andrew J. and Gerhardt,
              Teena and Hill, Michael A. and Lawson, Tyler and Mandell,
              Michael A.},
     TITLE = {Topological cyclic homology via the norm},
   JOURNAL = {Doc. Math.},
  FJOURNAL = {Documenta Mathematica},
    VOLUME = {23},
      YEAR = {2018},
     PAGES = {2101--2163},
      ISSN = {1431-0635,1431-0643},
   MRCLASS = {55P91 (16E40 19D55 55P43)},
  MRNUMBER = {3933034},
MRREVIEWER = {Bj\o rn\ Ian\ Dundas},
}

@book{may2004parametrized,
    AUTHOR = {May, J. P. and Sigurdsson, J.},
     TITLE = {Parametrized homotopy theory},
    SERIES = {Mathematical Surveys and Monographs},
    VOLUME = {132},
 PUBLISHER = {American Mathematical Society, Providence, RI},
      YEAR = {2006},
     PAGES = {x+441},
      ISBN = {978-0-8218-3922-5; 0-8218-3922-5},
   MRCLASS = {55P42 (19L99 55N20 55N22 55P91)},
  MRNUMBER = {2271789},
MRREVIEWER = {A.\ A.\ Ranicki},
       DOI = {https://doi.org/10.1090/surv/132},
       URL = {https://doi.org/10.1090/surv/132},
}

@unpublished{MAL,
  author  = {Cary Malkiewich},
  title   = {Parametrized Spectra, a Low-Tech Approach},
  note    = {arXiv:1906.04773 [math.AT]},
  year    = {2019},
  url     = {https://doi.org/10.48550/arXiv.1906.04773}
}

@article {BHM,
    AUTHOR = {B\"okstedt, M. and Hsiang, W. C. and Madsen, I.},
     TITLE = {The cyclotomic trace and algebraic {$K$}-theory of spaces},
   JOURNAL = {Invent. Math.},
  FJOURNAL = {Inventiones Mathematicae},
    VOLUME = {111},
      YEAR = {1993},
    NUMBER = {3},
     PAGES = {465--539},
      ISSN = {0020-9910,1432-1297},
   MRCLASS = {55P42 (19D55 19L99 55P60)},
  MRNUMBER = {1202133},
MRREVIEWER = {Roland\ Schw\"anzl},
       DOI = {10.1007/BF01231296},
       URL = {https://doi.org/10.1007/BF01231296},
}

@article{CLMPI,
  author    = {Jonathan A. Campbell and John A. Lind and Cary Malkiewich and Kate Ponto and Inna Zakharevich},
  title     = {The {K}-theory of endomorphisms, the trace, and zeta functions},
  journal   = {La Matematica},
  year      = {2025},
  volume    = {4},
  number    = {2},
  pages     = {214--292},
  doi       = {10.1007/s44007-025-00154-0},
  url       = {https://doi.org/10.1007/s44007-025-00154-0}
}

@incollection {DoldPuppe,
    AUTHOR = {Dold, A. and Puppe, D.},
     TITLE = {Duality, trace and transfer},
      NOTE = {Topology (Moscow, 1979)},
   JOURNAL = {Trudy Mat. Inst. Steklov.},
  FJOURNAL = {Akademiya Nauk SSSR. Trudy Matematicheskogo Instituta imeni V.
              A. Steklova},
    VOLUME = {154},
      YEAR = {1983},
     PAGES = {81--97},
      ISSN = {0371-9685},
   MRCLASS = {55P25},
  MRNUMBER = {733829},
}

@article{lewis1981theory,
  title={The theory of {Green} functors},
  author={Lewis Jr, L Gaunce},
  journal={Unpublished manuscript},
  year={1981}
}

@book{bouc1997,
  author = {Bouc, Serge},
  title = {Green Functors and G-sets},
  volume = {1671},
  series = {Lecture Notes in Mathematics},
  year = {1997},
  publisher = {Springer},
  address = {Berlin, Heidelberg}
}

@article{Merling2017,
  author  = {Merling, Mona},
  title   = {Equivariant algebraic {$K$}-theory of {$G$}-rings},
  journal = {Mathematische Zeitschrift},
  volume  = {285},
  number  = {3-4},
  pages   = {1205--1248},
  year    = {2017},
  doi     = {10.1007/s00209-016-1745-3}
}

@article{dress1971notes,
  title={Notes on the theory of representations of finite groups},
  author={Dress, Andreas},
  journal={(No Title)},
  year={1971}
}

@article{green1971,
  title={Axiomatic representation theory for finite groups},
  author={Green, James Alexander},
  journal={Journal of pure and applied algebra},
  volume={1},
  number={1},
  pages={41--77},
  year={1971},
  publisher={Elsevier}
}

@article{ventura2005homological,
  title={Homological algebra for the representation Green functor for abelian groups},
  author={Ventura, Joana},
  journal={Transactions of the American Mathematical Society},
  volume={357},
  number={6},
  pages={2253--2289},
  year={2005}
}

@book{EKMM97,
  title={Rings, modules, and algebras in stable homotopy theory},
  author={Elmendorf, Anthony D. and Kriz, Igor and Mandell, Michael A. and May, J. Peter},
  volume={47},
  year={1997},
  publisher={American Mathematical Society},
  series={Mathematical Surveys and Monographs},
  address={Providence, RI},
  isbn={978-0-8218-0638-8}
}

@article{CGK25,
  title={Trace methods for equivariant algebraic K-theory},
  author={Chan, David and Gerhardt, Teena and Klang, Inbar},
  journal={arXiv preprint arXiv:2505.11327},
  year={2025}
}

@book{loday2013cyclic,
  title={Cyclic homology},
  author={Loday, Jean-Louis},
  volume={301},
  year={2013},
  publisher={Springer Science \& Business Media}
}

@article{maclane1985coherence,
  title={Coherence for bicategories and indexed categories},
  author={MacLane, Saunders and Par{\'e}, Robert},
  journal={Journal of Pure and Applied Algebra},
  volume={37},
  pages={59--80},
  year={1985},
  publisher={North-Holland}
}
\bibliographystyle{amsalpha2}

\end{document}